\documentclass[11pt]{amsart}
\usepackage[utf8]{inputenc}
\usepackage[T2A,T1]{fontenc}
\usepackage[russian,english]{babel}

\DeclareRobustCommand{\ru}[1]{\foreignlanguage{russian}{#1}}
\usepackage{amsmath,amssymb,amsfonts, amsthm}
\usepackage{mathtools}
\usepackage{mathrsfs}
\usepackage{dsfont}
\usepackage{enumitem}
\usepackage{multicol}
\usepackage{graphicx}
\usepackage{rotating}
\usepackage{lscape}
\usepackage[table]{xcolor}
\definecolor{lightgray}{gray}{0.95}
\usepackage{geometry}
\theoremstyle{plain}
\newtheorem{theorem}{Theorem}[section]

\newtheorem{proposition}[theorem]{Proposition}
\newtheorem{corollary}[theorem]{Corollary}
\newtheorem{definition}[theorem]{Definition}
\newtheorem{example}[theorem]{Example}
\newtheorem{remark}[theorem]{Remark}

\setlist{itemsep=2pt, topsep=7pt, parsep=2pt}
\DeclareMathOperator{\Arg}{Arg}
\usepackage[hidelinks,hypertexnames=false]{hyperref}
\usepackage{tikz}
\usetikzlibrary{calc,patterns, fadings,decorations.markings}
\tikzfading[name=stripEnds,
  top color=transparent!97,
  bottom color=transparent!97,
  middle color=transparent!10
]
\newtheorem{notation}[theorem]{Notation}

\usepackage{float}
\usepackage{algorithm}
\usepackage{algpseudocode}
\usepackage[most]{tcolorbox}
\numberwithin{equation}{section}

\definecolor{algogray}{gray}{0.96}

\tcbset{
  algobox/.style={
    colback=algogray,
    colframe=gray!50,
    boxrule=0.7pt,
    arc=2pt,
    left=7pt,
    right=7pt,
    top=7pt,
    bottom=7pt,
    enhanced
  }
}

\title[Classical Orthogonal Polynomials]{Half-Step-Invariant Sets, Admissible Maps, and Classical Orthogonal Polynomials}

\author{K. Castillo}
\address{CMUC, Department of Mathematics, University of Coimbra, 3000-143 Coimbra,
Portugal}
\email{kenier@mat.uc.pt}
\subjclass[2020]{Primary 33C45, 33D45; Secondary 39A13, 46A13, 47B39}

\keywords{Classical orthogonal polynomials, half-step-invariant sets, admissible maps, \(q\)-exponential maps, quadratic maps, alternating maps}

\hypersetup{
  pdftitle={Half-Step-Invariant Sets, Admissible Maps, and Classical Orthogonal Polynomials},
  pdfauthor={K. Castillo},
  pdfsubject={Classical orthogonal polynomials on half-step-invariant sets},
  pdfkeywords={classical orthogonal polynomials, half-step-invariant sets, admissible maps, q-exponential maps, quadratic maps, alternating maps}
}

\date{\today}

\begin{document}

\begin{abstract}
This paper revisits the notion of classical orthogonal polynomials from a broader functional-analytic point of view. It is intended neither as a survey of known results nor as a review of the literature, but rather as a conceptual reappraisal of the subject from a perspective in which certain persistent distortions become plainly visible. The theory is developed on subsets of the complex plane that are stable under half-step translations, and both classicality and orthogonality are understood in the continuous dual of a suitable locally convex space of polynomials. The repeated reappearance of ostensibly new families of classical orthogonal polynomials arising from exotic maps, algebraically equivalent families artificially separated, unnecessary parameter restrictions inherited from positive-definite models, apparently distinct phenomena associated with root-of-unity values of \(q\) in the \(q\)-exponential case, naive \(q\to -1\) limits in that same setting, finite truncations of otherwise infinite orthogonal polynomial sequences, and geometric recastings of the local half-step relation in terms of plane conic curves all make the need for a broader structural framework increasingly clear. The present paper seeks to articulate such a framework in a way that allows the reader to distinguish the genuinely new from the merely artificial, without resorting to an exhaustive case-by-case examination of prior work. In the ordinary quadratic and non-torsion \(q\)-exponential regimes, the infinite regularity criteria retain their full force as if-and-only-if statements; finite resonances, by contrast, are governed by the exact moment system and the non-vanishing of its Hankel determinants.

\end{abstract}

\maketitle
\setcounter{tocdepth}{2}
\tableofcontents
\bigskip

\section{Introduction}
Hoffman records in \emph{The Man Who Loved Only Numbers} \cite[p.~49]{H} an anecdote about Erd\H{o}s, recounted by Purdy, who had Erd\H{o}s number \(1\). Whether or not one wishes to regard the story as literally true in every detail, it deserves to be taken seriously in spirit: during a visit to Texas A\&M in the mid-1970s, Erd\H{o}s is said to have encountered, over coffee in the mathematics lounge, a problem in functional analysis seemingly far removed from his natural territory. Two analysts had already produced a lengthy thirty-page argument, of which they were understandably rather proud. After asking for the meaning of a few symbols, Erd\H{o}s reportedly returned with a proof in two lines. Apocryphal or not in its exact literary form, the story points to something mathematically substantial: genuine understanding lies not in passive familiarity with a formalism, but in the ability to see through it to the structure that governs it. That distinction between formal appearance and structural reality is precisely what is at issue in the present paper, in the setting of classical orthogonal polynomials, whatever meaning the reader may presently attach to that expression.

In \cite{CG26a}, we argued that the inherited conception of classical orthogonal polynomials reflected in the \emph{NIST Handbook of Mathematical Functions} (2010), in the version of the \emph{NIST Digital Library of Mathematical Functions} consulted at the time of writing, and in much of the surrounding literature, has been distorted by an interpretation that is increasingly difficult to sustain, especially since the features by which classicality is usually recognised are algebraic in nature. In that sense, \cite{CG26a} may be read as exhibiting a phenomenon that becomes progressively more dramatic once one moves beyond the ordinary differential setting treated in the introduction of that work, and even beyond the better-known discrete classical situations examined there in detail, into the more general framework described by Nikiforov and Uvarov in the early 1980s \cite{NU83}. As \cite{CG26a} makes clear, once one leaves behind the tacit identification of orthogonality with the existence of a representing positive measure, the algebraic structure of the problem comes properly into view. From that standpoint, many distinctions that had seemed canonical reveal themselves to be contingent artefacts of a narrow realisation: algebraically equivalent families are no longer artificially separated, parameter restrictions inherited from positive-definite models lose their conceptual privilege, and the continuous and discrete theories can be understood within a common dual-topological framework.

The issue, however, does not lie solely in the way orthogonality is interpreted. A recent example from the literature, discussed in detail in \cite{CG25} and closely aligned with the concerns of the present work, illustrates this point quite vividly. Let \(N\) be a fixed positive odd integer, and consider the map
\[
W:\mathbb N_N=\{0,1,\dots,N\}\longrightarrow\mathbb R,
\quad
W(s)=s+\frac12(\gamma-1)\bigl(1-(-1)^s\bigr),
\]
where \(0<\gamma<2\). In \cite{VZ12}, in the course of an analysis of \(XX\) spin chains with perfect state transfer, a physical setting whose details are immaterial here, the authors define a certain positive discrete weight \(\omega\) on the image of \(W\). The associated positive-definite linear functional \(\mathbf w\) on the space of complex polynomials may then be written in the form
\[
\mathbf w(p)=\sum_{s\in\mathbb N_N}p\bigl(W(s)\bigr)\,\omega\bigl(W(s)\bigr).
\]
The resulting finite orthogonal polynomial sequence was presented as new and introduced under the name para--Krawtchouk polynomials, the terminology being motivated by an apparent connection with the Krawtchouk polynomials. The suggestion, explicitly advanced by the authors of \cite{VZ12}, that these polynomials share features commonly regarded, within the Nikiforov--Uvarov setting, as characteristic of the classical orthogonal polynomial sequences might at first suggest the presence of some structural fissure in that framework \cite{NU83,NU84,NUS85,NUS86,NSU91,ARS95}, since the map \(W\) is not among the admissible forms allowed there. Yet this is not in fact the case, nor is the intuition underlying the claim made in \cite{VZ12} simply mistaken.

The apparent tension is therefore the following: how can these two facts, seemingly irreconcilable, both be true at once? To answer this, no extensive familiarity with the subject is required. Following the line of thought suggested by the Erd\H{o}s anecdote, the natural question is in fact a simple one: what forms may the map have if the sequence is genuinely to belong to the Nikiforov--Uvarov framework? Once one knows the answer, and in particular that affine maps arise among the admissible possibilities, the matter becomes almost immediate, since the image of \(W\) admits a reparametrisation by a map of a type already allowed in that framework. Indeed, define
\[
X:\mathbb S_N\longrightarrow\mathbb R,
\quad
X(s)=2s,
\]
where
\[
\mathbb S_N=
\left\{
0,\,
\frac\gamma2,\,
1,\,
\frac\gamma2+1,\,
\dots,\,
\frac{N-1}{2},\,
\frac\gamma2+\frac{N-1}{2}
\right\}.
\]
Then
\begin{align*}
W(\mathbb N_N)
=
\{0,\gamma,2,\gamma+2,\dots,N-1,\gamma+N-1\}=X(\mathbb S_N).
\end{align*}

With the map \(X\) in place, and orthogonality understood in a substantially broader framework, there is no longer any compelling reason to keep \(s\) tied to the integers, nor to preserve the original positivity-driven restriction \(0<\gamma<2\), apart from the non-degeneracy conditions required by the chosen functional representation. What initially appeared to be a new construction then begins to reveal its genuine structural meaning. Accordingly, \(\mathbf w\) can be written as
\begin{align}\label{eq:para-k-functional}
\mathbf w(p)=\sum_{s\in\mathbb S_N}p\bigl(X(s)\bigr)\,\omega\bigl(X(s)\bigr).
\end{align}
Therefore, even the slightest suspicion that the orthogonal polynomial sequence with respect to \(\mathbf w\) satisfies some property rendering it classical is enough to cast serious doubt on its actual novelty, whether or not one is as yet able to recognise it as such. In the positive finite framework of \cite{VZ12}, this identification is in fact rather simple, precisely because the relevant facts are well documented \cite{KLS10}: if the connection with the Krawtchouk polynomials is not especially transparent, the link with the Hahn polynomials is not difficult to discern. In a broader setting, we identified in \cite{CG26a} a much wider class, of which this one is an explicit instance, and showed, through Examples~3.6, 4.4, 5.1, and 8.6, step by step, how naturally and completely it fits into the Nikiforov--Uvarov framework, while at the same time illustrating how the corresponding functional representation may be recovered in a particularly simple manner.

It remains to understand what led the authors of \cite{VZ12} to formulate the matter in precisely those terms. The answer, at least in part, lies in the fact that, once one starts from a pre-specified set---in \cite{VZ12}, \(\mathbb{Z}\)---that set acquires an artificial autonomy. It thereby comes to appear as the primary object, whereas the parametrising map is treated as secondary, if not altogether incidental. From a structural point of view, however, the opposite is the case. The initial set is more naturally regarded as
\[
\mathbb S=\mathbb Z\cup\left(\frac\gamma2+\mathbb Z\right),
\]
that is, in the non-degenerate case, as the union of two distinct translates of the same arithmetic progression. The finite set \(\mathbb S_N\) is only the portion on which the non-zero weights sit; the ambient half-step-invariant set is \(\mathbb S\). Thus the finiteness belongs to the weight, not to the underlying domain. At this structural level, the essential feature is a proper description of the underlying set, one that allows a parametrising map compatible with the theory. This distinction is easy to underestimate, especially because, for many readers, the set \(\mathbb S\) appears to stand in tension with the traditional way of thinking about the matter. Yet the theory is not without internal latitude: it admits a genuine margin of structural freedom, but not to the point of compromising its overall coherence. In that sense, what is at stake in \cite{VZ12} is not a matter of calculation but one of perspective: the formulas are correct, the orthogonality is genuine, but the structural level at which the sequence ought to be interpreted has been misidentified. Analogous questions of bi-lattice structure and reparametrisation arise in \cite{LVZ16,LVZ18}.

This is also the appropriate place to clarify the status of the conics frequently invoked in the literature in this context. Such objects naturally give rise to an algebraic relation involving neighbouring values of a map, and this relation may be encoded by a curve in the plane. In the family presented in \cite{VZ12}, for instance, this phenomenon already appears as soon as the support is reparametrised by the affine map \(X\): the neighbouring half-step values become algebraically governed by \(X\), and one is thereby led, in the general formalism, to a conic relation of the kind underlying what may be called the Magnus conic; in the affine case this relation is degenerate \cite{Mag88,Mag95}. The point is not that such a conic fails to appear; on the contrary, it arises naturally. However, its role must be interpreted correctly. The conic is not the source of the structure, but one of its algebraic consequences. By itself, it neither determines the domain on which the parameter moves, nor reveals the decomposition of that domain into two translates of the same arithmetic progression, nor decides whether two apparently different supports arise from the same underlying mechanism after reparametrisation. In short, the essential point in the situation under discussion is that the support arises from a set that decomposes into two such arithmetic progressions under an affine map. Once this has been recognised, the conic is seen for what it is: one further trace of the same underlying mechanism.

Another source of confusion appears in the \(q\)-setting, where the map is of the form
\[
X:U\subseteq\mathbb C \longrightarrow\mathbb C,
\quad
X(s)=a\,q^{-s}+b\,q^s+c,
\]
with \(q\in\mathbb C^\times\), with a fixed determination of
\(q^s=e^{s\log q}\), and with \(a,b,c\in\mathbb C\) and
\((a,b)\neq(0,0)\). The difficulty becomes especially acute when \(q\) is a root of unity. This situation has sometimes been treated as though it gave rise, from within the theory itself, to a genuinely different kind of object; see, for instance, \cite{SZ96,SZ97}. Here, however, some care is needed. The values \(q=1\) and \(q=-1\) are not merely special cases of the general torsion phenomenon: the corresponding quadratic and alternating regimes are not obtained by direct substitution in the \(q\)-exponential formulas, and any scaled limiting procedure requires a separate analysis. The case \(q=1\) leads to the quadratic regime, where the \(q\)-geometric mechanism has collapsed into an additive, second-degree structure. The case \(q=-1\), by contrast, leads to an alternating regime, in which the even and odd parts of the parametrisation interact in a genuinely different way. As we shall see, the latter case has given rise to a substantial and partly self-contained literature, which, from our perspective, is often interpreted at a structurally inappropriate level through its association with certain Dunkl-type operators; see, for instance, \cite{TVZ12, VZ12b, GVZ13, TVZ13, GVZ14}. By contrast, when \(q\) is a root of unity of order at least \(3\), one remains within the same general pattern, now seen in a finite-order, or torsion, regime. Indeed, as long as \(q\) is not a root of unity, the values of \(X\) along a non-degenerate full-step progression typically lie on an infinite \(q\)-geometric configuration. When \(q\) has finite order, that configuration closes up in the relevant parametrisations, the image may become finite, and truncation phenomena may occur. If attention is confined to the visible support, it is easy to conclude that one has entered a genuinely different regime. Yet that impression is often deceptive. What may have changed is not the nature of the parametrising mechanism, but only the arithmetic regime in which one observes it. For this reason, finite families arising at roots of unity must be handled with particular care: the finiteness of the support and the presence of truncation are real phenomena, but neither of them, taken by itself, signals the emergence of a genuinely new situation. One may simply be seeing the finite shadow of an already existing mechanism. Seen in this way, the constructions studied in \cite{SZ96,SZ97} are best understood as explicit manifestations of the same general framework in a torsion regime.

The purpose of the present paper is not merely to reorganise known formulas at the level of presentation, nor to restore to what we take to be their proper place all familiar families that reappear under the mask of a specific problem, a particular notation, or a local point of view. It is, rather, to relocate the study of classical orthogonal polynomials to the level at which their underlying structure can be properly understood, and thereby to encourage a thorough re-examination of the subject. We do so by working not with image sets in isolation but with domains stable under half-step translation and with maps defined on them. This shift in viewpoint may seem modest at first, but its consequences are substantial. It explains why supports that look unfamiliar may nevertheless arise from very classical underlying configurations, why local algebraic relations among neighbouring values should be read as consequences rather than as points of departure, and why torsion phenomena in the \(q\)-regime belong to the same general picture rather than to a separate universe. At the same time, this relocation of the problem is inseparable from the way orthogonality itself is understood here. Throughout the paper, orthogonality is not tied in advance to positivity, to the existence of a measure, or to a preferred real support. It is formulated, rather, in the continuous dual of a suitable locally convex space (LCS) of polynomials. This is not an artificial generalisation introduced merely for the sake of abstraction; it is the natural level at which the algebraic content of the theory becomes fully visible. Once orthogonality is placed in that dual-topological setting, one can separate what is genuinely structural from what belongs only to a particular realisation: positivity becomes a special circumstance rather than a defining principle, finite and infinite situations can be treated within the same framework, and distinctions inherited from a narrow measure-theoretic presentation cease to obscure algebraic equivalences. In that sense, the topological point of view adopted here is not ornamental. It is part of the mechanism by which the correct structural picture comes into focus.

The examples discussed above should not be read as isolated curiosities, but rather as representative instances of a much broader phenomenon, one that recurs throughout the literature on the subject: known objects, presented under a different parametrisation or under a different visible support, come to be regarded as new because the level at which the comparison ought to be made has shifted unnoticed. We shall return to these examples later, once the necessary framework has been developed to place them where they properly belong. What matters here is that the accumulation of such cases makes it unmistakably clear that the issue is structural rather than accidental. In the theory of orthogonal polynomials, one may easily be misled by the surface form of a parametrisation, by the visible shape of a support, or by a particularly striking local identity, and thereby mistake a disguised reformulation for a genuinely new phenomenon. One should not begin by asking whether a visible support, or a particular weight arising in a positive orthogonality setting, appears to be new. One should begin by asking from what sets it arises, what class of map gives rise to it, and in what sense the orthogonality is to be understood. Only at that level can one distinguish what is intrinsic from what is merely presentational. Much of what follows is guided by this principle.

Section~2 introduces the class of half-step-invariant sets, for a fixed non-zero complex parameter \(h\), which provide the natural setting for the present theory, and shows that such sets decompose canonically into cosets of the additive subgroup \(\tfrac h2\mathbb Z\). Section~3 studies maps defined on these sets and isolates the basic notion of parametrising map, while Section~4 introduces the stronger notion of admissibility, whose role is to ensure that the symmetric divided-difference calculus closes polynomially and thus to make possible an intrinsic treatment of the geometry encoded by neighbouring values. Section~5 contains the corresponding classification theorem: every admissible map is shown to be necessarily of quadratic, alternating, or \(q\)-exponential type, with no other behaviour compatible with admissibility. Section~6 turns to the functional-analytic core of the paper. Orthogonality is formulated in the continuous dual of a suitable LCS of polynomials, and classicality is introduced by duality, through the transposed divided-difference and averaging operators. On that basis, we develop in Section~7 the regularity theory attached to the admissible map types and obtain explicit formulas for the recurrence coefficients of the corresponding classical orthogonal polynomial sequences in the ordinary quadratic and \(q\)-exponential regimes, including the torsion case of finite order; the alternating case is analysed separately in Section~8, since there the mechanism is governed instead by quadratic substitution and a first-order structural constraint. Section~9 then turns to the associated Nikiforov--Uvarov operator, establishes the precise relation between the notion of classicality in that framework and the one proposed here, and concludes with forward first-divided-difference stability in the ordinary quadratic and non-torsion \(q\)-exponential regimes. No complete Sonine--Hahn converse is asserted; indeed, without further hypotheses, we do not believe such a converse to be valid at this level of generality.

\section{Half-step-invariant sets}

We begin with a small piece of notation, followed by the simple invariance
condition on sets that will be used throughout the paper.

\begin{notation}\label{not:punctured-set}
Whenever \(E\subseteq\mathbb C\) is a set containing \(0\), we write
\[
E^\times=E\setminus\{0\}.
\]
\end{notation}

Fix \(h\in\mathbb C^\times\). In this section, \(U\subseteq\mathbb C\) denotes a non-empty set satisfying
\[
U\pm \frac{h}{2}\subseteq U.
\]
No further assumptions are imposed on \(U\); in particular, \(U\) need not be discrete, need not be contained in \(\mathbb R\), and need not carry any presupposed order.

We begin with the basic notion underlying the section.

\begin{definition}[Half-step-invariant set]\label{def:h2Z-invariant}
Fix \(h\in\mathbb C^\times\). A non-empty subset \(U\subseteq\mathbb C\) is called \emph{half-step-invariant} if
\[
U\pm \frac{h}{2}\subseteq U.
\]
The dependence on the fixed parameter \(h\) will remain implicit throughout.
\end{definition}

The following proposition justifies the terminology introduced in
Definition~\ref{def:h2Z-invariant}.

\begin{proposition}\label{lem:U-classification}
Fix \(h\in\mathbb C^\times\), and let \(U\subseteq\mathbb C\) be a half-step-invariant set. Then
\[
U+\frac{h}{2}\mathbb Z = U.
\]
In particular, \(U\) is a disjoint union of cosets of the additive subgroup
\(\tfrac{h}{2}\mathbb Z\); that is,
\begin{equation}\label{eq:U-decomposition}
U=\bigsqcup_{v\in V}\left(v+\frac{h}{2}\mathbb Z\right),
\end{equation}
where \(V\subseteq U\) is a set of representatives of the distinct cosets
of \(\tfrac{h}{2}\mathbb Z\) that meet \(U\).
\end{proposition}

\begin{proof}
By half-step invariance, for every \(k\in\mathbb Z\) we have \(U+\frac{h}{2} k\subseteq U\), hence
\[
U+\frac{h}{2}\mathbb Z\subseteq U.
\]
The reverse inclusion is immediate since \(0\in\mathbb Z\), so \(U\subseteq U+\tfrac{h}{2}\mathbb Z\). Thus \(U+\tfrac{h}{2}\mathbb Z=U\). Define a relation \(\sim\) on \(U\) by declaring that \(u\sim v\) if and only if
\[
u-v\in\frac{h}{2}\mathbb Z.
\]
Then \(\sim\) is an equivalence relation. The equivalence class of \(u\in U\) is initially
\[
\left(u+\frac h2\mathbb Z\right)\cap U.
\]
Since \(U+\tfrac h2\mathbb Z=U\), the full coset \(u+\tfrac h2\mathbb Z\) is contained in \(U\), and hence the class is precisely \(u+\tfrac h2\mathbb Z\). These classes partition \(U\). Choosing one representative from each class yields \eqref{eq:U-decomposition}.
\end{proof}

Proposition~\ref{lem:U-classification} shows that every half-step-invariant set \(U\) is a disjoint union of cosets of \(\tfrac{h}{2}\mathbb Z\). Equivalently, \(U\) may be viewed as a union of bi-infinite half-step arithmetic progressions, each lying in an affine line in \(\mathbb C\) parallel to \(h\).

The following examples illustrate some possible forms of half-step-invariant sets.

\begin{example}[Half-step-invariant sets]\label{ex:U-examples}
We retain the notation of Proposition~\ref{lem:U-classification}. We emphasise that the choice of representatives is not canonical: it depends on selecting one element from each coset contained in \(U\subseteq\mathbb C\).

\begin{enumerate}
\item
Let
\[
U=\frac{h}{2}\mathbb Z.
\]
Then \(U\) consists of a single coset, namely \(\tfrac{h}{2}\mathbb Z\) itself. Hence the decomposition \eqref{eq:U-decomposition} holds for any choice of a singleton set
\[
V=\{v\}\subseteq\frac{h}{2}\mathbb Z.
\]

\item
Let \(h=2\), and fix \(a,b\in\mathbb C\) such that \(a-b\notin\mathbb Z\).
Define
\[
U=\left\{\ldots,\,
a-\frac{3}{2},\,b-\frac{3}{2},\,
a-\frac{1}{2},\,b-\frac{1}{2},\,
a+\frac{1}{2},\,b+\frac{1}{2},\,
a+\frac{3}{2},\,b+\frac{3}{2},\,
\ldots\right\}.
\]
This displayed enumeration is purely suggestive; \(U\) is a set, and no ordering of its elements is intended. Equivalently,
\[
U=\left(a-\frac{1}{2}+\mathbb Z\right)\cup\left(b-\frac{1}{2}+\mathbb Z\right).
\]
Since \(a-b\notin\mathbb Z\), \(U\) is the disjoint union of two distinct cosets. Figure~\ref{fig:two-Z-orbits-h2} is schematic and serves only to indicate the decomposition of \(U\) into two half-step arithmetic progressions. Taking
\[
a=\frac12, \quad b=\frac{\gamma}{2}+\frac{1}{2},
\]
or vice versa, one recovers the set \(\mathbb S\) considered in the introduction. More general configurations of the same kind may readily be envisaged, as the following case shows.

\begin{figure}[!htbp]
\centering
\begin{tikzpicture}[scale=1.35,>=latex]
  \pgfmathsetmacro{\aRe}{0.30}
  \pgfmathsetmacro{\aIm}{0.55}
  \pgfmathsetmacro{\bRe}{1.35}
  \pgfmathsetmacro{\bIm}{-0.20}
  \pgfmathsetmacro{\ARe}{\aRe-0.5}
  \pgfmathsetmacro{\AIm}{\aIm}
  \pgfmathsetmacro{\BRe}{\bRe-0.5}
  \pgfmathsetmacro{\BIm}{\bIm}
  \pgfmathsetmacro{\xmin}{-3.30}
  \pgfmathsetmacro{\xmax}{ 3.30}
  \pgfmathsetmacro{\trackH}{0.20}
  \pgfmathsetmacro{\yTop}{max(\AIm,\BIm)}
  \pgfmathsetmacro{\yBot}{min(\AIm,\BIm)}
  \pgfmathsetmacro{\ymin}{\yBot-1.10}
  \pgfmathsetmacro{\ymax}{\yTop+1.10}
  \pgfmathsetmacro{\labelAboveY}{\yTop+0.38}
  \pgfmathsetmacro{\labelBelowY}{\yBot-0.38}
  \pgfmathsetmacro{\UlabelY}{\yTop+0.20}
  \coordinate (A0) at (\ARe,\AIm);
  \coordinate (B0) at (\BRe,\BIm);
   \path[shade, shading=axis,
        left color=white, middle color=gray!12, right color=white,
        shading angle=0, draw=none, rounded corners=1.2pt]
    (\xmin,{\AIm-\trackH}) rectangle (\xmax,{\AIm+\trackH});
  \path[shade, shading=axis,
        left color=white, middle color=gray!12, right color=white,
        shading angle=0, draw=none, rounded corners=1.2pt]
    (\xmin,{\BIm-\trackH}) rectangle (\xmax,{\BIm+\trackH});
  \draw[gray!55,dashed] (\xmin,\AIm) -- (\xmax,\AIm);
  \draw[gray!55,dashed] (\xmin,\BIm) -- (\xmax,\BIm);
\coordinate (Odir) at (-2.75,-0.95);
\pgfmathsetmacro{\Ldir}{0.55}
\draw[->,black,thick] (Odir) -- ++(\Ldir,0)
  node[below right,xshift=-1pt,yshift=1pt] {$h$};
\draw[->,black,thick] (Odir) -- ++(0,\Ldir)
  node[above left,xshift=1pt,yshift=-1pt] {$ih$};
  \foreach \n in {-2,-1,0,1,2}{
    \fill[black] ($(A0)+(\n,0)$) circle (1.15pt);
    \filldraw[black]
      ($($(B0)+(\n,0)$)+(-0.032,-0.032)$) rectangle ($($(B0)+(\n,0)$)+(0.032,0.032)$);
  }
  \pgfmathsetmacro{\tick}{0.06}
  \pgfmathsetmacro{\dymeas}{0.18}
  \coordinate (mA) at ($(B0)+(-2,-\dymeas)$);
  \coordinate (mB) at ($(B0)+(-1,-\dymeas)$);
  \draw[black] (mA) -- (mB);
  \draw[black] ($(mA)+(0,\tick)$) -- ($(mA)-(0,\tick)$);
  \draw[black] ($(mB)+(0,\tick)$) -- ($(mB)-(0,\tick)$);
  \node[black] at ($ (mA)!0.5!(mB) + (0,-0.16) $) {$1$};
  \ifdim \AIm pt>\BIm pt
    \node[black] at (2,\labelAboveY) {$a-\tfrac12+\mathbb Z$};
    \node[black] at (3,\labelBelowY) {$b-\tfrac12+\mathbb Z$};
  \else
    \node[black] at (0,\labelAboveY) {$b-\tfrac12+\mathbb Z$};
    \node[black] at (0,\labelBelowY) {$a-\tfrac12+\mathbb Z$};
  \fi
  \node[black] at (\xmin+1.2,\UlabelY) {\(U\)};
\end{tikzpicture}
\caption[Two cosets of \(\mathbb Z\)]{Black discs denote points of the coset \(a-\tfrac12+\mathbb Z\), and black squares denote points of the coset \(b-\tfrac12+\mathbb Z\).}
\label{fig:two-Z-orbits-h2}
\end{figure}

\item
Fix \(h\in\mathbb C^\times\) and consider the strip
\[
S=\left\{(x+iy)h:\ 0\le x<\frac12,\ y\in\mathbb R\right\}.
\]
Let \(E\subseteq S\) be non-empty and define
\[
U=\bigcup_{v\in E}\left(v+\frac{h}{2}\mathbb Z\right).
\]
Then \(U\) is half-step-invariant. Moreover, the union is disjoint. Indeed, each coset occurring in the union meets \(S\) in exactly one point, namely its chosen representative \(v\in E\). Existence is immediate from the definition, since
\[
v\in \left(v+\frac h2\mathbb Z\right)\cap S
\]
for every \(v\in E\). For uniqueness, let \(v,w\in S\) and assume that
\(
v+\frac h2\mathbb Z=w+\frac h2\mathbb Z.
\)
Then \(v-w\in\tfrac h2\mathbb Z\). On the other hand, \(\Re\left(\frac vh\right), \Re\left(\frac wh\right)\in\left[0,\frac12\right)\), and therefore
\(
\Re(\frac{v-w}{h})\in(-\frac12,\frac12).
\)
Since
\[
\frac{v-w}{h}\in\frac12\mathbb Z\subset\mathbb R,
\]
it follows that
\[
\frac{v-w}{h}\in \frac12\mathbb Z\cap\left(-\frac12,\frac12\right)=\{0\},
\]
hence \(v=w\). Therefore distinct \(v\in E\) lie in distinct cosets, and the sets
\[
\left\{\,v+\frac h2\mathbb Z:\ v\in E\,\right\}
\]
are pairwise disjoint. Conversely, if \(z/h=x+iy\), with \(x,y\in\mathbb R\),
then, on taking \(k=\lfloor2x\rfloor\), the point \(v=z-kh/2\) belongs to
\(S\) and to the coset \(z+\tfrac h2\mathbb Z\). The strip \(S\) therefore acts as a transversal for the cosets of
\(\tfrac h2\mathbb Z\): each coset has exactly one representative in \(S\).
This is illustrated schematically in Figure~\ref{fig:strip-discrete-E}.
The choice of \(E\) determines how \(U\) sits inside \(\mathbb C\). For example, if \(E\) is finite, then \(U\) is a finite disjoint union of discrete arithmetic progressions; in Figure~\ref{fig:strip-discrete-E} one has \(E=\{v_1,v_2,v_3\}\). Conversely, if \(E\) is dense in \(S\), then \(U\) is dense in \(\mathbb C\). Indeed, for every \(z\in\mathbb C\), there are unique \(v\in S\) and \(k\in\mathbb Z\) such that
\[
z=v+\frac h2 k.
\]
If \(E\) is dense in \(S\), then one may choose a sequence \(v_n\in E\) with \(v_n\to v\). Hence \(v_n+\frac h2 k\in U\) and
\[
v_n+\frac h2 k\to z,
\]
so \(U\) is dense in \(\mathbb C\).

A concrete dense example is
\[
E=\left\{(r+si)h:\ r,s\in\mathbb Q,\ 0\le r<\frac12\right\}.
\]
Then \(E\) is dense in \(S\), and since \(\mathbb{Q}=\bigl(\mathbb{Q}\cap[0,\tfrac12)\bigr)+\tfrac12\,\mathbb{Z}\) it follows that
\(
U=\left\{(p+qi)h:\ p,q\in\mathbb Q\right\},
\)
which is dense in \(\mathbb C\). The preceding examples should not, however, be taken to suggest that half-step-invariant sets must essentially resemble a discretisation of \(\mathbb C\). As the following case shows, they may also arise in a rather different geometric form, in which the underlying coset structure is much less immediately apparent.

\begin{figure}[!htbp]
\centering
\begin{tikzpicture}[scale=1.5,>=latex]

  \pgfmathsetmacro{\hx}{1.10}
  \pgfmathsetmacro{\hy}{0.65}
  \coordinate (O)  at (0,0);
  \coordinate (h)  at (\hx,\hy);
  \coordinate (ih) at (-\hy,\hx);          
  \coordinate (hh) at ({\hx/2},{\hy/2});   

  \pgfmathsetmacro{\Bstrip}{1.55}  
  \pgfmathsetmacro{\HalfLen}{2.55} 


  \path[fill=gray!9,draw=none,rounded corners=1.2pt]
    ($(O)-\Bstrip*(ih)$) -- ($(O)+\Bstrip*(ih)$) --
    ($(hh)+\Bstrip*(ih)$) -- ($(hh)-\Bstrip*(ih)$) -- cycle;

  \draw[gray!70,thick] ($(O)-\Bstrip*(ih)$) -- ($(O)+\Bstrip*(ih)$)
    node[pos=0.965,above left,xshift=-5pt] {$x=0$};
  \draw[gray!70,thick,dashed] ($(hh)-\Bstrip*(ih)$) -- ($(hh)+\Bstrip*(ih)$)
    node[pos=0.965,above left,yshift=5pt] {$x=\tfrac12$};

  \coordinate (SbotO) at ($(O)-\Bstrip*(ih)$);
  \coordinate (SbotH) at ($(hh)-\Bstrip*(ih)$);
  \node[black] at ($($(SbotO)!0.70!(SbotH)$) + 0.18*(ih)$) {$S$};

  \coordinate (Odir) at ($(O)-1.15*(h)+1.20*(ih)$);
  \draw[->,black,thick] (Odir) -- ($(Odir)+(h)$)
    node[below right,xshift=-2pt] {$h$};
  \draw[->,black,thick] (Odir) -- ($(Odir)+(ih)$)
    node[above left,yshift=-2pt] {$ih$};

  \coordinate (v1) at ($(O) + 0.14*(h) + 0.95*(ih)$);
  \coordinate (v2) at ($(O) + 0.36*(h) - 0.40*(ih)$);
  \coordinate (v3) at ($(O) + 0.06*(h) + 0.10*(ih)$);

  \draw[black!25,dashed] ($(v1)-\HalfLen*(h)$) -- ($(v1)+\HalfLen*(h)$);
  \draw[black!25,dashed] ($(v2)-\HalfLen*(h)$) -- ($(v2)+\HalfLen*(h)$);
  \draw[black!25,dashed] ($(v3)-\HalfLen*(h)$) -- ($(v3)+\HalfLen*(h)$);

  \foreach \n in {-4,-3,...,5}{
    \fill[black] ($(v1)+\n*(hh)$) circle (0.85pt);
    \fill[black] ($(v2)+\n*(hh)$) circle (0.85pt);
    \fill[black] ($(v3)+\n*(hh)$) circle (0.85pt);
  }

  \node[black] at ($(O)+0.95*(h)+0.65*(ih)$) {\(U\)};

  \coordinate (mA) at ($(v3)+2*(hh)$);
  \coordinate (mB) at ($(v3)+3*(hh)$);

  \coordinate (mAoff) at ($(mA)-0.06*(ih)$);
  \coordinate (mBoff) at ($(mB)-0.06*(ih)$);

  \draw[black] (mAoff) -- (mBoff);
  \draw[black] ($(mAoff)+0.025*(ih)$) -- ($(mAoff)-0.025*(ih)$);
  \draw[black] ($(mBoff)+0.025*(ih)$) -- ($(mBoff)-0.025*(ih)$);

  \node[black] at ($ (mAoff)!0.5!(mBoff) - 0.135*(ih) $) {$\tfrac h2$};

  \filldraw[black] ($(v1)+(-0.024,-0.024)$) rectangle ($(v1)+(0.024,0.024)$);
  \filldraw[black] ($(v2)+(-0.024,-0.024)$) rectangle ($(v2)+(0.024,0.024)$);
  \filldraw[black] ($(v3)+(-0.024,-0.024)$) rectangle ($(v3)+(0.024,0.024)$);

  \node[above left,yshift=2pt]  at (v1) {$v_1$};
  \node[below right,yshift=-2pt] at (v2) {$v_2$};
  \node[above right,yshift=2pt]  at (v3) {$v_3$};

\end{tikzpicture}

\caption[Discrete arithmetic progressions and a transversal strip]{The strip \(S\) provides one representative for each coset of \(\tfrac h2\mathbb Z\); the figure shows three selected representatives. Black squares denote representatives, black discs denote points in the corresponding arithmetic progressions, and the pale dashed lines denote the affine lines parallel to \(h\) passing through these representatives.}\label{fig:strip-discrete-E}
\end{figure}

\item
Let \(E\subseteq\mathbb C^\times\) be non-empty, and define
\[
U=\left\{z\in\mathbb C:\exp\!\left(\frac{4\pi i}{h}z\right)\in E\right\}.
\]
Equivalently, \(U\) is the preimage of \(E\) under the map
\[
z\longmapsto \exp\!\left(\frac{4\pi i}{h}z\right).
\]
Because the complex exponential maps \(\mathbb C\) onto \(\mathbb C^\times\), the set \(U\) is non-empty whenever \(E\neq\varnothing\). Moreover,
\[
\exp\!\left(\frac{4\pi i}{h}\left(z\pm\frac h2\right)\right)
=
\exp\!\left(\frac{4\pi i}{h}z\right),
\]
so \(U\) is half-step-invariant. By Proposition~\ref{lem:U-classification}, \(U\) is therefore a disjoint union of cosets, even though this coset decomposition is not immediately evident from the defining condition. Now specialise to \(h=1\) and take \(E=A\), where
\[
A=
\left\{
w\in\mathbb C :
\frac12 < |w| < 2,
\,
-\frac{\pi}{4} < \Arg(w) < \frac{\pi}{4}
\right\},
\]
where \(\Arg\) denotes the principal argument. Writing \(z=x+iy\), the condition \(\exp(4\pi i z)\in A\) becomes
\[
-\frac{\ln 2}{4\pi}<y<\frac{\ln 2}{4\pi},
\quad
x\in \bigcup_{n\in\mathbb Z}\left(\frac{n}{2}-\frac{1}{16},\,\frac{n}{2}+\frac{1}{16}\right).
\]
Thus, in this case,
\[
U
=
\bigsqcup_{n\in\mathbb Z}
\left(
\left(
\frac{n}{2}-\frac{1}{16},
\frac{n}{2}+\frac{1}{16}
\right)
\times
\left(
-\frac{\ln 2}{4\pi},
\frac{\ln 2}{4\pi}
\right)
\right).
\]
The set \(U\) is therefore a periodic array of translates of a single rectangle. This configuration is illustrated schematically in Figure~\ref{fig:exp-preimage-sector}: the left panel shows the annular sector \(A\), while the right panel shows its preimage \(U\) under \(z\mapsto \exp(4\pi i z)\).

\begin{figure}[t]
\centering

\begin{minipage}[c]{0.40\textwidth}
\centering
\resizebox{\linewidth}{!}{%
\begin{tikzpicture}[>=latex]
  \def\rIn{0.5}
  \def\rOut{2.0}
  \def\aUp{45}
  \def\aDn{-45}

  \begin{scope}
    \clip
      (\aUp:\rOut) arc[start angle=45,end angle=-45,radius=\rOut]
      -- (\aDn:\rIn) arc[start angle=-45,end angle=45,radius=\rIn] -- cycle;

    \path[shade, shading=axis,
          left color=gray!6, right color=gray!6, middle color=gray!14,
          shading angle=0, draw=none]
      (-3.0,-3.0) rectangle (3.0,3.0);
  \end{scope}

  \draw[->,gray!60] (-2.35,0) -- (2.45,0) node[right] {$\Re w$};
  \draw[->,gray!60] (0,-2.35) -- (0,2.45) node[above] {$\Im w$};

  \draw[gray!55,dashed] (\aUp:\rOut) arc[start angle=45,end angle=315,radius=\rOut];
  \draw[gray!55,dashed] (\aUp:\rIn)  arc[start angle=45,end angle=315,radius=\rIn];

  \draw[gray!55,dashed] (\aUp:\rOut) arc[start angle=45,end angle=-45,radius=\rOut];
  \draw[gray!55,dashed] (\aUp:\rIn)  arc[start angle=45,end angle=-45,radius=\rIn];
  \draw[gray!55,dashed] (\aUp:\rIn) -- (\aUp:\rOut);
  \draw[gray!55,dashed] (\aDn:\rIn) -- (\aDn:\rOut);

  \draw[gray!55,dashed] (0,0) -- (\aUp:\rIn);
  \draw[gray!55,dashed] (0,0) -- (\aDn:\rIn);

  \draw[gray!45] (145:\rOut) -- ++(-0.34,0.16)
    node[left,gray!70] {$|w|=2$};
  \draw[gray!45] (205:\rIn) -- ++(-0.42,-0.12)
    node[left,gray!70] {$|w|=\tfrac12$};

  \node[black] at (10:1.20) {$A$};
\end{tikzpicture}%
}
\end{minipage}
\begin{minipage}[c]{0.16\textwidth}
\centering
\[
\xleftarrow{\;w=e^{4\pi iz}\;}
\]
\end{minipage}
\begin{minipage}[c]{0.40\textwidth}
\centering
\resizebox{\linewidth}{!}{%
\begin{tikzpicture}[>=latex]
  \pgfmathsetmacro{\dx}{1/16}
  \pgfmathsetmacro{\dy}{ln(2)/(4*pi)}

  \begin{scope}[x=3.0cm,y=3.0cm]

    \foreach \c in {-0.5,0.0,0.5,1.0}{
      \path[shade, shading=axis,
            left color=gray!6, right color=gray!6, middle color=gray!14,
            shading angle=0, draw=none]
        (\c-\dx,-\dy) rectangle (\c+\dx,\dy);
    }

    \draw[->,gray!60] (-1.05,0) -- (1.40,0) node[right] {$x$};
    \draw[->,gray!60] (0,-0.15) -- (0,0.15) node[above] {$y$};
    \foreach \c in {-0.5,0.0,0.5,1.0}{
      \draw[gray!55,dashed] (\c-\dx,-\dy) rectangle (\c+\dx,\dy);
      \draw[gray!60] (\c,0.006) -- (\c,-0.006);
    }

    \node[black] at (0.5,{\dy+0.02}) {\(U\)};

    \coordinate (Odir) at (1.07,-0.07);

    \pgfmathsetmacro{\Ldir}{0.14}
    \pgfmathsetmacro{\LdirV}{\Ldir}

    \draw[->,black,thick] (Odir) -- ++(\Ldir,0)
      node[below right,xshift=-1pt,yshift=1pt] {$h$};

    \draw[->,black,thick] (Odir) -- ++(0,\LdirV)
      node[above left,xshift=1pt,yshift=-1pt] {$ih$};

    \pgfmathsetmacro{\TickS}{0.006}
    \pgfmathsetmacro{\TickL}{0.018}
    \pgfmathsetmacro{\LabDn}{0.022}
    \pgfmathsetmacro{\LabUp}{0.012}

    \coordinate (pA) at (0,{-\dy-0.030});
    \coordinate (pB) at (0.5,{-\dy-0.030});
    \draw[black] (pA) -- (pB);
    \draw[black] ($(pA)+(0,\TickS)$) -- ($(pA)-(0,\TickS)$);
    \draw[black] ($(pB)+(0,\TickS)$) -- ($(pB)-(0,\TickS)$);
    \node[black] at ($ (pA)!0.5!(pB) + (0,-\LabDn) $) {$\tfrac12$};

    \coordinate (wA) at (-\dx,{\dy+0.022});
    \coordinate (wB) at ( \dx,{\dy+0.022});
    \draw[black] (wA) -- (wB);
    \draw[black] ($(wA)+(0,\TickS)$) -- ($(wA)-(0,\TickS)$);
    \draw[black] ($(wB)+(0,\TickS)$) -- ($(wB)-(0,\TickS)$);
    \node[black,anchor=west] at ($ (wA)!0.5!(wB) + (0.060,\LabUp) $) {$\tfrac18$};

    \pgfmathsetmacro{\xh}{-0.78}
    \coordinate (hA) at (\xh,-\dy);
    \coordinate (hB) at (\xh,\dy);
    \draw[black] (hA) -- (hB);
    \draw[black] ($(hA)+(\TickL,0)$) -- ($(hA)-(\TickL,0)$);
    \draw[black] ($(hB)+(\TickL,0)$) -- ($(hB)-(\TickL,0)$);

    \node[black,anchor=east] at (\xh-0.03,0) {$\tfrac{\ln 2}{2\pi}$};

  \end{scope}
\end{tikzpicture}%
}
\end{minipage}

\caption[Preimage of an annular sector]{For \(h=1\), every displayed rectangle, of width \(\tfrac18\) and height \(\ln2/(2\pi)\), is carried bijectively onto \(A\) by \(w=e^{4\pi iz}\); successive rectangles differ by the period \(\tfrac12\).}
\label{fig:exp-preimage-sector}
\end{figure}

\end{enumerate}
\end{example}

\section{Parametrising maps}

Fix \(h\in\mathbb C^\times\), let \(U\subseteq\mathbb C\) be a half-step-invariant set, for instance any of those presented in Example~\ref{ex:U-examples}, and let \(X:U\to\mathbb C\) be a map.

\begin{notation}
With \(h\), \(U\), and \(X\) as above, observe first that, since \(U\) is half-step-invariant, it is invariant under translations by \(\pm h/2\) and hence also by \(\pm h\). Therefore, for every \(s\in U\), the points \(s\pm h/2\) and \(s\pm h\) still belong to \(U\). We may consequently define
\begin{align*}
Y(s)&= X\!\left(s+\frac{h}{2}\right), \quad Z(s)= X\!\left(s-\frac{h}{2}\right),\\[7pt]
Y_1(s)&= X(s+h), \quad \,\,\,\, Z_1(s)= X(s-h).
\end{align*}
Thus \(Y\) and \(Z\) denote the half-step neighbours of \(X\), whereas
\(Y_1\) and \(Z_1\) denote its full-step neighbours.
All identities involving the functions \(X\), \(Y\), \(Z\), \(Y_1\), and \(Z_1\) are understood to hold on \(U\).
\end{notation}

We now single out the class of maps on \(U\) for which the two half-step neighbours never coincide.

\begin{definition}[Parametrising map]\label{def:lattice}
Fix \(h\in\mathbb C^\times\), and let \(U\subseteq\mathbb C\) be a half-step-invariant set. A \emph{parametrising map on \(U\)} is a map
\(
X:U\to\mathbb C
\)
such that \(Y-Z\) is nowhere-vanishing on \(U\). This notion is understood relative to the same fixed parameter \(h\) that governs the half-step-invariance of \(U\) and the definitions of \(Y\) and \(Z\). That dependence will remain implicit throughout.
\end{definition}

No condition is imposed in Definition~\ref{def:lattice} on the full-step difference \(Y_1-Z_1\). Without some additional restriction, however, the class of parametrising maps is far too broad.

\begin{example}[Parametrising maps on the preceding half-step-invariant sets]\label{ex:lattices-on-U}
The following examples illustrate parametrising maps on the half-step-invariant sets described in Example~\ref{ex:U-examples}.
\begin{enumerate}
\item
Let \(h=2\) and \(U=\mathbb Z\). Fix \(\alpha,\beta\in\mathbb C\) with \(\alpha\neq \tfrac12\), and define
\begin{align}\label{eq:bi-lattice-map}
X(s)=\left(\alpha-\frac12\right)s+\left(\beta-\frac12\right)\bigl(1-(-1)^s\bigr)
\end{align}
for every \(s\in U\). Then
\[
Y-Z=2\left(\alpha-\frac12\right)
\]
throughout \(U\), and hence \(X\) is a parametrising map on \(U\). Moreover, setting
\[
c=2\alpha-1,\quad d=\alpha+2\beta-\frac32,
\]
we have
\begin{align}\label{eq:bi-lattice-image}
X(U)=c\,\mathbb Z \cup \bigl(d+c\,\mathbb Z\bigr).
\end{align}
In particular, if one takes
\[
\alpha=\frac32, \quad \beta=\frac{\gamma}{2},
\]
one recovers
\[
X(s)=s+\frac12(\gamma-1)\bigl(1-(-1)^s\bigr),
\]
namely, the map discussed in the introduction.

\item
Let \(U\) be as in Example~\ref{ex:U-examples}(2). Fix \(c\in\mathbb C^\times\) and \(d\in\mathbb C\), and assume that
\[
b-a\in\left(\frac{d}{c}+\mathbb Z\right)\cup\left(-\frac{d}{c}+\mathbb Z\right).
\]
Then there exists an affine map \(X(s)=cs+e\) such that
\[
X(U)=c\,\mathbb Z\ \cup\ \bigl(d+c\,\mathbb Z\bigr).
\]
Indeed, if \(b-a-\frac{d}{c}\in\mathbb Z\), one may take
\[
X(s)=cs-c\left(a-\frac12\right),
\]
whereas if \(b-a+\frac{d}{c}\in\mathbb Z\), one may take
\[
X(s)=cs-c\left(a-\frac12\right)+d.
\]
Then
\[
Y-Z=2c
\]
throughout \(U\), and hence \(X\) is a parametrising map on \(U\). In either case, the image \(X(U)\) is equal to a set of the form displayed in item~(1). In particular, taking
\[
a=\frac12,\quad b=\frac{\gamma+1}{2},\quad c=2,\quad d=\gamma,
\]
we have
\[
b-a=\frac{\gamma}{2}=\frac{d}{c},
\]
and therefore one may choose
\[
X(s)=2s,
\]
which is precisely the affine map presented in the introduction as an alternative to the map of item~(1).

\item
Let \(U\) be as in Example~\ref{ex:U-examples}(3). Assume that \(0<q<1\)\footnote{The restriction \(0<q<1\) is imposed here only for simplicity of exposition. It allows one to write \(q^s=e^{s\ln q}\) with the ordinary real logarithm, thereby avoiding any discussion of branches.}, and define
\[
X(s)=q^{\,s},\quad q^{\,s}=e^{s\ln q},
\]
for every \(s\in U\). Then
\[
(Y-Z)(s)
= q^{\,s-\frac h2}\bigl(q^{\,h}-1\bigr).
\]
Since \(q^{\,s-\frac h2}\neq0\) for every \(s\in\mathbb C\), it follows that \(Y-Z\) is nowhere vanishing on \(U\) if and only if \(q^{\,h}\neq1\). Thus \(X\) is a parametrising map on \(U\) precisely under this condition. Moreover,
\[
X(U)
=\bigcup_{v\in E} \left\{ q^s : s \in v + \frac{h}{2}\mathbb Z \right\}.
\]
Figure~\ref{fig:XU-qexp-three-spirals} is schematic: for each \(v\in E\), the coset \(v+\tfrac h2\mathbb Z\) is mapped by \(X(s)=q^s\) onto the geometric progression
\[
q^{\,v}\left(q^{\frac h2}\right)^{\mathbb Z},
\]
and \(X(U)\) is obtained as the union of these images.

\begin{figure}[t]
\centering
\begin{tikzpicture}[scale=1,>=latex]

  \pgfmathsetmacro{\Den}{20}
  \pgfmathsetmacro{\lnq}{-1/\Den}

  \pgfmathsetmacro{\hx}{1.0}
  \pgfmathsetmacro{\hy}{10.0}
  \pgfmathsetmacro{\hhx}{\hx/2}
  \pgfmathsetmacro{\hhy}{\hy/2}

  \pgfmathsetmacro{\rOne}{0.18} \pgfmathsetmacro{\sOne}{ 0.04}
  \pgfmathsetmacro{\rTwo}{0.38} \pgfmathsetmacro{\sTwo}{-0.03}
  \pgfmathsetmacro{\rThr}{0.08} \pgfmathsetmacro{\sThr}{ 0.09}

  \pgfmathsetmacro{\voneX}{\rOne*\hx + \sOne*(-\hy)}
  \pgfmathsetmacro{\voneY}{\rOne*\hy + \sOne*( \hx)}
  \pgfmathsetmacro{\vtwoX}{\rTwo*\hx + \sTwo*(-\hy)}
  \pgfmathsetmacro{\vtwoY}{\rTwo*\hy + \sTwo*( \hx)}
  \pgfmathsetmacro{\vthrX}{\rThr*\hx + \sThr*(-\hy)}
  \pgfmathsetmacro{\vthrY}{\rThr*\hy + \sThr*( \hx)}

  \pgfmathsetmacro{\Rmax}{4.10}   
  \pgfmathsetmacro{\AxExt}{0.35}  
  \pgfmathsetmacro{\Tmin}{-160}
  \pgfmathsetmacro{\Tmax}{  40}
  \pgfmathsetmacro{\Nmin}{-150}
  \pgfmathsetmacro{\Nmax}{  35}

  \pgfmathsetmacro{\rcirc}{1.35}  
  \pgfmathsetmacro{\msq}{0.032}   
  \pgfmathsetmacro{\mtr}{0.040}   

  \tikzset{guide/.style={
    gray!55,
    line width=0.65pt,
    dash pattern=on 2.1pt off 1.7pt,
    line cap=round
  }}

  \begin{scope}[x=0.90cm,y=0.90cm]

    \draw[->,gray!60] (-\Rmax,0) -- (\Rmax+\AxExt,0)
      node[right,scale=0.75] {$\Re\!\bigl(X(s)\bigr)$};
    \draw[->,gray!60] (0,-\Rmax) -- (0,\Rmax+\AxExt)
      node[above,scale=0.75] {$\Im\!\bigl(X(s)\bigr)$};

    \begin{scope}
      \clip (0,0) circle (\Rmax);

      \draw[guide,domain=\Tmin:\Tmax,samples=1400,smooth,variable=\t]
        plot ({exp((\voneX+\t*\hhx)*\lnq)*cos(deg((\voneY+\t*\hhy)*\lnq))},
              {exp((\voneX+\t*\hhx)*\lnq)*sin(deg((\voneY+\t*\hhy)*\lnq))});

      \draw[guide,domain=\Tmin:\Tmax,samples=1400,smooth,variable=\t]
        plot ({exp((\vtwoX+\t*\hhx)*\lnq)*cos(deg((\vtwoY+\t*\hhy)*\lnq))},
              {exp((\vtwoX+\t*\hhx)*\lnq)*sin(deg((\vtwoY+\t*\hhy)*\lnq))});

      \draw[guide,domain=\Tmin:\Tmax,samples=1400,smooth,variable=\t]
        plot ({exp((\vthrX+\t*\hhx)*\lnq)*cos(deg((\vthrY+\t*\hhy)*\lnq))},
              {exp((\vthrX+\t*\hhx)*\lnq)*sin(deg((\vthrY+\t*\hhy)*\lnq))});

      \foreach \n in {\Nmin,...,\Nmax}{

        \pgfmathsetmacro{\sx}{\voneX + \n*\hhx}
        \pgfmathsetmacro{\sy}{\voneY + \n*\hhy}
        \pgfmathsetmacro{\rad}{exp(\sx*\lnq)}
        \pgfmathsetmacro{\ang}{\sy*\lnq}
        \fill[black] ({\rad*cos(deg(\ang))},{\rad*sin(deg(\ang))}) circle (\rcirc pt);

        \pgfmathsetmacro{\sx}{\vtwoX + \n*\hhx}
        \pgfmathsetmacro{\sy}{\vtwoY + \n*\hhy}
        \pgfmathsetmacro{\rad}{exp(\sx*\lnq)}
        \pgfmathsetmacro{\ang}{\sy*\lnq}
        \fill[black] ({\rad*cos(deg(\ang))},{\rad*sin(deg(\ang))}) circle (\rcirc pt);

        \pgfmathsetmacro{\sx}{\vthrX + \n*\hhx}
        \pgfmathsetmacro{\sy}{\vthrY + \n*\hhy}
        \pgfmathsetmacro{\rad}{exp(\sx*\lnq)}
        \pgfmathsetmacro{\ang}{\sy*\lnq}
        \fill[black] ({\rad*cos(deg(\ang))},{\rad*sin(deg(\ang))}) circle (\rcirc pt);
      }
    \end{scope}

    \node[black, scale=0.7] at (2.10,2.35) {$X(U)$};

  \end{scope}

\end{tikzpicture}
\caption[Discrete spirals for \(X(s)=q^s\)]
{The set \(X(U)\) for \(X(s)=q^s\) with \(q=e^{-1/20}\) and \(h=1+10i\), where \(U\) is the union of the three cosets \(v_j+\tfrac h2\mathbb Z\) determined by the representatives \(v_1=(0.18+0.04\,i)h\), \(v_2=(0.38-0.03\,i)h\), and \(v_3=(0.08+0.09\,i)h\). Each arithmetic progression is mapped onto a discrete subset of a logarithmic spiral; the black points represent the plotted elements of \(X(U)\), and the dashed curves are included to indicate the underlying spirals.}
\label{fig:XU-qexp-three-spirals}
\end{figure}

\item
Let \(U\) be as in Example~\ref{ex:U-examples}(4). Define
\[
X(s)=s^2+\frac14\,s,
\]
for every \(s\in U\). A direct computation gives
\[
(Y-Z)(s)
= h\left(2s+\frac14\right),
\]
so, since \(h\neq0\), the function \(Y-Z\) has the unique zero \(s=-\tfrac18\). Hence \(X\) is a parametrising map on \(U\) if and only if \(-\tfrac18\notin U\). In the particular case
\[
U
=
\bigsqcup_{n\in\mathbb Z}
\left(
\left(
\frac{n}{2}-\frac{1}{16},
\frac{n}{2}+\frac{1}{16}
\right)
\times
\left(
-\frac{\ln 2}{4\pi},
\frac{\ln 2}{4\pi}
\right)
\right),
\]
one has \(-\tfrac18\notin U\). Indeed, if \(-\tfrac18\in U\), then for some \(n\in\mathbb Z\) one would have
\[
\left|-\frac18-\frac n2\right|<\frac1{16}.
\]
However,
\[
\left|-\frac18-\frac n2\right|
=
\frac{|4n+1|}{8}
\ge
\frac18
>
\frac1{16},
\]
since \(4n+1\) is a non-zero integer. This contradiction shows that \(-\tfrac18\notin U\). Consequently, \(X\) is a parametrising map on \(U\). Figure~\ref{fig:quadratic-image-sector} shows the image, under \(X\), of the four rectangles displayed in Figure~\ref{fig:exp-preimage-sector}. In contrast with the original periodic rectangular decomposition, these images are bounded by quadratic arcs, thereby reflecting the non-affine nature of the parametrisation.
\end{enumerate}

\begin{figure}[t]
\centering
\resizebox{0.80\linewidth}{!}{%
\begin{tikzpicture}[>=latex,x=5.4cm,y=5.4cm]

  \pgfmathsetmacro{\dx}{1/16}
  \pgfmathsetmacro{\dy}{ln(2)/(4*pi)}

  \tikzset{
    imagepiece/.style={
      fill=gray!12,
      draw=gray!72,
      line width=0.35pt,
      dash pattern=on 1.8pt off 1.4pt,
      line cap=round,
      line join=round
    }
  }

  \newcommand{\DrawImagePiece}[1]{%
    \pgfmathsetmacro{\c}{#1}
    \pgfmathsetmacro{\xl}{\c-\dx}
    \pgfmathsetmacro{\xr}{\c+\dx}

    \path[imagepiece]
      plot[domain=-\dy:\dy,samples=120,variable=\t]
        ({\xl*\xl-\t*\t+\xl/4},{\t*(2*\xl+1/4)})
      --
      plot[domain=\xl:\xr,samples=120,variable=\x]
        ({\x*\x-\dy*\dy+\x/4},{\dy*(2*\x+1/4)})
      --
      plot[domain=\dy:-\dy,samples=120,variable=\t]
        ({\xr*\xr-\t*\t+\xr/4},{\t*(2*\xr+1/4)})
      --
      plot[domain=\xr:\xl,samples=120,variable=\x]
        ({\x*\x-\dy*\dy+\x/4},{-\dy*(2*\x+1/4)})
      -- cycle;
  }

  \foreach \c in {-0.5,0,0.5,1}{
    \DrawImagePiece{\c}
  }

  \draw[->,gray!68,line width=0.35pt] (-0.06,0) -- (1.48,0)
    node[right,scale=0.72] {$\Re\!\bigl(X(s)\bigr)$};
  \draw[->,gray!68,line width=0.35pt] (0,-0.155) -- (0,0.155)
    node[above,scale=0.72] {$\Im\!\bigl(X(s)\bigr)$};

  \foreach \u in {0,0.125,0.375,1.25}{
    \draw[gray!65,line width=0.30pt] (\u,0.005) -- (\u,-0.005);
  }

\end{tikzpicture}%
}
\caption[Quadratic image of the four displayed rectangles]
{The four rectangles appearing in the right-hand panel of
Figure~\ref{fig:exp-preimage-sector}, mapped by
\(X(s)=s^2+\tfrac14s\). The two coordinate directions are drawn to the
same scale.}
\label{fig:quadratic-image-sector}
\end{figure}

\end{example}
In Examples~\ref{ex:lattices-on-U}(1) and~\ref{ex:lattices-on-U}(2), the parametrising maps are genuinely different, even though their image sets coincide. This may at first suggest that both choices are equally compatible with the theory under consideration here. The admissibility condition introduced in the following section is intended precisely to determine when each choice is compatible with the theory. This conclusion should not be read as saying that a visible image determines a unique polynomial calculus: in general, it does not.

\section{Admissible maps}

The following definition introduces the sole restriction placed here on parametrising maps. It constrains simultaneously the map itself and the half-step-invariant set on which it is defined.

\begin{definition}[Admissible map\footnote{The term \emph{admissible map} is used here in place of \emph{lattice}, as found in the English-language literature on orthogonal polynomials, since the latter may suggest an algebraic or group-theoretic lattice structure that is not intended here. For that reason, we prefer not to retain the traditional terminology. In the original Russian literature, including \cite{NU83}, one finds instead the word \ru{сетка} (plural \ru{сетки}), meaning ``grid'' or ``mesh''. Within the Nikiforov and Uvarov framework, this is arguably the more accurate term.}]\label{def:admissible-lattice}
Fix \(h\in\mathbb C^\times\), let \(U\subseteq\mathbb C\) be a half-step-invariant set, and let \(X:U\to\mathbb C\) be a parametrising map on \(U\). The map \(X\) is said to be \emph{admissible on \(U\)} if \(X(U)\) is infinite\footnote{This guarantees that the polynomial \(Q\) in \eqref{Q} is uniquely determined.} and, for every polynomial \(p\) of degree at least \(1\), there exists a polynomial \(Q\) of degree at most \(\deg p-1\) such that
\begin{equation}\label{Q}
\frac{p(Y)-p(Z)}{Y-Z}=Q(X)
\end{equation}
throughout \(U\).
\end{definition}

\begin{example}[Admissibility for the preceding parametrising maps]\label{ex:admissibility-checks}
The following examples determine which of the parametrising maps introduced in Example~\ref{ex:lattices-on-U} are admissible.
\begin{enumerate}
\item
In Example~\ref{ex:U-examples}(1), take \(h=2\), so that \(U=\mathbb Z\). Let \(X:U\to\mathbb C\) be given by
\[
X(s)=\left(\alpha-\frac12\right)s+\left(\beta-\frac12\right)\bigl(1-(-1)^s\bigr),
\]
for every \(s\in U\), where \(\alpha,\beta\in\mathbb C\) and \(\alpha\neq \tfrac12\). Then
\[
Y(s)=X(s+1),\quad Z(s)=X(s-1).
\]
Specialising \(p\) in \eqref{Q} to \(p(x)=x^2\), the divided difference reduces to \(Y+Z\), so there must exist an affine polynomial \(Q\) such that
\(
Q\bigl(X\bigr)=Y+Z.
\)
A direct computation shows that
\[
(Y+Z)(s)=2X(s)+(-1)^s(4\beta-2),
\]
for every \(s\in\mathbb Z\). Since \(\alpha\neq \frac12\), both \(X(2\mathbb Z)\) and \(X(2\mathbb Z+1)\) are infinite. Hence the affine polynomial \(Q\) satisfies
\[
Q(x)=2x\pm(4\beta-2)
\]
for infinitely many values of \(x\), with the sign \(+\) on \(X(2\mathbb Z)\) and the sign \(-\) on \(X(2\mathbb Z+1)\). Since both expressions define the same affine polynomial \(Q\), they must therefore agree identically. Thus
\(
4\beta-2=-(4\beta-2),
\)
and therefore \(\beta=\frac12\). It follows that admissibility forces \(\beta=\tfrac12\).
In that case,
\[
X(s)=\left(\alpha-\frac12\right)s,
\]
which is a non-constant affine map, and hence is admissible by the argument given in the next item.

\item
Let \(U\) be as in Example~\ref{ex:U-examples}(2). Let \(X:U\to\mathbb C\) be given by
\[
X(s)=cs+e,
\]
for every \(s\in U\), where \(c\in\mathbb C^\times\) and \(e\in\mathbb C\), as in Example~\ref{ex:lattices-on-U}(2). In this case \(X(U)\) is infinite, and \(X\) is admissible on \(U\).
Indeed,
\[
Y=X+c, \quad Z=X-c,
\]
throughout \(U\), so for every polynomial \(p\) of degree at least \(1\),
\[
\frac{p(X+c)-p(X-c)}{2c}=Q(X).
\]
To see that \(Q\) is a polynomial of degree at most \(\deg p-1\), it suffices by linearity to treat monomials. If \(p(x)=x^n\) with \(n\in\mathbb N^\times\), then
\[
\frac{(X+c)^n-(X-c)^n}{2c}
=
\sum_{\substack{j=0\\ j\ \mathrm{odd}}}^{n}
\binom{n}{j}X^{\,n-j}c^{\,j-1},
\]
which is a polynomial in \(X\) of degree \(n-1\). Hence \eqref{Q} holds throughout \(U\). Since \(c\neq0\), the image \(X(U)\) contains the infinite arithmetic progression \(X(v+\tfrac h2\mathbb Z)\) for any \(v\in U\), and is therefore infinite. The same argument, with \(c\) replaced by \(c\,h/2\), shows that every non-constant affine map is admissible on any half-step-invariant domain \(U\). In particular, in the situation discussed in the introduction,
\[
\mathbb S=\mathbb Z\cup\left(\frac{\gamma}{2}+\mathbb Z\right)
\]
is a half-step-invariant set. Therefore \(X(s)=2s\) is admissible on the whole of \(\mathbb S\), that is, simultaneously on both arithmetic progressions that make up the domain. From this point of view, the sum in \eqref{eq:para-k-functional} is simply the sum written over the finite portions of those two arithmetic progressions, whose ambient progressions are represented schematically by the dashed tracks in Figure~\ref{fig:two-Z-orbits-h2}. It is here that the structural issue must be stated with some care. Should the aim be merely to recover the set displayed in \eqref{eq:bi-lattice-image}, one cannot---within the constraints of admissibility, and hence within any theory that requires it---begin with \(U=\mathbb Z\) and introduce \eqref{eq:bi-lattice-map} under the guise of a ``bi-lattice''. The correct way to formulate the situation is to start from a half-step-invariant set \(U\) already decomposed into distinct cosets, as in Figure~\ref{fig:strip-discrete-E}, and then map those cosets under the affine transformation. Accordingly, as already observed in the introduction, the relevant decomposition takes place at the level of the half-step-invariant set \(U\), which splits into a disjoint union of cosets, and not at the level of the parametrising map, whose form may suggest an interlacing that is not intrinsic to the admissible map structure, as we saw in the preceding example.

\item
Let \(U\) be as in Example~\ref{ex:U-examples}(3), assume that \(0<q<1\) and \(q^{\,h}\neq1\), and let \(X:U\to\mathbb C\) be given by
\[
X(s)=q^{\,s},
\]
for every \(s\in U\), as in Example~\ref{ex:lattices-on-U}(3). Set
\(
a=q^{h/2}.
\)
Then
\[
Y=aX,\quad Z=a^{-1}X,
\]
throughout \(U\). Since \(Y-Z=(a-a^{-1})X\), it follows that
\[
\frac{p(aX)-p(a^{-1}X)}{(a-a^{-1})X}
=Q(X).
\]
As in the preceding item, to see that the resulting expression is in fact a polynomial in \(X\) of degree at most \(\deg p-1\), it suffices by linearity to treat monomials. If \(p(x)=x^n\) with \(n\in\mathbb N^\times\), then pointwise on \(U\) one has
\[
\frac{p(aX)-p(a^{-1}X)}{(a-a^{-1})X}
=
\frac{a^n-a^{-n}}{a-a^{-1}}\,X^{n-1}.
\]
Thus, although the quotient is first written with a denominator, it coincides on
\(U\) with a genuine polynomial in the indeterminate \(X\), of degree at most
\(n-1\); its degree is exactly \(n-1\) unless \(a^{2n}=1\). Moreover, since
\(a^2=q^{\,h}\neq1\), we have \(a\neq a^{-1}\), and since
\(X(s)=q^{\,s}\neq0\) for every \(s\in U\), the denominator in the original pointwise expression never vanishes. Hence \eqref{Q} holds throughout \(U\). Consequently, \(X\) is admissible on \(U\) exactly when the remaining requirement in Definition~\ref{def:admissible-lattice}, namely the infinitude of \(X(U)\), is also satisfied. Since
\[
U=\bigsqcup_{v\in E}\left(v+\frac h2\mathbb Z\right),
\]
it follows that
\[
X\!\left(v+\frac h2\mathbb Z\right)=q^{\,v}a^{\mathbb Z},
\]
and therefore
\[
X(U)=\bigcup_{v\in E} q^{\,v}a^{\mathbb Z}.
\]
Thus the infinitude of \(X(U)\) depends on the combined effect of the collection of cosets contained in \(U\), and not merely on the behaviour along a single coset. In particular, if \(a\) is not a root of unity, then each set \(q^{\,v}a^{\mathbb Z}\) is infinite, and hence \(X(U)\) is infinite.

\item
Let \(U\) be any half-step-invariant set with \(-\tfrac18\notin U\), as in Example~\ref{ex:U-examples}(4). Let \(X:U\to\mathbb C\) be given by
\[
X(s)=s^2+\frac14\,s,
\]
for every \(s\in U\), as in Example~\ref{ex:lattices-on-U}(4). A direct computation yields
\begin{align*}
Y+Z=2X+\frac{h^2}{2},\quad
YZ=X^2-\frac{h^2}{2}\,X+\frac{h^2(4h^2-1)}{64},
\end{align*}
throughout \(U\), so that \(Y+Z\) and \(YZ\) are polynomials in \(X\). Again, it suffices by linearity to consider a monomial \(p(x)=x^n\) with \(n\in\mathbb N^\times\). One then has
\[
\frac{Y^n-Z^n}{Y-Z}
=\sum_{k=0}^{n-1}Y^{n-1-k}Z^k,
\]
which is symmetric in \(Y\) and \(Z\). It is therefore a polynomial in the elementary symmetric functions \(Y+Z\) and \(YZ\). More precisely, the above polynomial is symmetric in \(Y\) and \(Z\), homogeneous of total degree \(n-1\), and can therefore be written as a linear combination of monomials of the form
\[
(Y+Z)^{\,n-1-2j}(YZ)^j,
\]
with \(j\) ranging over the integers satisfying \(0\le j\le \lfloor\tfrac{n-1}{2}\rfloor\). Now \(Y+Z\) is affine in \(X\), whereas \(YZ\) is quadratic in \(X\). Hence each such monomial has degree at most
\(
(n-1-2j)+2j=n-1
\)
as a polynomial in \(X\). Hence \eqref{Q} holds throughout \(U\). It remains to verify that \(X(U)\) is infinite. Since \(U\) is non-empty and half-step-invariant, for every \(v\in U\) it contains the whole coset
\(
v+\tfrac h2\mathbb Z.
\)
On that coset one has
\[
X\!\left(v+k\frac h2\right)
=
\left(v+k\frac h2\right)^2+\frac14\left(v+k\frac h2\right),
\]
which is a non-constant quadratic polynomial in \(k\in\mathbb Z\), with leading coefficient \(h^2/4\neq0\). Its set of values is therefore infinite, and hence so is \(X(U)\). Thus \(X\) is admissible on \(U\). In particular, \(X\) is admissible on the half-step-invariant periodic array of rectangles whose visible portion is shown on the right-hand side of Figure~\ref{fig:exp-preimage-sector}.

\end{enumerate}
\end{example}

The next example treats an important exponential case that will play a central role below.

\begin{example}
Let \(q\in\mathbb C^\times\). Assume that either
\[
q=e^{2\pi iM/\nu},
\quad
1\le M<\nu,
\quad
\gcd(M,\nu)=1,
\quad
\nu\ge3,
\]
or else
\[
|q|\neq1.
\]
Fix a determination of \(\log q\), let \(h\in\mathbb C^\times\), and let \(U\subseteq\mathbb C\) be a half-step-invariant set. Let \(X:U\to\mathbb C\) be given by
\[
X(s)=\frac12\bigl(q^{\,s}+q^{-s}\bigr),
\quad
q^{\,s}=e^{s\log q},
\]
for every \(s\in U\). Set
\(
a=q^{h/2}.
\)
Then
\[
Y(s)=\frac12\bigl(a\,q^{\,s}+a^{-1}q^{-s}\bigr),
\quad
Z(s)=\frac12\bigl(a^{-1}q^{\,s}+a\,q^{-s}\bigr),
\]
and a direct computation gives
\begin{equation}\label{eq:qcosh-YminusZ}
(Y-Z)(s)=\frac12\,(a-a^{-1})\bigl(q^{\,s}-q^{-s}\bigr).
\end{equation}
Set
\[
K=\frac{\pi i}{\log q}\,\mathbb Z.
\]
Since \(a^2=q^{\,h}\), it follows from \eqref{eq:qcosh-YminusZ} that \(X\) is a parametrising map on \(U\) if and only if
\[
q^{\,h}\neq1,
\quad
U\cap K=\varnothing.
\]
Equivalently, since \(q^{2s}=1\) if and only if \(q^{\,s}=\pm1\), and
\[
X(s)=\frac12\bigl(q^{\,s}+q^{-s}\bigr)=\pm1
\]
holds if and only if \(q^{\,s}=\pm1\), this condition may be written as
\[
q^{\,h}\neq1,
\quad
X(s)\neq\pm1
\]
for all \(s\in U\). Assume henceforth that these conditions are satisfied. Then \(Y+Z\) and \(YZ\) are polynomial functions of \(X\); indeed,
\begin{align*}
Y+Z=(a+a^{-1})\,X,\quad
YZ=X^2+\frac{(a-a^{-1})^2}{4}.
\end{align*}
Consequently, the same argument as in Example~\ref{ex:admissibility-checks}\emph{(4)} shows that \(X\) is admissible on \(U\) whenever \(X(U)\) is infinite. We next relate this criterion to the sets introduced in Example~\ref{ex:U-examples}.

\begin{enumerate}
\item
In Example~\ref{ex:U-examples}(1), the coset
\[
U=\frac h2\mathbb Z
\]
is not suitable in the present setting, since \(0\in U\cap K\), and hence \(X(0)=1\). One therefore replaces \(U\) by a translated coset
\[
U_\varepsilon=\varepsilon+\frac h2\mathbb Z,
\]
where \(\varepsilon\notin K-\tfrac h2\mathbb Z\). Since \(K-\tfrac h2\mathbb Z\) is countable, such a choice of \(\varepsilon\) always exists.

\item
In Example~\ref{ex:U-examples}(2), one argues similarly: it suffices to choose representatives of two distinct cosets such that neither of the corresponding cosets meets \(K\).

\item
In Example~\ref{ex:U-examples}(3), define
\begin{align*}
E'=\left\{v\in E:\ \left(v+\frac h2\mathbb Z\right)\cap K=\varnothing\right\},
\quad
U'=\bigcup_{v\in E'}\left(v+\frac h2\mathbb Z\right).
\end{align*}
Then \(X(s)\neq\pm1\) for all \(s\in U'\). Moreover, \(U'\neq\varnothing\) whenever
\[
E\not\subseteq K-\frac h2\mathbb Z;
\]
in particular, this holds whenever \(E\) is uncountable. In the situation represented in Figure~\ref{fig:strip-discrete-E}, passing from \(E\) to \(E'\) amounts simply to discarding those representatives whose cosets meet \(K\).

\item
In Example~\ref{ex:U-examples}(4), with \(h=1\) and \(E=A\), the original set meets \(K\) at \(0\). Indeed, one has \(0\in K\), and also \(0\in U\), since \(e^{4\pi i\cdot 0}=1\in A\). Thus the corresponding map is not a parametrising map on that original set. One could again restrict to those cosets that avoid \(K\), but that configuration will play no role in what follows.

\end{enumerate}

For admissibility, it remains only to ensure that the image of the chosen set is infinite. To simplify notation, write again \(U\) for any one of the modified sets arising from
Example~\ref{ex:U-examples}\emph{(1)}--\emph{(3)}. Since \(U\) is a union of cosets of \(\tfrac h2\mathbb Z\), we have
\[
X(U)=\bigcup_{v\in V}X\!\left(v+\frac h2\mathbb Z\right),
\]
where \(V\subseteq U\) is a set of representatives of the cosets contained in \(U\). For each such \(v\),
\[
X\!\left(v+k\frac h2\right)
=
\frac12\left(a^k q^{\,v}+a^{-k}q^{-v}\right),
\]
for every \(k\in\mathbb Z\). If \(a=q^{h/2}\) is not a root of unity, then \(X\!\left(v+\tfrac h2\mathbb Z\right)\) is infinite for every \(v\). Indeed, if
\[
\frac12\left(a^k q^{\,v}+a^{-k}q^{-v}\right)=\zeta,
\]
then, upon setting \(x=a^k\), one obtains the quadratic equation
\[
q^{\,v}x^2-2\zeta x+q^{-v}=0.
\]
Thus, for every fixed \(\zeta\), the quantity \(x=a^k\) can take at most two values. Since \(a\) is not a root of unity, the map \(k\mapsto a^k\) is injective on \(\mathbb Z\), and therefore only finitely many integers \(k\) can give the same value \(\zeta\). It follows that \(X\!\left(v+\tfrac h2\mathbb Z\right)\) is infinite. Consequently, in the situations arising from Examples~\ref{ex:U-examples}(1) and~\ref{ex:U-examples}(2), the corresponding modified domains yield admissible maps whenever \(q^{\,h}\neq1\) and \(a=q^{h/2}\) is not a root of unity. The same conclusion holds for the restricted set arising from Example~\ref{ex:U-examples}(3), provided \(E'\neq\varnothing\). When \(a=q^{h/2}\) is a root of unity, the image of each individual coset
\(
v+\tfrac h2\mathbb Z
\)
is finite, since, for every \(k\in\mathbb Z\),
\[
X\!\left(v+k\frac h2\right)
=
\frac12\left(a^k q^{\,v}+a^{-k}q^{-v}\right)
\]
depends only on the residue class of \(k\) modulo the order of \(a\). This does not, however, preclude admissibility. What matters is the behaviour of the full set \(U\), not merely that of a single arithmetic progression. In that case,
\[
X(U)=\bigcup_{v\in V}X\!\left(v+\frac h2\mathbb Z\right),
\]
and \(X(U)\) is infinite if and only if infinitely many of the finite sets
\(
X\!(v+\tfrac h2\mathbb Z),
\)
are distinct. This possibility is not merely formal, and it already occurs in the setting of Example~\ref{ex:U-examples}\emph{(3)}. Indeed, take \(h=1\), let
\(
q=e^{2\pi i/3},
\)
and choose
\(
\log q=2\pi i/3.
\)
Then
\(
a=q^{1/2}=e^{\pi i/3}
\)
is a root of unity, whereas
\(
q^h=q\neq1.
\)
Now let
\[
E=\{iy:\ y\in\mathbb R\}\subset
\left\{x+iy:\ 0\le x<\frac12,\ y\in\mathbb R\right\},
\]
and define
\[
U=\bigcup_{y\in\mathbb R}\left(iy+\frac12\mathbb Z\right).
\]
Since
\[
K=\frac32\,\mathbb Z\subset\mathbb R,
\]
every coset \(iy+\tfrac12\mathbb Z\) with \(y\neq0\) avoids \(K\). After removing the single bad coset \(\tfrac12\mathbb Z\), one obtains a half-step-invariant set \(U'\) on which \(X\) is parametrising and which still contains infinitely many cosets. For each fixed \(y\neq0\), the set
\(
X\!(iy+\frac12\mathbb Z)
\)
is finite, since \(a\) has finite order. On the other hand,
\[
X(iy)=\frac12\left(q^{iy}+q^{-iy}\right)
=\cosh\!\left(\frac{2\pi y}{3}\right),
\]
so \(X(U')\) is infinite. Thus the torsion case already exhibits the basic phenomenon that will be important later: although the image of each individual coset may be finite, the full image may still be infinite because infinitely many distinct cosets contribute distinct finite pieces.
\end{example}

One may ask from where the map
\[
X(s)=\frac12\bigl(q^s+q^{-s}\bigr)
\]
really arises. Its most familiar classical antecedent is the cosine variable
\[
\cos \theta=\frac12\bigl(e^{i\theta}+e^{-i\theta}\bigr),
\]
which already lies behind the Chebyshev polynomials and, more generally, behind the passage from an exponential parametrisation to a variable symmetric under inversion. From this point of view, the map \(X\) is more fundamental than any particular family in which it later came to prominence, including the Askey--Wilson polynomials; see \cite{AW85} and the references therein. For the purposes of the present paper, the historical question, whether for this map or for any other, is secondary. Our concern is not to trace the first appearance of a particular parametrisation, but to explain why such forms arise at all. What matters is that, once admissibility is imposed, the form above emerges as one of the possibilities singled out by the Nikiforov--Uvarov mechanism itself; this is precisely the content of the next section.

\section{The form of the admissible maps}
The content of this section marks our first contact with the material of the
1983 preprint of Nikiforov and Uvarov \cite{NU83}, closely related to the
treatment in their 1984 book \cite{NU84}\footnote{Although the later book written with Suslov \cite{NSU91} gives a refined and very useful exposition of this material, we shall not need to rely on it here. The recurrence relation used below already appears as equation~(39) of \cite{NU83}; \cite{NU84} contains the closely related published treatment.}. Their starting point is the discretisation of the Bochner differential equation: derivatives are replaced by symmetric finite differences, yielding a second-order difference equation that serves both as a discrete analogue of the continuous equation and, in the appropriate regime, as a second-order approximation to it. From the present perspective, that choice should be viewed not as a substantive restriction but rather as an expository normalisation naturally adapted to their way of understanding orthogonality. Once this discretisation has been adopted, Nikiforov and Uvarov ask a structural question: for which underlying discrete domains does the resulting equation retain the expected hypergeometric features, in the sense that the relevant operators preserve polynomiality and the coefficient functions retain the appropriate algebraic form?
The following theorem gives the classification of admissible maps, placing the corresponding classification of Nikiforov and Uvarov in the present framework.

\begin{theorem}\label{thm:admissible-lattice-classification-U}
Fix \(h\in\mathbb C^\times\), let \(U\subseteq\mathbb C\) be half-step-invariant, and let \(X:U\to\mathbb C\) be an admissible map. Then there exist constants \(A,B\in\mathbb C\) such that
\begin{equation}\label{eq:global-full-step-relation}
Y_1+Z_1=A\,X+B
\end{equation}
throughout \(U\). In particular, on every full-step arithmetic progression
\(s_0+h\mathbb Z\subseteq U\), the values of \(X\) satisfy a
second-order linear recurrence with constant coefficients, and are therefore
determined along that progression by the pair \(X(s_0)\) and \(Y_1(s_0)\). Moreover, exactly one of the following cases holds.
\begin{itemize}
\item[\emph{(i)}] \(A=2\).
For every \(s_0\in U\) and every \(k\in\mathbb Z\),
\[
X(s_0+kh)=a\,k^2+ b_{s_0}\,k+c_{s_0},
\]
where
\[
a=\frac{B}{2},\quad b_{s_0}=Y_1(s_0)-X(s_0)-\frac{B}{2}, \quad c_{s_0}=X(s_0).
\]

\item[\emph{(ii)}] \(A=-2\). Admissibility forces the additional first-order relation
\begin{equation}\label{eq:A=-2-first-order}
Y_1+X=\frac{B}{2}
\end{equation}
throughout \(U\), and consequently, for every \(s_0\in U\) and every \(k\in\mathbb Z\),
\[
X(s_0+k h)=(-1)^k\, a_{s_0}+b,
\]
where
\[
a_{s_0}=X(s_0)-\frac{B}{4},\quad b=\frac{B}{4}.
\]

\item[\emph{(iii)}] \(A^2\neq4\).
Fix \(q\in\mathbb C^\times\) with \(q\neq\pm1\) such that \(A=q+q^{-1}\), and set
\[
c=\frac{B}{2-A}.
\]
Then, for every \(s_0\in U\) and every \(k\in\mathbb Z\),
\[
X(s_0+kh)=a_{s_0}\,q^{k}+b_{s_0}\,q^{-k}+c,
\]
where \(a_{s_0}\) and \(b_{s_0}\) are given by
\begin{align*}
a_{s_0}=\frac{Y_1(s_0)-c-(X(s_0)-c)\,q^{-1}}{q-q^{-1}},\quad
b_{s_0}=\frac{(X(s_0)-c)\,q-Y_1(s_0)+c}{q-q^{-1}}.
\end{align*}
\end{itemize}
In particular, along each full-step arithmetic progression \(s_0+h\mathbb Z\subseteq U\), the map \(X\) is necessarily of quadratic type \emph{(i)}, alternating type \emph{(ii)}, or \(q\)-exponential type \emph{(iii)}. No other dependence of \(X\) on \(s\) along a full-step arithmetic progression can be compatible with admissibility. The constants \(A\) and \(B\) are global, that is, independent of the choice of \(s_0\). By contrast, the coefficients not already fixed globally by \(A\) and \(B\) are determined by the initial values along the chosen arithmetic progression and may therefore vary from one arithmetic progression to another.
\end{theorem}

The following remark explains the relation with the original Nikiforov--Uvarov recurrence.

\begin{remark}[Relation with Nikiforov--Uvarov, 1983]\label{rem:NU2order}
The relation \eqref{eq:global-full-step-relation} becomes, upon restriction to a full-step arithmetic progression \(s_0+h\mathbb Z\subseteq U\), the constant-coefficient recurrence considered by Nikiforov and Uvarov in \cite[(39)]{NU83}. Fix \(s_0\in U\) and restrict \(X\) to the full-step arithmetic progression \(s_0+h\mathbb Z\subseteq U\). If one sets, for every \(k\in\mathbb Z\),
\[
x_k=X(s_0+kh),
\]
then \eqref{eq:global-full-step-relation} becomes
\[
x_{k+1}+x_{k-1}=A\,x_k+B,
\]
which is exactly the recurrence appearing in \cite[(39)]{NU83}.
\end{remark}

The following example illustrates the alternating case in Theorem~\ref{thm:admissible-lattice-classification-U}.

\begin{example}[Normalised alternating case]\label{ex:A=-2-minus-one-power}
Let \(U\subseteq\mathbb C\) be the dense half-step-invariant set from Example~\ref{ex:U-examples}\emph{(3)} with \(h=1\), namely
\[
U=\{p+qi:\ p,q\in\mathbb Q\}\subset\mathbb C.
\]
Let \(X:U\to\mathbb C\) be given by
\[
X(s)=e^{\pi i s},
\]
for every \(s\in U\). Equivalently, one may write \(X(s)=(-1)^s\), provided the convention
\[
(-1)^s=e^{\pi i s}
\]
is understood. For each fixed \(s_0\in U\), the restriction of \(X\) to the full-step arithmetic progression \(s_0+\mathbb Z\subseteq U\) is given by
\[
X(s_0+k)=e^{\pi i(s_0+k)}=(-1)^kX(s_0),
\]
so the dependence on the chosen arithmetic progression enters only through the initial value \(X(s_0)=e^{\pi i s_0}\). Moreover, for every \(s\in U\),
\begin{align*}
Y(s)&=i\,e^{\pi i s}=i\,X(s),\quad \,\,\,\,\,
Z(s)=-i\,e^{\pi i s}=-i\,X(s),\\[7pt]
Y_1(s)&=-e^{\pi i s}=-X(s),\quad
Z_1(s)=-e^{\pi i s}=-X(s).
\end{align*}
Since \(Y=iX\) and \(Z=-iX\), one has \(Y-Z=2iX\neq0\) throughout \(U\). Moreover, writing
\[
p(x)=a(-x^2)+x\,b(-x^2),
\]
for polynomials \(a\) and \(b\), one finds
\[
\frac{p(Y)-p(Z)}{Y-Z}=b(X^2),
\]
so \eqref{Q} holds with a polynomial in \(X\) of degree at most \(\deg p-1\). Since \(X(U)\) is infinite, it follows that \(X\) is admissible on \(U\). Furthermore, \eqref{eq:global-full-step-relation} and \eqref{eq:A=-2-first-order} hold with \(A=-2\) and \(B=0\). Therefore the present map is a normalised realisation of case~\emph{(ii)} of Theorem~\ref{thm:admissible-lattice-classification-U}.

Why, then, does a single progression fail to display the admissible map,
whereas the ambient set succeeds? The mechanism becomes completely visible
once we follow a few progressions under \(X\). Take
\[
v_1=\frac18,
\quad
v_2=\frac14+\frac i4,
\quad
v_3=\frac38-\frac i6.
\]
For every \(j\in\{1,2,3\}\) and \(n\in\mathbb Z\),
\[
X\!\left(v_j+\frac n2\right)=i^nX(v_j).
\]
Thus each individual half-step arithmetic progression still produces only
four values. What varies from one progression to another is the initial value
\(X(v_j)\). Figure~\ref{fig:strip-discrete-E2} records this distinction
directly. The finite display is not intended to replace the density argument
given above; it isolates the mechanism that a restriction to one progression
necessarily conceals.

\begin{figure}[!htbp]
\centering

\begin{minipage}[c]{0.40\textwidth}
\centering
\resizebox{\linewidth}{!}{%
\begin{tikzpicture}[>=latex,x=2.4cm,y=2.4cm]

  \coordinate (v1) at ({1/8},0);
  \coordinate (v2) at ({1/4},{1/4});
  \coordinate (v3) at ({3/8},{-1/6});
  \coordinate (hh) at (0.5,0);

  \path[fill=gray!10]
    (0,-0.44) rectangle (0.5,0.44);
  \draw[gray!70,line width=0.40pt]
    (0,-0.44) -- (0,0.44);
  \draw[gray!70,line width=0.40pt,dashed]
    (0.5,-0.44) -- (0.5,0.44);

  \draw[->,gray!68,line width=0.35pt] (-1.05,0) -- (1.55,0)
    node[right,scale=0.72] {$\Re s$};
  \draw[->,gray!68,line width=0.35pt] (0,-0.48) -- (0,0.48)
    node[above,scale=0.72] {$\Im s$};

  \foreach \v in {v2,v3}{
    \draw[gray!55,line width=0.32pt,
      dash pattern=on 1.8pt off 1.4pt]
      ($(\v)-2.25*(hh)$) -- ($(\v)+2.25*(hh)$);
  }
  \foreach \v in {v1,v2,v3}{
    \foreach \n in {-2,-1,0,1,2}{
      \fill[black] ($(\v)+\n*(hh)$) circle (1.0pt);
    }
    \filldraw[black]
      ($(\v)+(-0.018,-0.018)$)
      rectangle
      ($(\v)+(0.018,0.018)$);
  }

  \node[above,yshift=2pt] at (v1) {$v_1$};
  \node[above,yshift=2pt] at (v2) {$v_2$};
  \node[below,yshift=-2pt] at (v3) {$v_3$};
  \node at (0.25,-0.39) {$S$};
  \node at (0.88,0.35) {$U$};

\end{tikzpicture}%
}
\end{minipage}\hfill
\begin{minipage}[c]{0.14\textwidth}
\centering
\[
\xrightarrow{\;X\;}
\]
\end{minipage}
\hfill
\begin{minipage}[c]{0.40\textwidth}
\centering
\resizebox{\linewidth}{!}{%
\begin{tikzpicture}[>=latex,x=1.7cm,y=1.7cm]

  \pgfmathsetmacro{\rOne}{1}
  \pgfmathsetmacro{\rTwo}{exp(-pi/4)}
  \pgfmathsetmacro{\rThree}{exp(pi/6)}

  \tikzset{
    imagebase/.style={
      rectangle,
      draw=black,
      fill=black,
      inner sep=0pt,
      minimum size=2.7pt
    }
  }

  \draw[->,gray!68,line width=0.35pt] (-1.92,0) -- (1.92,0)
    node[right,scale=0.72] {$\Re\!\bigl(X(s)\bigr)$};
  \draw[->,gray!68,line width=0.35pt] (0,-1.92) -- (0,1.92)
    node[above,scale=0.72] {$\Im\!\bigl(X(s)\bigr)$};

  \newcommand{\DrawFourCycle}[2]{%
    \draw[gray!55,line width=0.32pt,
      dash pattern=on 1.8pt off 1.4pt] (0,0) circle (#1);
    \foreach \j in {0,1,2,3}{
      \fill[black] ({#2+90*\j}:#1) circle (1.05pt);
    }
    \node[imagebase] at (#2:#1) {};
  }

  \DrawFourCycle{\rOne}{22.5}
  \DrawFourCycle{\rTwo}{45}
  \DrawFourCycle{\rThree}{67.5}

  \draw[->,gray!68,line width=0.40pt]
    (5:1.28) arc[start angle=5,end angle=95,radius=1.28];
  \node[scale=0.78] at (50:1.48) {$\times i$};

\end{tikzpicture}%
}
\end{minipage}

\caption[Half-step progressions and their four-cycles]
{The three selected representatives
\(v_1=\tfrac18\), \(v_2=\tfrac14+\tfrac i4\), and
\(v_3=\tfrac38-\tfrac i6\), together with portions of their half-step
arithmetic progressions. Under \(X(s)=e^{\pi i s}\), these progressions
become four-cycles with initial polar data
\((1,\pi/8)\), \((e^{-\pi/4},\pi/4)\), and
\((e^{\pi/6},3\pi/8)\), respectively. The squares mark the images of
the selected representatives. The dashed circles are polar guides only.}
\label{fig:strip-discrete-E2}
\end{figure}

\end{example}

The next example illustrates the same alternating mechanism on a non-discrete half-step-invariant set.

\begin{example}[Normalised alternating case revisited]\label{ex:rectangular-U-minus-one-power}
Let \(U\subseteq\mathbb C\) be the half-step-invariant set from Example~\ref{ex:U-examples}\emph{(4)} with \(h=1\), namely
\[
U
=
\bigsqcup_{k\in\mathbb Z}
\left(
\left(
\frac{k}{2}-\frac{1}{16},
\frac{k}{2}+\frac{1}{16}
\right)
\times
\left(
-\frac{\ln 2}{4\pi},
\frac{\ln 2}{4\pi}
\right)
\right)
\subset\mathbb C,
\]
where we identify \((x,y)\) with \(x+iy\). Let \(X:U\to\mathbb C\) be given by
\[
X(s)=e^{\pi i s},
\]
for every \(s\in U\). As in Example~\ref{ex:A=-2-minus-one-power}, one has
\[
Y(s)=iX(s),\quad Z(s)=-iX(s),
\]
for every \(s\in U\), and the same decomposition argument shows that \eqref{Q} holds with a polynomial in \(X\) of degree at most \(\deg p-1\). Thus, provided \(X(U)\) is infinite, the map \(X\) is admissible on \(U\). This remaining condition is indeed satisfied. For
\[
y\in\left(-\frac{\ln 2}{4\pi},\,\frac{\ln 2}{4\pi}\right)
\]
one has \(iy\in U\), and
\[
X(iy)=e^{\pi i(iy)}=e^{-\pi y}\in\bigl(2^{-1/4},\,2^{1/4}\bigr),
\]
so \(X(U)\) contains an open real interval.

The rectangular example permits us to display not merely a schematic
traversal, but the whole image. Put
\[
R_0=
\left(-\frac1{16},\frac1{16}\right)
+i\left(-\frac{\ln2}{4\pi},\frac{\ln2}{4\pi}\right),
\quad
R_k=R_0+\frac k2,
\]
and
\[
A_0=
\left\{
w\in\mathbb C:
2^{-1/4}<|w|<2^{1/4},
\quad
|\Arg w|<\frac{\pi}{16}
\right\}.
\]
Since
\[
X\!\left(s+\frac k2\right)=i^kX(s),
\]
we have
\[
X(R_k)=i^kA_0,
\quad
X(U)=\bigcup_{j=0}^{3}i^jA_0.
\]
Thus the four-valued behaviour on an individual half-step arithmetic
progression and the infinitude of the total image are not competing facts:
the former records the action of translation by \(\tfrac12\), while the latter
comes from the continuum of points already present in one rectangle.
Figure~\ref{fig:rectangular-U-orbits} makes both facts visible.

\begin{figure}[!htbp]
\centering

\begin{minipage}[c]{0.40\textwidth}
\centering
\resizebox{\linewidth}{!}{%
\begin{tikzpicture}[>=latex,x=3.2cm,y=3.2cm]

  \pgfmathsetmacro{\dx}{1/16}
  \pgfmathsetmacro{\dy}{ln(2)/(4*pi)}
  \pgfmathsetmacro{\yp}{ln(2)/(8*pi)}
  \pgfmathsetmacro{\ym}{-ln(2)/(8*pi)}

  \foreach \k in {0,1,2,3}{
    \pgfmathsetmacro{\c}{\k/2}

    \path[fill=gray!12,fill opacity=0.66]
      (\c-\dx,-\dy) rectangle (\c+\dx,\dy);
  }

  \draw[->,gray!68,line width=0.35pt] (-0.14,0) -- (1.68,0)
    node[right,scale=0.72] {$\Re s$};
  \draw[->,gray!68,line width=0.35pt] (0,-0.145) -- (0,0.145)
    node[above,scale=0.72] {$\Im s$};
  \foreach \k in {0,1,2,3}{
    \pgfmathsetmacro{\c}{\k/2}
    \draw[gray!72,line width=0.35pt,
      dash pattern=on 1.8pt off 1.4pt]
      (\c-\dx,-\dy) rectangle (\c+\dx,\dy);
  }

  \draw[gray!52,line width=0.30pt,
    dash pattern=on 1.8pt off 1.4pt] (-0.10,\yp) -- (1.62,\yp);
  \draw[gray!52,line width=0.30pt,
    dash pattern=on 1.8pt off 1.4pt] (-0.10,\ym) -- (1.62,\ym);

  \draw[->,gray!68,line width=0.40pt]
    (0,0.105) -- (0.5,0.105)
    node[midway,above,scale=0.78] {$+\tfrac12$};

  \foreach \k in {0,1,2,3}{
    \pgfmathsetmacro{\c}{\k/2}
    \fill[black] (\c,\yp) circle (1.0pt);
    \filldraw[black]
      (\c-0.012,\ym-0.012)
      rectangle
      (\c+0.012,\ym+0.012);
  }

  \node at (0.75,-0.125) {$U$};

\end{tikzpicture}%
}
\end{minipage}\hfill
\begin{minipage}[c]{0.14\textwidth}
\centering
\[
\xrightarrow{\;X\;}
\]
\end{minipage}
\hfill
\begin{minipage}[c]{0.40\textwidth}
\centering
\resizebox{\linewidth}{!}{%
\begin{tikzpicture}[>=latex,x=1.8cm,y=1.8cm]

  \pgfmathsetmacro{\rin}{pow(2,-1/4)}
  \pgfmathsetmacro{\rout}{pow(2,1/4)}
  \pgfmathsetmacro{\rplus}{pow(2,-1/8)}
  \pgfmathsetmacro{\rminus}{pow(2,1/8)}
  \pgfmathsetmacro{\halfangle}{180/16}

  \tikzset{
    sector/.style={
      fill=gray!12,
      fill opacity=0.66,
      draw=gray!72,
      line width=0.35pt,
      dash pattern=on 1.8pt off 1.4pt,
      line cap=round,
      line join=round
    },
    imagebase/.style={
      rectangle,
      draw=black,
      fill=black,
      inner sep=0pt,
      minimum size=2.7pt
    }
  }

  \foreach \a in {0,90,180,270}{
    \path[fill=gray!12,fill opacity=0.66,draw=none]
      ({\a-\halfangle}:\rin)
      --
      ({\a-\halfangle}:\rout)
      arc[
        start angle={\a-\halfangle},
        end angle={\a+\halfangle},
        radius=\rout
      ]
      --
      ({\a+\halfangle}:\rin)
      arc[
        start angle={\a+\halfangle},
        end angle={\a-\halfangle},
        radius=\rin
      ]
      -- cycle;
  }

  \draw[->,gray!68,line width=0.35pt] (-1.55,0) -- (1.55,0)
    node[right,scale=0.72] {$\Re\!\bigl(X(s)\bigr)$};
  \draw[->,gray!68,line width=0.35pt] (0,-1.55) -- (0,1.55)
    node[above,scale=0.72] {$\Im\!\bigl(X(s)\bigr)$};
  \foreach \a in {0,90,180,270}{
    \draw[gray!72,line width=0.35pt,
      dash pattern=on 1.8pt off 1.4pt,line cap=round,line join=round]
      ({\a-\halfangle}:\rin)
      --
      ({\a-\halfangle}:\rout)
      arc[
        start angle={\a-\halfangle},
        end angle={\a+\halfangle},
        radius=\rout
      ]
      --
      ({\a+\halfangle}:\rin)
      arc[
        start angle={\a+\halfangle},
        end angle={\a-\halfangle},
        radius=\rin
      ]
      -- cycle;
    \draw[gray!68,line width=0.35pt]
      ({\a-\halfangle}:\rplus)
      arc[
        start angle={\a-\halfangle},
        end angle={\a+\halfangle},
        radius=\rplus
      ];
    \draw[gray!68,line width=0.35pt]
      ({\a-\halfangle}:\rminus)
      arc[
        start angle={\a-\halfangle},
        end angle={\a+\halfangle},
        radius=\rminus
      ];
    \fill[black] (\a:\rplus) circle (1.0pt);
    \node[imagebase] at (\a:\rminus) {};
  }

  \draw[->,gray!68,line width=0.40pt]
    (5:1.38) arc[start angle=5,end angle=95,radius=1.38];
  \node[scale=0.78] at (50:1.53) {$\times i$};

  \node[scale=0.78] at (225:1.42) {$X(U)$};

\end{tikzpicture}%
}
\end{minipage}

\caption[The rectangular domain and its alternating image]
{Four consecutive rectangles in \(U\) and their images under
\(X(s)=e^{\pi i s}\). Translation by \(\tfrac12\) becomes multiplication
by \(i\), and \(X(U)=\bigcup_{j=0}^{3}i^jA_0\), where
\(A_0=\{w:2^{-1/4}<|w|<2^{1/4},\ |\Arg w|<\pi/16\}\).
The discs and squares correspond respectively to the levels
\(y_+=\ln2/(8\pi)\) and \(y_-=-\ln2/(8\pi)\).}
\label{fig:rectangular-U-orbits}
\end{figure}
\end{example}

\begin{proof}[Proof of Theorem~\ref{thm:admissible-lattice-classification-U}]

\medskip
\noindent{\em Step 1:}
By admissibility, for every polynomial \(p\) of degree at least \(1\), there exists a unique polynomial
\(Q\), with \(\deg Q\le \deg p-1\), such that
\[
\frac{p(Y)-p(Z)}{Y-Z}=Q(X)
\]
throughout \(U\). Since \(X\) is a parametrising map, \(Y\neq Z\) throughout \(U\). Taking \(p(x)=x^2\), we obtain
\begin{equation}\label{eq:halfstep-affine}
Y+Z=\alpha X+\beta
\end{equation}
for suitable \(\alpha,\beta\in\mathbb C\).

\medskip
\noindent{\em Step 2:}
Evaluating the identity \eqref{eq:halfstep-affine} at \(s+\tfrac h2\) and \(s-\tfrac h2\), we obtain
\begin{align*}
X+Y_1=\alpha Y+\beta,\quad
X+Z_1=\alpha Z+\beta.
\end{align*}
Adding these identities and using \eqref{eq:halfstep-affine}, we get
\[
Y_1+Z_1=(\alpha^2-2)X+(\alpha+2)\beta.
\]
Hence \eqref{eq:global-full-step-relation} holds throughout \(U\), with
\[
A=\alpha^2-2,
\quad
B=(\alpha+2)\beta.
\]
\medskip
\noindent\emph{Step 3:}
Fix \(s_0\in U\), and, for every \(k\in\mathbb Z\), set
\[
x_k=X(s_0+kh).
\]
Then \eqref{eq:global-full-step-relation} restricts to
\begin{equation}\label{eq:recurrence-n}
x_{k+1}-A\,x_k+x_{k-1}=B.
\end{equation}
This is an elementary constant-coefficient recurrence, and its closed-form solutions are recorded in \cite[pp.~18--20]{NU83}; see Remark~\ref{rem:NU2order}.

\medskip
\noindent\emph{Step 4:} \noindent\emph{Case (i): \(A=2\).}
Then \eqref{eq:recurrence-n} becomes
\[
x_{k+1}-2x_k+x_{k-1}=B.
\]
Let \(d_k=x_{k+1}-x_k\). Then \(d_k-d_{k-1}=B\), and therefore
\[
d_k=d_0+Bk,
\]
where
\[
d_0=x_1-x_0=Y_1(s_0)-X(s_0).
\]
It follows that the quadratic polynomial
\[
u_k=\frac{B}{2}\,k^2+\left(x_1-x_0-\frac{B}{2}\right)k+x_0
\]
satisfies
\[
u_{k+1}-2u_k+u_{k-1}=B,
\]
together with the initial conditions \(u_0=x_0\) and \(u_1=x_1\). By uniqueness for the recurrence \eqref{eq:recurrence-n}, we conclude that \(x_k=u_k\) for all \(k\in\mathbb Z\), and therefore the expression for \(x_k\) is exactly the one stated in Theorem~\ref{thm:admissible-lattice-classification-U}\emph{(i)}.

\smallskip
\noindent\emph{Case (ii): \(A=-2\).}
Since \(A=\alpha^2-2\), the condition \(A=-2\) forces \(\alpha=0\). Thus \eqref{eq:halfstep-affine} reads
\[
Y+Z=\beta,
\]
and from \(B=(\alpha+2)\beta\) we obtain \(B=2\beta\).
Replacing \(s\) by \(s+\tfrac{h}{2}\) gives
\[
X+Y_1=\beta=\frac{B}{2},
\]
which is \eqref{eq:A=-2-first-order}. Evaluating at \(s=s_0+kh\) yields
\[
x_{k+1}+x_k=\frac{B}{2},
\]
so \(x_{k+1}=\frac{B}{2}-x_k\), whence \(x_{k+2}=x_k\). Writing
\[
x_k=\frac{B}{4}+u_k,
\]
we obtain
\[
u_{k+1}=-u_k,
\]
so that \(u_k=(-1)^k u_0\). Hence the corresponding formula for \(x_k\) is precisely that stated in Theorem~\ref{thm:admissible-lattice-classification-U}\emph{(ii)}.

\smallskip
\noindent\emph{Case (iii): \(A^2\neq4\).}
Choose \(q\in\mathbb C^\times\), \(q\neq\pm1\), with \(A=q+q^{-1}\), and set
\[
c=\frac{B}{2-A},\quad y_k=x_k-c.
\]
Then \eqref{eq:recurrence-n} becomes the homogeneous recurrence
\[
y_{k+1}-A\,y_k+y_{k-1}=0,
\]
whose characteristic polynomial \(r^2-Ar+1\) has distinct roots \(q\) and \(q^{-1}\). Hence
\[
y_k=a_{s_0}\,q^k+b_{s_0}\,q^{-k},
\]
so
\[
x_k=c+a_{s_0}\,q^k+b_{s_0}\,q^{-k}.
\]
Moreover,
\[
y_0=x_0-c=X(s_0)-c,\quad y_1=x_1-c=Y_1(s_0)-c,
\]
so \(a_{s_0}\) and \(b_{s_0}\) are uniquely determined by
\[
a_{s_0}+b_{s_0}=y_0,\quad a_{s_0}q+b_{s_0}q^{-1}=y_1,
\]
and solving gives the stated formulas for \(a_{s_0}\) and \(b_{s_0}\). This is Theorem~\ref{thm:admissible-lattice-classification-U}\emph{(iii)}.

\end{proof}

\begin{corollary}[Half-step refinement of the classification]
\label{cor:admissible-half-step-branches}
Under the hypotheses of
Theorem~\ref{thm:admissible-lattice-classification-U}, there exist unique
constants \(\eta,\beta\in\mathbb C\) such that
\[
Y+Z=\eta X+\beta
\]
throughout \(U\). The constants \(A\) and \(B\) in
\eqref{eq:global-full-step-relation} are then given by
\[
A=\eta^2-2,
\quad
B=(\eta+2)\beta.
\]
If \(v\in U\) is fixed and
\[
x_m=X\!\left(v+m\frac h2\right),
\quad m\in\mathbb Z,
\]
then the following half-step alternatives result.
\begin{itemize}
\item[\emph{(i)}]
If \(\eta=2\), then
\[
x_m=\frac{\beta}{2}m^2+r_v m+t_v.
\]
This is the \emph{ordinary quadratic branch}.

\item[\emph{(ii)}]
If \(\eta=-2\), then
\[
x_m=\frac{\beta}{4}+(r_v+t_v m)(-1)^m.
\]
This is the \emph{reflected quadratic branch}. Both this branch and the
preceding one give \(A=2\), although in the reflected branch one necessarily
has \(B=0\).

\item[\emph{(iii)}]
If \(\eta=0\), then
\[
x_m=\frac{\beta}{2}+r_v i^m+t_v(-i)^m.
\]
This is the general alternating branch, and it gives \(A=-2\).

\item[\emph{(iv)}]
If \(\eta\notin\{0,\pm2\}\), choose \(\lambda\in\mathbb C^\times\) such that
\[
\lambda+\lambda^{-1}=\eta.
\]
Then
\[
x_m=\frac{\beta}{2-\eta}+r_v\lambda^m+t_v\lambda^{-m}.
\]
Writing \(q=\lambda^2\), this is precisely the \(q\)-exponential branch at
full step. Once one of the two global characteristic roots \(q\) satisfying
\[
A=q+q^{-1}
\]
has been fixed, there is a unique square root, denoted by \(q^{1/2}\), such
that
\[
q^{1/2}+q^{-1/2}=\eta.
\]
The parameter \(q\) itself remains fixed; the preceding condition merely
records which of its two square roots is dictated by the half-step translation.
\end{itemize}
Here \(r_v,t_v\in\mathbb C\) are determined by the initial values on the
half-step progression through \(v\).
\end{corollary}

\begin{proof}
The first assertion is the admissibility condition applied to
\(p(x)=x^2\). The identities relating \(A,B\) to \(\eta,\beta\) are those
obtained in Step~2 of the proof of
Theorem~\ref{thm:admissible-lattice-classification-U}. On the half-step
progression through \(v\), the same relation becomes
\[
x_{m+1}-\eta x_m+x_{m-1}=\beta.
\]
The four displayed forms are the elementary solutions of this recurrence.
Finally, replacing a square root of \(q\) by its negative changes the sign of
\(q^{1/2}+q^{-1/2}\). Since \(\eta\neq0\) in case~\emph{(iv)}, precisely one
of the two roots is compatible.
\end{proof}

\begin{remark}[The two quadratic half-step branches]
The distinction made in
Corollary~\ref{cor:admissible-half-step-branches} must not be silently erased
by the full-step classification. The value \(A=2\) does not, by itself,
determine the sign of the half-step coefficient. Nor is the reflected branch
an empty formal possibility. Take \(h=2\), \(U=\mathbb Z\), and
\[
X(s)=s(-1)^s.
\]
Then
\[
Y+Z=-2X,
\quad
YZ=X^2-1,
\quad
Y-Z=-2(-1)^s,
\]
so \(X\) is admissible and belongs to the reflected quadratic branch,
although its full-step relation lies in case~\emph{(i)} of
Theorem~\ref{thm:admissible-lattice-classification-U}. At the operator level
the distinction is equally real. In the ordinary branch,
\[
D(x^n)=n x^{n-1}+\cdots,
\quad
S(x^n)=x^n+\cdots,
\]
whereas in the reflected branch,
\[
D(x^n)=n(-1)^{n-1}x^{n-1}+\cdots,
\quad
S(x^n)=(-1)^n x^n+\cdots.
\]
Thus any argument using the first pair of leading terms belongs to the
ordinary branch and should say so.
\end{remark}

\begin{remark}[Progression-wise coordinates and global operators]
There is a small but decisive point of interpretation here. A full-step
arithmetic progression \(s_0+h\mathbb Z\) is not half-step-invariant, and
therefore the restriction of \(X\) to that progression does not, by itself,
carry the operators \(D\) and \(S\). The progression-wise coefficients in
Theorem~\ref{thm:admissible-lattice-classification-U} are coordinates for one
and the same global admissible structure \((U,X,h)\), not parameters of
separate divided-difference calculi.

The missing half-step values are nevertheless recovered from the global
relation. In the \(q\)-exponential branch, if
\[
X(s_0+kh)=c+a_{s_0}q^k+b_{s_0}q^{-k}
\]
and \(q^{1/2}\) is the compatible root fixed in
Corollary~\ref{cor:admissible-half-step-branches}, then
\begin{align*}
Y(s_0+kh)
&=
c+a_{s_0}q^{1/2}q^k+b_{s_0}q^{-1/2}q^{-k},
\\[7pt]
Z(s_0+kh)
&=
c+a_{s_0}q^{-1/2}q^k+b_{s_0}q^{1/2}q^{-k}.
\end{align*}
In the ordinary quadratic branch, if
\[
f_{s_0}(t)=at^2+b_{s_0}t+c_{s_0},
\quad
X(s_0+kh)=f_{s_0}(k),
\]
then
\[
Y(s_0+kh)=f_{s_0}\!\left(k+\frac12\right),
\quad
Z(s_0+kh)=f_{s_0}\!\left(k-\frac12\right).
\]
These identities, and not an illegitimate restriction of the domain of the
operators, are what permit the progression-wise coefficients to enter the
global formulas for \(D\) and \(S\).
\end{remark}

\begin{remark}\label{rem:classification-from-degree-two}
It is worth stressing precisely what the degree-two argument proves, and what
it does not. Applied to an admissible map, the single test polynomial
\[
p(x)=x^2
\]
already forces
\[
Y+Z=\eta X+\beta
\]
and hence the global full-step relation
\[
Y_1+Z_1=AX+B
\]
follows immediately. It therefore restricts every admissible map to one of the
recurrence types listed in
Theorem~\ref{thm:admissible-lattice-classification-U}. The converse statement
would be false: the degree-two relation alone does not imply admissibility on a
general half-step-invariant set. Different full-step components may satisfy
the same second-order recurrence and yet fail to share the global polynomial
identities required in higher degree. Accordingly, higher-degree tests
introduce no further full-step recurrence type, but they do impose
indispensable compatibility conditions among the progression-wise
coefficients. The first of these already appears for \(p(x)=x^3\).
\end{remark}

\begin{remark}[Magnus's conic]\label{rem:Magnus-conic-solid}
Assume the hypotheses of Theorem~\ref{thm:admissible-lattice-classification-U}, and retain the notation introduced above. The proof of Theorem~\ref{thm:admissible-lattice-classification-U} already shows that \(Y+Z\) is an affine function of \(X\). Thus there exist constants \(B,D\in\mathbb C\), not to be confused with the constants in Theorem~\ref{thm:admissible-lattice-classification-U}, such that
\[
Y+Z=-2BX-2D.
\]
At that point, the progression-wise recurrence type is already fixed. Passing
to divided differences of higher degree produces no new such type, but it does
impose further global compatibility conditions. Indeed, taking \(p(t)=t^3\) in Definition~\ref{def:admissible-lattice} shows that
\[
\frac{Y^3-Z^3}{Y-Z}
\]
is a polynomial function of \(X\) of degree at most \(2\). Since
\[
\frac{Y^3-Z^3}{Y-Z}=(Y+Z)^2-YZ,
\]
and \(Y+Z\) is already known to be affine in \(X\), it follows that \(YZ\) is a polynomial function of \(X\) of degree at most \(2\). Thus there exist constants \(C,E,F\in\mathbb{C}\) such that
\[
YZ=CX^2+2EX+F.
\]
Fix \(s\in U\), and set
\[
x=X(s),\quad y=Y(s),\quad z=Z(s).
\]
By Viète's formulae, the numbers \(y\) and \(z\) are the roots of the quadratic equation
\[
u^2+(2Bx+2D)u+Cx^2+2Ex+F=0.
\]
This is precisely the form used by Magnus in \cite[(1.2), p.~262]{Mag88} and \cite[(5), p.~254]{Mag95}. The point, however, is that this conic is secondary rather than fundamental. It records one consequence of the local half-step structure, but neither determines the structure of \(U\) nor fixes the set of attained values \(X(U)\).
\end{remark}

\begin{remark}[Global compatibility already contained in admissibility]
\label{rem:degree-three-global-compatibility}
Let \(R\in\mathcal P\) be the unique polynomial of degree at most \(2\) such
that
\[
YZ=R(X)
\]
throughout \(U\), whose existence follows by applying admissibility to
\(p(x)=x^3\). This single global polynomial already controls part of the
freedom left by the full-step classification.

In the \(q\)-exponential branch, write
\[
X(s_0+kh)=c+a_{s_0}q^k+b_{s_0}q^{-k}
\]
and let \(\eta=q^{1/2}+q^{-1/2}\), with the compatible root understood. The
half-step identities above give
\[
R(x)
=
(x-c)^2+\eta c(x-c)+c^2
+(\eta^2-4)a_{s_0}b_{s_0}
\]
on the image of the chosen progression. If \(q\) is not a root of unity, that
image is infinite, and the displayed equality is therefore a polynomial
identity. Suppose instead that \(q\) is a primitive \(\nu\)-th root of unity.
The \(q\)-exponential branch excludes \(q=\pm1\), so \(\nu\ge3\). We claim that
the values on one full-step progression are then pairwise distinct. Indeed, if
\[
c+a_{s_0}q^k+b_{s_0}q^{-k}
=
c+a_{s_0}q^\ell+b_{s_0}q^{-\ell},
\quad
k\not\equiv\ell\pmod\nu,
\]
and \(a_{s_0}b_{s_0}\neq0\), then
\[
\frac{b_{s_0}}{a_{s_0}}=q^{k+\ell}.
\]
Put \(r=(k+\ell)/2\). The compatible half-step formula gives
\[
(Y-Z)(s_0+rh)
=
\bigl(q^{1/2}-q^{-1/2}\bigr)
\bigl(a_{s_0}q^r-b_{s_0}q^{-r}\bigr)
=0,
\]
contrary to the parametrising condition. Here \(r\in\tfrac12\mathbb Z\), so
the point in question belongs to the same half-step progression. If one of the
two coefficients vanishes, the claim is immediate; they cannot both vanish.
Thus the progression contains
\(\nu\ge3\) distinct values, which again suffice to identify two polynomials
of degree at most \(2\). Since \(\eta^2\neq4\), it follows in either case that
\[
a_{s_0}b_{s_0}
\]
is independent of the chosen full-step progression.

Likewise, in the ordinary quadratic branch, if
\[
X(s_0+kh)=ak^2+b_{s_0}k+c_{s_0},
\]
then
\[
R(x)
=
x^2-\frac a2x+\frac{a^2}{16}
-\frac14\bigl(b_{s_0}^{\,2}-4ac_{s_0}\bigr).
\]
The image of this progression is infinite: if \(a\neq0\), the displayed
quadratic polynomial is non-constant, while, if \(a=0\), the parametrising
condition forces \(b_{s_0}\neq0\). The equality above is therefore a polynomial
identity, and hence
\[
b_{s_0}^{\,2}-4ac_{s_0}
\]
is independent of the progression. These compatibilities belong to
admissibility itself; they do not first arise from regularity or from a
recurrence coefficient of an orthogonal polynomial sequence.

There is also a compatibility in the reflected quadratic branch. There the
full-step representation has the form
\[
X(s_0+kh)=b_{s_0}k+c_{s_0},
\]
and the parametrising condition forces \(b_{s_0}\neq0\). The half-step
relation gives
\[
R(x)
=
\left(x-\frac\beta2\right)^2-\frac{b_{s_0}^{\,2}}4,
\]
so \(b_{s_0}^{\,2}\) is again independent of the progression. In the general
alternating branch, by contrast, an individual full-step progression may
have only two image values. The polynomial \(R\) remains global, but no
polynomial-identity argument based on one such progression is then available.
\end{remark}

\begin{remark}\label{remarkq}
Neither case~\emph{(i)} nor case~\emph{(ii)} of Theorem~\ref{thm:admissible-lattice-classification-U} should be understood as arising by direct substitution of \(q=1\) or \(q=-1\) into the formula of case~\emph{(iii)}. Indeed, in Step~4 of the proof, the representation
\[
x_k=c+a_{s_0}q^k+b_{s_0}q^{-k}
\]
is obtained under the assumption that the characteristic polynomial
\[
r^2-Ar+1
\]
has two distinct roots \(q\) and \(q^{-1}\), equivalently \(A^2\neq4\). In particular, the coefficients \(a_{s_0}\) and \(b_{s_0}\) involve the denominator \(q-q^{-1}\), so the parametrisation in case~\emph{(iii)} is not defined when \(q=\pm1\). The relation between the three cases is instead one of degeneration at the level of the characteristic equation. When \(A=2\), we have
\[
r^2-Ar+1=(r-1)^2,
\]
so the two characteristic roots coalesce at \(r=1\). The corresponding homogeneous recurrence then has basis \(1\) and \(k\), and the inhomogeneous recurrence yields the quadratic behaviour described in case~\emph{(i)}. When \(A=-2\), one has
\[
r^2-Ar+1=(r+1)^2,
\]
so the double root is \(r=-1\). At the level of the second-order recurrence alone, the general solution is
\[
x_k=\frac{B}{4}+(u+vk)(-1)^k,
\]
where \(u,v\in\mathbb C\). On the other hand, admissibility yields the additional first-order relation
\[
x_{k+1}+x_k=\frac{B}{2},
\]
which is the progression-wise form of \eqref{eq:A=-2-first-order}. Substituting the above expression for \(x_k\) into this identity forces \(v=0\), and hence
\[
x_k=\frac{B}{4}+u(-1)^k,
\]
which is precisely the alternating form stated in case~\emph{(ii)}. Thus cases~\emph{(i)} and~\emph{(ii)} arise when the two characteristic roots merge, rather than by direct substitution into the formula of case~\emph{(iii)}.
\end{remark}

\begin{remark}
The alternatives in Theorem~\ref{thm:admissible-lattice-classification-U} are
global. Indeed, the constants \(A\) and \(B\) are determined on all of \(U\),
and therefore the same full-step recurrence holds on every arithmetic
progression
\(
s_0+h\mathbb Z\subseteq U.
\)
Consequently, only one of the three cases in the theorem can occur, and it
occurs simultaneously on all such progressions. In particular, one cannot have
quadratic behaviour on one full-step arithmetic progression and
\(q\)-exponential behaviour on another. What may vary from one progression to another are only those coefficients which
are determined by the initial values of \(X\) on the progression in question.
Thus, in case~\emph{(iii)}, once a parameter \(q\in\mathbb C^\times\) satisfying
\[
A=q+q^{-1}
\]
has been fixed, the same parameter \(q\) governs the description on every
full-step arithmetic progression, whereas the coefficients \(a_{s_0}\) and
\(b_{s_0}\) may depend on \(s_0\). This distinction is particularly relevant in
the alternating case and, more generally, in torsion \(q\)-exponential
situations. In those regimes, the restriction of \(X\) to a single full-step
arithmetic progression may have finite image. Hence, whenever the full image
\(X(U)\) is infinite, that infinitude cannot be inferred from the behaviour
along one progression alone; it must come from the contribution of sufficiently
many distinct full-step progressions in \(U\).
\end{remark}

\begin{remark}[Non-constant affine map]
The affine situation is already rigid at the level of
Theorem~\ref{thm:admissible-lattice-classification-U}. Indeed, if an admissible
map is non-constant affine on a full-step arithmetic progression
\(s_0+h\mathbb Z\subseteq U\), then the corresponding global case must be
case~\emph{(i)}, and the quadratic coefficient must vanish, that is, \(B=0\).
In that case the formula in the theorem reduces to
\[
X(s_0+kh)=\bigl(Y_1(s_0)-X(s_0)\bigr)k+X(s_0),
\]
so the coefficient
\[
Y_1(s_0)-X(s_0)=X(s_0+h)-X(s_0)
\]
is exactly the increment of the image along one full step of the underlying
arithmetic progression. In other words, it is precisely the common difference
of the affine parametrisation on that full-step progression. Since the constants \(A\) and \(B\) are global, and since in the affine case one
has \(A=2\) and \(B=0\), every full-step arithmetic progression in the domain is
mapped onto an arithmetic progression, possibly degenerate. However,
Theorem~\ref{thm:admissible-lattice-classification-U} alone does not imply that
the corresponding common difference is the same on all such progressions; a
priori, the quantity
\[
Y_1(s_0)-X(s_0)
\]
may depend on the chosen progression \(s_0+h\mathbb Z\). In particular, in the situation discussed in the introduction, one takes
\(h=2\), so that
\[
\mathbb S=\mathbb Z\cup\left(\frac{\gamma}{2}+\mathbb Z\right)
\]
is a union of two half-step arithmetic progressions, that is, of two cosets of
\(\tfrac h2\mathbb Z=\mathbb Z\). For the admissible affine parametrisation
\[
X(s)=2s
\]
one has
\[
Y-X=2,
\quad
Y_1-X=4.
\]
The two arithmetic progressions are therefore not being treated by two separate admissible parametrisations: the single map \(X(s)=2s\) is admissible on their union \(\mathbb S\), making explicit the structural reparametrisation described in the Introduction. Thus the two half-step progressions are mapped with the same common difference
\(2\), while the corresponding full-step progressions are mapped with common
difference \(4\). A complete treatment of the affine case, together with the classification of
the corresponding classical orthogonal polynomial families, was given in
\cite{CG26a}. Unlike the affine case, however, the genuinely quadratic case is
not governed by a quantity directly interpretable as a slope. There the step of
the arithmetic progression merely fixes the discrete scale, while the actual
structural datum is the constant second difference
\[
X(s_0+(k+2)h)-2X(s_0+(k+1)h)+X(s_0+kh)=B.
\]
\end{remark}

\section{Classicality and orthogonality \`a la Maroni}
The material presented in this section is drawn principally from the work of
Maroni \cite{M84, M88, M91a, M91b}; we also use the exposition in
\cite{CP25}. The material has been adapted to the present setting and made
self-contained. The underlying facts concerning topological
vector spaces are standard and may be found in almost any textbook on the
subject. Although Maroni did not consider the structural objects studied in the
present paper in the form adopted here, the notions of orthogonality and
classicality used below are directly indebted to his functional point of view. A word should also be added about a terminology sometimes encountered in the
literature. Maroni's approach is occasionally described as a ``formal algebraic
approach''. The adjective ``algebraic'' is justified by the aim of isolating the
intrinsic polynomial structure of orthogonal polynomial systems, up to the
limits of their existence as algebraic objects; this is the point of view
discussed in \cite{CG26a}. The adjective ``formal'', however, is misleading and
should be avoided here. As we shall see below, Maroni's framework is
functional-analytic in its very formulation: its basic objects are continuous
linear functionals on an LCS of polynomials, and its identities
are identities in the continuous dual of that space. Thus what is sometimes
called ``formal'' is not an analytically ungrounded symbolic calculus, but a
precise analytic setting designed to separate the intrinsic algebraic structure
of orthogonal polynomials from the accidental features of any particular
measure-theoretic representation.

For each \(n\in\mathbb N\), let \(\mathcal P_n\) denote the vector space of all complex polynomials of degree at most \(n\). Since \(\mathcal P_n\) is finite-dimensional, it carries a unique Hausdorff locally convex vector-space topology; equivalently, after choosing a basis, one may identify it with \(\mathbb C^{\,n+1}\) endowed with its usual Euclidean topology. For \(n\le m\), the inclusion
\(
\iota_{n,m}:\mathcal P_n\hookrightarrow\mathcal P_m
\)
is a continuous linear embedding, and its range is closed. Set
\[
\mathcal P=\bigcup_{n=0}^\infty \mathcal P_n.
\]
We equip \(\mathcal P\) with the locally convex inductive-limit topology with respect to the canonical inclusions
\(
\iota_n:\mathcal P_n\to\mathcal P.
\)
Thus
\[
\mathcal P=\varinjlim \mathcal P_n
\]
as a strict countable inductive limit of finite-dimensional spaces. In particular, \(\mathcal P\) is Hausdorff and locally convex. A subset \(\mathcal B\subset\mathcal P\) is bounded if and only if there exists \(m\) such that \(\mathcal B\subset\mathcal P_m\) and \(\mathcal B\) is bounded in the finite-dimensional space \(\mathcal P_m\). Likewise, if \(E\) is an LCS, a linear map \(T:\mathcal P\to E\) is continuous if and only if each restriction
\(
T|_{\mathcal P_n}:\mathcal P_n\to E
\)
is continuous. Let
\[
\mathcal P'=\mathcal L(\mathcal P,\mathbb C)
\]
denote the continuous dual of \(\mathcal P\). We use the canonical pairing
\[
\langle\cdot,\cdot\rangle:\mathcal P'\times\mathcal P\to\mathbb C,\quad
\langle \mathbf u,p\rangle=\mathbf u(p).
\]
No algebraic functional is lost by passing to this continuous dual. Indeed,
every linear functional on \(\mathcal P\) restricts continuously to each
finite-dimensional space \(\mathcal P_n\), and is therefore continuous for the
inductive-limit topology. Thus
\[
\mathcal P'=\mathcal P^*
\]
as vector spaces. The topology is needed to place the transposes and their
convergence on a precise footing; it does not impose an unspoken restriction on
the possible moment sequences.
Since \(\mathcal P\) is Hausdorff and locally convex, this pairing separates points: if \(p\in\mathcal P\) is non-zero, then there exists \(\mathbf u\in\mathcal P'\) such that
\[
\langle \mathbf u,p\rangle\neq0.
\]
Unless explicitly stated otherwise, \(\mathcal P'\) is endowed with the weak topology
\(
\sigma(\mathcal P',\mathcal P),
\)
that is, the coarsest topology for which all maps
\[
\mathbf u\longmapsto \langle \mathbf u,p\rangle,
\]
where \(p\in\mathcal P\), are continuous.

If \(p_0\in\mathcal P\) is fixed, the multiplication operator
\[
M_{p_0}:\mathcal P\to\mathcal P,\quad q\longmapsto p_0\,q
\]
is continuous. Indeed, if \(\deg p_0=d\), then
\(
M_{p_0}(\mathcal P_n)\subseteq\mathcal P_{n+d},
\)
so each restriction
\(
M_{p_0}|_{\mathcal P_n}:\mathcal P_n\to\mathcal P_{n+d}
\)
is continuous, and the continuity criterion for the inductive limit yields continuity of \(M_{p_0}\) on \(\mathcal P\). More generally, if \(T:\mathcal P\to\mathcal P\) is linear and there exists \(d\ge0\) such that
\[
T(\mathcal P_n)\subseteq \mathcal P_{n+d},
\]
then \(T\) is continuous. Whenever \(T:\mathcal P\to\mathcal P\) is continuous, its transpose
\(
T':\mathcal P'\to\mathcal P'
\)
is well defined by
\[
\langle T'\mathbf u,p\rangle=\langle \mathbf u,Tp\rangle
\]
for every \(\mathbf u\in\mathcal P'\) and every \(p\in\mathcal P\). In particular, for \(p_0\in\mathcal P\) we write
\[
p_0\mathbf u=M_{p_0}'\mathbf u,
\]
so that
\[
\langle p_0\mathbf u,q\rangle=\langle \mathbf u,p_0\,q\rangle,
\]
for every \(q\in\mathcal P\). An identity in \(\mathcal P'\) is always understood in the functional sense: for \(\mathbf u,\mathbf v\in\mathcal P'\),
\[
\mathbf u=\mathbf v
\]
means that
\(
\langle \mathbf u,p\rangle=\langle \mathbf v,p\rangle
\)
for all \(p\in\mathcal P\). Since the pairing separates points, this is equivalent to equality as elements of \(\mathcal P'\). Finally, if \(T:\mathcal P\to\mathcal P\) is continuous, then \(T'\) is
\(\sigma(\mathcal P',\mathcal P)\)-continuous. Indeed, for each fixed \(p\in\mathcal P\),
\[
\mathbf u\longmapsto \langle T'\mathbf u,p\rangle=\langle \mathbf u,Tp\rangle
\]
is one of the defining weak-coordinate maps. Accordingly, all duality statements involving transposes of continuous linear operators on \(\mathcal P\) are to be understood in this strictly functional sense.

\begin{notation}\label{not:degree-indexed-polynomial-sequences}
Throughout this paper, whenever a polynomial sequence
\(
(P_n)_{n\in I}\subset \mathcal P
\)
is indexed by a set \(I\subseteq \mathbb N\), it is understood that \(P_n\) has degree \(n\) for every \(n\in I\).
\end{notation}

\begin{definition}[Orthogonality]\label{def:OPS-functional}
Let \(\mathbf u\in\mathcal P'\), and let \(I\subseteq\mathbb N\) be such that \(0\in I\). A family of polynomials
\(
(P_n)_{n\in I}
\)
is said to be \emph{orthogonal with respect to} \(\mathbf u\) if
\[
\langle \mathbf u,P_nP_m\rangle=0
\]
for every \(m,n\in I\) with \(m<n\), and
\[
\langle \mathbf u,P_n^2\rangle\neq0
\]
for every \(n\in I\). We say that \(\mathbf u\) is \emph{regular} if there exists an orthogonal polynomial sequence
\(
(P_n)_{n\in\mathbb N}
\)
with respect to \(\mathbf u\). For \(N\in\mathbb N\), we shall say that
\(\mathbf u\) is \emph{regular of order \(N+1\)} if there exists an orthogonal polynomial sequence
\(
P_0,P_1,\dots,P_N
\)
with respect to \(\mathbf u\). Thus the order counts the number of polynomials, not the largest degree.
\end{definition}

At this point, the role of admissibility becomes more concrete. Its purpose is precisely to ensure that the symmetric divided difference of a polynomial, computed from the neighbouring values \(Y\) and \(Z\), is again a polynomial in the base variable \(X\). Thus admissibility is exactly the condition that allows one to pass from the pointwise quotient in \eqref{Q} to a well-defined operator on \(\mathcal P\). The next definition isolates that operator.

\begin{definition}[Divided-difference operator]\label{def:divided-difference-operator}
Fix \(h\in\mathbb C^\times\), let \(U\subseteq\mathbb C\) be half-step-invariant, and let \(X:U\to\mathbb C\) be an admissible map. The symmetric divided-difference operator associated with \(X\) is the linear map
\(
D:\mathcal P\to\mathcal P
\)
which assigns to each \(p\in\mathcal P\) the unique polynomial \(Dp\in\mathcal P\) satisfying
\[
(Dp)(X)
=
\frac{p\!\left(Y\right)
      -p\!\left(Z\right)}
     {Y-Z}
\]
throughout \(U\). The dependence of \(D\) on the underlying admissible map \(X\) is understood and will remain implicit unless explicit notation is needed. The distributional transpose of \(D\), hereafter simply called its transpose, is denoted by
\(
\mathbf D:\mathcal P'\to\mathcal P',
\)
and is defined by
\[
\langle \mathbf D \mathbf u ,\, p\rangle
=
-\langle \mathbf u ,\, Dp\rangle.
\]
Thus
\[
\mathbf D=-D',
\]
where \(D'\) denotes the ordinary transpose of \(D\).
\end{definition}

\begin{definition}[Averaging operator]\label{def:averaging-operator}
Fix \(h\in\mathbb C^\times\), let \(U\subseteq\mathbb C\) be half-step-invariant, and let \(X:U\to\mathbb C\) be an admissible map. The symmetric averaging operator associated with \(X\) is the linear map
\(
S:\mathcal P\to\mathcal P
\)
which assigns to each \(p\in\mathcal P\) the unique polynomial \(Sp\in\mathcal P\) satisfying
\[
(Sp)(X)
=
\frac{
p\!\left(Y\right)
+
p\!\left(Z\right)
}{2}
\]
throughout \(U\). The dependence of \(S\) on the underlying admissible map \(X\) is understood and will remain implicit unless explicit notation is needed. The transpose of \(S\) is denoted by
\(
\mathbf S:\mathcal P'\to\mathcal P',
\)
and is defined by
\[
\langle \mathbf S \mathbf u ,\, p\rangle
=
\langle \mathbf u ,\, Sp\rangle,
\]
for every \(\mathbf u\in\mathcal P'\) and every \(p\in\mathcal P\).
\end{definition}

\begin{remark}\label{rem:D-and-S-well-defined}
The operator \(D\) is well defined directly from
Definition~\ref{def:admissible-lattice}. Indeed, if \(\deg p\ge1\), admissibility gives a polynomial \(Q\), uniquely determined because \(X(U)\) is infinite, such that
\[
Q(X)=\frac{p(Y)-p(Z)}{Y-Z}
\]
throughout \(U\), and we set \(Dp=Q\). If \(p\) is constant, then \(p(Y)-p(Z)=0\), and we set \(Dp=0\). The operator \(S\) is also well defined. By
Remark~\ref{rem:Magnus-conic-solid}, the quantities \(Y+Z\) and \(YZ\) are
polynomial functions of \(X\). Since
\[
\frac{p(Y)+p(Z)}{2}
\]
is symmetric in \(Y\) and \(Z\), it is a polynomial in the elementary symmetric
functions \(Y+Z\) and \(YZ\). Hence it is a polynomial function of \(X\). Since
\(X(U)\) is infinite, this polynomial is uniquely determined, and we denote it
by \(Sp\).
\end{remark}

The operators \(D\) and \(S\) are continuous on \(\mathcal P\). Indeed, by admissibility, for every \(n\in\mathbb N^\times\) one has
\(
D(\mathcal P_n)\subseteq \mathcal P_{n-1},
\)
while \(D(\mathcal P_0)=\{0\}\). Moreover,
\(
S(\mathcal P_n)\subseteq \mathcal P_n
\)
for every \(n\in\mathbb N\). For \(n\in \mathbb N^\times\), the restrictions
\(
D|_{\mathcal P_n}:\mathcal P_n\to\mathcal P_{n-1}
\)
are linear maps between finite-dimensional spaces, and for every \(n\in\mathbb N\), the restrictions
\(
S|_{\mathcal P_n}:\mathcal P_n\to\mathcal P_n
\)
are likewise linear maps between finite-dimensional spaces. Hence all these restrictions are continuous. By the continuity criterion for the inductive-limit topology on \(\mathcal P\), both \(D\) and \(S\) are therefore continuous. Consequently, the dual operators
\[
\mathbf D=-D',
\quad
\mathbf S=S',
\]
are well defined and \(\sigma(\mathcal P',\mathcal P)\)-continuous. In particular, finite compositions involving multiplication by fixed polynomials and the operators \(D\) and \(S\) are again continuous linear endomorphisms of \(\mathcal P\).

\begin{definition}[Classicality]\label{def:classical}
Fix \(h\in\mathbb C\). If \(h\neq0\), let \(U\subseteq\mathbb C\) be half-step-invariant, and let
\(X:U\to\mathbb C\) be an admissible map on \(U\). Let
\(
D,S:\mathcal P\to\mathcal P
\)
be the associated divided-difference and averaging operators. If \(h=0\), set
\[
D=\frac{d}{dx},
\quad
S=\mathrm{id}_{\mathcal P}.
\]
In either case, let \(\mathbf D=-D'\) and \(\mathbf S=S'\) denote the corresponding dual operators on \(\mathcal P'\). A functional \(\mathbf u\in\mathcal P'\) is said to be
\emph{classical} if it is regular of order at least \(3\) and there exist polynomials \(\phi\) and \(\psi\),
where \(\phi\) has degree at most \(2\) and \(\psi\) has degree at most \(1\), not both identically zero, such that
\[
\mathbf D(\phi\,\mathbf u)
=
\mathbf S(\psi\,\mathbf u).
\]
\end{definition}

One point should be stated without equivocation. Classicality is always
understood relative to the fixed calculus. When \(h\neq0\), this means the
full structure \((U,X,h)\); for the purpose of the definition, only the
operators \(D,S\) that it induces enter. Classicality is not a property of the
visible image \(X(U)\) alone.

The requirement, in the preceding definition, that \(\mathbf u\) be regular of
order at least \(3\), rather than merely of order at least \(2\), is imposed in
order to exclude certain degenerate situations that will be described in the
following remark.

\begin{remark}\label{rem:continuous-degenerate-classicality}
In the case \(h=0\), one has
\(
\mathbf D(\phi\,\mathbf u)=\psi\,\mathbf u.
\)
We record what happens in the two degenerate cases. Assume first that \(\phi=0\). Then
\[
\psi\,\mathbf u=0.
\]
Since \(\phi\) and \(\psi\) are not both identically zero, one has
\(\psi\neq0\). If \(\psi\) is a non-zero constant, then
\(\mathbf u=0\), which is impossible under the standing regularity
requirement. Hence
\(\deg\psi=1\), and we may write
\(
\psi(x)=a(x-\tau),
\)
with \(a\in\mathbb C^\times\) and \(\tau\in\mathbb C\). Thus
\[
(x-\tau)\mathbf u=0.
\]
Equivalently,
\[
\langle \mathbf u,p\rangle
=
p(\tau)\langle \mathbf u,1\rangle
\]
for every \(p\in\mathcal P\). Hence \(\mathbf u\) is a non-zero scalar multiple
of the evaluation functional at \(\tau\), and such a functional cannot be
regular beyond order \(1\). For if \(P_1\) were
orthogonal to \(P_0=1\), then
\[
0=\langle \mathbf u,P_1\rangle
=
P_1(\tau)\langle \mathbf u,1\rangle,
\]
so \(P_1(\tau)=0\), and consequently
\[
\langle \mathbf u,P_1^2\rangle
=
P_1(\tau)^2\langle \mathbf u,1\rangle
=
0.
\]
Assume next that \(\psi=0\). Then \(\phi\neq0\) and
\[
\mathbf D(\phi\,\mathbf u)=0.
\]
Since \(D=\frac{d}{dx}\) is surjective on \(\mathcal P\), it follows that
\(
\phi\,\mathbf u=0.
\)
Indeed, given \(q\in\mathcal P\), choose \(p\in\mathcal P\) such that
\(Dp=q\). Then
\[
0=\langle \mathbf D(\phi\,\mathbf u),p\rangle
=
-\langle \phi\,\mathbf u,q\rangle.
\]
Thus \(\mathbf u\) vanishes on the ideal \(\phi\mathcal P\), and therefore
factors through the quotient \(\mathcal P/\phi\mathcal P\). Since
\(\deg\phi\le2\), this quotient has dimension at most \(2\). Consequently,
\(\mathbf u\) cannot be regular of order greater than \(2\). To see this, if
\(P_0,\dots,P_N\) were orthogonal with respect to \(\mathbf u\), then their
classes in \(\mathcal P/\phi\mathcal P\) would be linearly independent: whenever
\[
\sum_{j=0}^{N}c_jP_j\in\phi\mathcal P,
\]
one has, for every \(m=0,\dots,N\),
\[
0=
\left\langle\mathbf u,
\left(\sum_{j=0}^{N}c_jP_j\right)P_m
\right\rangle
=
c_m\langle \mathbf u,P_m^2\rangle,
\]
and hence \(c_m=0\). Therefore
\[
N+1\le \dim(\mathcal P/\phi\mathcal P)\le2.
\]
Thus, in the continuous case, degenerate choices of \(\phi\) and \(\psi\)
cannot give classical functionals in the sense of
Definition~\ref{def:classical}. More precisely, when \(\phi=0\), regularity is
possible only up to order \(1\), whereas, when \(\psi=0\), regularity is
possible only up to order \(2\).
\end{remark}
\section{Regularity and recurrence coefficients}
In this section we refine the known necessary and sufficient conditions for the regularity of classical functionals in the \(q\)-exponential and ordinary quadratic cases, and derive the corresponding recurrence coefficients. Before turning to their closed non-resonant forms, we isolate the finite moment mechanism which remains valid when one of the usual divisors vanishes. We then treat the \(q\)-exponential case when \(q\) is not a root of unity, the ordinary quadratic case, and finally the finite torsion regime when \(q\) is a root of unity.

\subsection{Finite moment systems and resonances}\label{subsec:finite-resonance}
The finite case contains a point at which a formally innocent division may
change the statement one is trying to prove. What happens if the coefficient
which ordinarily determines the next moment vanishes? It does not follow that
finite regularity has disappeared. The same equation may instead become a
compatibility condition and leave one moment free. The correct finite object is
therefore the moment system itself.

Fix \(h\in\mathbb C^\times\), a half-step-invariant set \(U\), an admissible
map \(X:U\to\mathbb C\), and the associated operators \(D,S\). Fix also
\(\phi\in\mathcal P_2\) and \(\psi\in\mathcal P_1\), not both zero. For every
\(n\in\mathbb N\), write
\begin{equation}\label{eq:Theta-functional-equation}
\Theta_n(x)
:=
\phi(x)D(x^n)+\psi(x)S(x^n)
=
\sum_{k=0}^{n+1}\vartheta_{n,k}x^k,
\end{equation}
with \(\vartheta_{n,k}=0\) outside the displayed range. If
\[
\mu_k=\langle\mathbf u,x^k\rangle,
\]
then the functional equation
\[
\mathbf D(\phi\mathbf u)=\mathbf S(\psi\mathbf u)
\]
is equivalent to
\begin{equation}\label{eq:finite-moment-system}
\sum_{k=0}^{n+1}\vartheta_{n,k}\mu_k=0,
\quad n\in\mathbb N.
\end{equation}
Indeed,
\[
\big\langle
\mathbf D(\phi\mathbf u)-\mathbf S(\psi\mathbf u),x^n
\big\rangle
=
-\langle\mathbf u,\Theta_n\rangle.
\]
In the ordinary quadratic and \(q\)-exponential branches considered below,
the leading coefficient will be denoted by \(d_n\); there one has
\[
\vartheta_{n,n+1}=d_n.
\]
At the present abstract level, we shall call \(n\) a \emph{resonant index}
when \(\vartheta_{n,n+1}=0\).

For \(N\in\mathbb N\), let
\[
A_N
=
(\vartheta_{n,k})_{
0\le n\le2N+1,\,0\le k\le2N+2}
\]
and define the affine solution set of normalised finite moment vectors
\[
\mathfrak M_N
=
\left\{
\boldsymbol\mu=(\mu_0,\dots,\mu_{2N+2})\in\mathbb C^{2N+3}:
\mu_0=1,\ A_N\boldsymbol\mu=0
\right\}.
\]
For \(0\le m\le N+1\), put
\[
\Delta_m(\boldsymbol\mu)
=
\det(\mu_{i+j})_{i,j=0}^{m}.
\]
The vector \(\boldsymbol\mu\) defines a linear functional
\(\mathscr L_{\boldsymbol\mu}\) on \(\mathcal P_{2N+2}\) by
\[
\mathscr L_{\boldsymbol\mu}(x^k)=\mu_k,
\quad 0\le k\le2N+2.
\]
Whenever orthogonality with respect to
\(\mathscr L_{\boldsymbol\mu}\) is mentioned below,
Definition~\ref{def:OPS-functional} is understood in the truncated sense:
every product under consideration belongs to \(\mathcal P_{2N+2}\).
Equivalently, \(\mathscr L_{\boldsymbol\mu}\) may be extended to an element
of \(\mathcal P'\), and the choice of extension is immaterial through degree
\(N+1\).

\begin{theorem}[Finite resonant regularity]
\label{thm:finite-resonant-regularity}
Let \(N\in\mathbb N\). The set \(\mathfrak M_N\) consists precisely of the
normalised finite moment vectors, through degree \(2N+2\), which satisfy the
functional equation when it is tested on \(\mathcal P_{2N+1}\). Such a vector
admits a monic orthogonal polynomial sequence
\[
P_0,P_1,\dots,P_{N+1}
\]
with respect to \(\mathscr L_{\boldsymbol\mu}\) if and only if
\[
\Delta_m(\boldsymbol\mu)\neq0,
\quad m=0,1,\dots,N+1.
\]
When these conditions hold, the sequence is unique. With \(\Delta_{-1}=1\),
its squared norms and its subdiagonal recurrence coefficients are
\[
\mathscr L_{\boldsymbol\mu}(P_m^2)
=
\frac{\Delta_m}{\Delta_{m-1}},
\quad
C_m
=
\frac{\Delta_m\Delta_{m-2}}{\Delta_{m-1}^2},
\quad 1\le m\le N+1.
\]
\end{theorem}

\begin{proof}
Testing the functional equation on the monomial basis
\[
1,x,\dots,x^{2N+1}
\]
gives exactly \(A_N\boldsymbol\mu=0\), and the converse follows by linearity.
The remaining assertions are the Gram determinant criterion and the standard
determinantal construction of the monic orthogonal polynomials. The formulas
for the norms and for \(C_m\) follow by taking successive quotients of the
Hankel determinants.
\end{proof}

\begin{remark}\label{rem:finite-resonant-elimination}
Equation~\eqref{eq:finite-moment-system} has the triangular form
\[
\vartheta_{n,n+1}\mu_{n+1}
+\sum_{k=0}^{n}\vartheta_{n,k}\mu_k=0.
\]
If \(\vartheta_{n,n+1}\neq0\), it determines \(\mu_{n+1}\). If
\(\vartheta_{r,r+1}=0\), the equation at level \(r\) is instead the
compatibility condition
\[
\sum_{k=0}^{r}\vartheta_{r,k}\mu_k=0.
\]
When that condition holds, \(\mu_{r+1}\) is free at this level; a later
resonance may, of course, impose a further condition on moments already
obtained. Thus no division by \(\vartheta_{r,r+1}\) is legitimate, but
nothing in the finite problem has become indeterminate: the affine system
\(A_N\boldsymbol\mu=0\), followed by the Hankel conditions in
Theorem~\ref{thm:finite-resonant-regularity}, decides the truncated finite
problem completely at the prescribed order.
The same formulation accommodates several resonant indices without alteration,
as may occur in the torsion regime.
\end{remark}

\begin{example}[A resonance which remains regular]
\label{ex:quadratic-resonance}
Let \(h=2\), \(U=\mathbb C\), and \(X(s)=s\). Then
\[
(Dp)(x)=\frac{p(x+1)-p(x-1)}2,
\quad
(Sp)(x)=\frac{p(x+1)+p(x-1)}2.
\]
Take \(\phi(x)=x^2+1\) and \(\psi(x)=-2x\). For
\(\tau\in\mathbb C\), define
\begin{align*}
\langle\mathbf u_\tau,p\rangle
&=
\left(\frac12-\frac\tau4\right)p(1)
+\left(\frac12+\frac\tau4\right)p(-1)
\\[7pt]
&\quad
+\frac\tau4\bigl(p'(1)+p'(-1)\bigr).
\end{align*}
Now
\[
\Theta_p:=\phi Dp+\psi Sp
=
\frac12\bigl((x-1)^2p(x+1)-(x+1)^2p(x-1)\bigr),
\]
and direct evaluation gives
\begin{align*}
\Theta_p(1)&=-2p(0),
&
\Theta_p(-1)&=2p(0),
\\[7pt]
\Theta_p'(1)&=-2p(0)-2p'(0),
&
\Theta_p'(-1)&=-2p(0)+2p'(0).
\end{align*}
Substitution in the definition of \(\mathbf u_\tau\) now gives
\[
\langle\mathbf u_\tau,\phi Dp+\psi Sp\rangle=0
\]
for every \(p\in\mathcal P\). Hence
\[
\mathbf D(\phi\mathbf u_\tau)=\mathbf S(\psi\mathbf u_\tau).
\]
Its moments are
\[
\mu_{2k}=1,
\quad
\mu_{2k+1}=k\tau,
\]
and its first Hankel determinants are
\[
\Delta_0=1,
\quad
\Delta_1=1,
\quad
\Delta_2=-\tau^2,
\quad
\Delta_3=\tau^4,
\quad
\Delta_4=0.
\]
Here \(d_n=n-2\), so \(d_2=0\). Consequently, for \(\tau\neq0\), the
functional is regular of order \(4\) despite the resonance, but not of order
\(5\). This is exactly the possibility lost if the moment equation is divided
by every \(d_n\) before its compatibility has been examined.
\end{example}

\begin{remark}[The two-string finite representation]
The same mechanism is already visible, in a particularly concrete form, in
the affine treatment of the para--Krawtchouk functional in
\cite[Examples~3.6 and~8.6]{CG26a}. There the finite functional representation
is carried by two arithmetic strings. The step equations determine the weights
along each string, while the resonant equation leaves their relative
normalisation to be fixed by the functional itself. The explicit weights in
\cite{CG26a} do precisely that. Thus the two-string representation is not an
exception placed outside the present theory; it is one of the clearest ways in
which the finite moment freedom becomes visible.
\end{remark}

\subsection{\texorpdfstring{\(q\)}{q}-exponential map: \texorpdfstring{\(q\)}{q} is not a root of unity}
At the level of ideas, the proof of the following theorem is not essentially different from that of \cite[Theorem~4.1]{CMP22a} and \cite[Theorem~9.2]{CP25}. Nevertheless, the argument given there is not written in a form that makes entirely transparent how finite orthogonal polynomial sequences are to be included; moreover, it assumes unnecessarily that \(q>0\), and it does not make explicit how the decomposition of the underlying half-step-invariant sets enters the argument.

\begin{theorem}\label{thm:regularity-admissible-orbit}
Fix \(h\in\mathbb C^\times\), let \(U\subseteq\mathbb C\) be half-step-invariant, and let
\(X:U\to\mathbb C\) be an admissible map.
Let
\(
D,S:\mathcal P\to\mathcal P
\)
be the associated divided-difference and averaging operators, with transposes
\(
\mathbf D,\mathbf S:\mathcal P'\to\mathcal P'.
\)
Let \(\mathbf u\in\mathcal P'\) be such that
\(
\langle \mathbf u,1\rangle\neq0,
\)
and assume that there exist polynomials \(\phi\) and \(\psi\), where \(\phi\) has degree at most \(2\) and \(\psi\) has degree at most \(1\), not both identically zero, such that
\[
\mathbf D(\phi\,\mathbf u)=\mathbf S(\psi\,\mathbf u)
\]
in \(\mathcal P'\). Write
\[
\phi(x)=\phi_2 x^2+\phi_1 x+\phi_0,
\quad
\psi(x)=\psi_1 x+\psi_0.
\]
Assume that \(X\) belongs to the \(q\)-exponential branch of
Corollary~\ref{cor:admissible-half-step-branches}. Fix one of the two global
characteristic roots \(q\), with \(q\) not a root of unity, and let \(\eta\)
be the global coefficient determined by
\[
Y+Z=\eta X+\beta.
\]
Fix \(s_0\in U\), and write the progression-wise representation furnished by
Theorem~\ref{thm:admissible-lattice-classification-U} as
\[
X(s_0+kh)=a_{s_0}q^k+b_{s_0}q^{-k}+c,
\]
for every \(k\in\mathbb Z\), where \(c=B/(2-A)\) is global. Denote by
\(q^{1/2}\) the unique square root compatible with the half-step relation,
namely
\[
q^{1/2}+q^{-1/2}=\eta,
\]
and define \(q^{n/2}=(q^{1/2})^n\) for every \(n\in\mathbb Z\). Set
\[
\Pi=a_{s_0}b_{s_0};
\]
by Remark~\ref{rem:degree-three-global-compatibility}, \(\Pi\) is independent
of the chosen progression. Define
\[
\alpha_n=\frac{q^{n/2}+q^{-n/2}}{2},
\quad
\gamma_n=\frac{q^{n/2}-q^{-n/2}}{q^{1/2}-q^{-1/2}}.
\]
Write also \(\alpha=\alpha_1\). Set
\[
d_n=\phi_2\,\gamma_n+\psi_1\,\alpha_n,
\quad
e_n=(2\phi_2 c+\phi_1)\gamma_n+(\psi_1 c+\psi_0)\alpha_n,
\]
and define
\begin{align*}
\phi^{[n]}(x)
&=
\bigl(\psi_1(\alpha^2-1)\gamma_{2n}+\phi_2\alpha_{2n}\bigr)
\bigl((x-c)^2-2\Pi\bigr)
\\[7pt]
&\quad
+\bigl(\phi'(c)\alpha_n+\psi(c)(\alpha^2-1)\gamma_n\bigr)(x-c)
+\phi(c)+2\phi_2\Pi.
\end{align*}
Let \(I\subseteq\mathbb N\) be either \(I=\mathbb N\), or
\(
I=\{0,1,\dots,N+1\}
\)
for some \(N\in\mathbb N\). In the finite case set
\[
J_I=\{0,1,\dots,2N+1\},
\quad
K_I=\{0,1,\dots,N\},
\]
whereas in the infinite case set
\[
J_I=K_I=\mathbb N.
\]
If \(I=\mathbb N\), then there exists a monic orthogonal polynomial sequence
with respect to \(\mathbf u\) if and only if
\[
d_j\neq0,
\quad j\in\mathbb N,
\]
and
\[
\phi^{[n]}\!\left(c-\frac{e_n}{d_{2n}}\right)\neq0,
\quad n\in\mathbb N.
\]

If \(I=\{0,1,\dots,N+1\}\), assume first that
\[
d_j\neq0,
\quad j=0,1,\dots,2N+1.
\]
Then there exists a monic orthogonal polynomial sequence
\((P_n)_{n\in I}\) with respect to \(\mathbf u\) if and only if
\[
\phi^{[n]}\!\left(c-\frac{e_n}{d_{2n}}\right)\neq0,
\quad n=0,1,\dots,N.
\]
The finite solutions excluded by this non-resonance hypothesis are governed by
Theorem~\ref{thm:finite-resonant-regularity}. Whenever the non-vanishing
conditions displayed in the present theorem hold, the corresponding monic orthogonal polynomial sequence is
uniquely determined by \(P_{-1}=0\), \(P_0=1\), and by the recurrence
\[
P_{n+1}(x)=(x-B_n)P_n(x)-C_nP_{n-1}(x),
\]
for every \(n\in K_I\), with the convention that the term involving
\(C_0P_{-1}\) is void. The coefficients are given by
\[
B_0=c-\frac{\gamma_1e_0}{d_0},
\quad
C_1=-\frac{\gamma_1}{d_1}\,
\phi^{[0]}\!\left(c-\frac{e_0}{d_0}\right),
\]
and, for every \(n\in K_I^\times\),
\begin{align*}
B_n
&=
c+\frac{\gamma_n e_{n-1}}{d_{2n-2}}
-\frac{\gamma_{n+1}e_n}{d_{2n}},\\[7pt]
C_{n+1}
&=
-\frac{\gamma_{n+1}d_{n-1}}{d_{2n-1}d_{2n+1}}
\,
\phi^{[n]}\!\left(c-\frac{e_n}{d_{2n}}\right).
\end{align*}
In the finite case, the last coefficient \(C_{N+1}\) is the terminal norm
coefficient; it is not used to construct a polynomial \(P_{N+2}\).
\end{theorem}

\begin{proof}
Since \(q\) is not a root of unity, one has
\(
\gamma_n\neq0
\)
for every \(n\in\mathbb Z^\times\). Moreover, \(\alpha\neq0\), since
\(\alpha=0\) would imply \(q=-1\). Thus all normalisations below are well
defined. We first treat the case where
\(
I=\{0,1,\dots,N+1\}
\)
for some \(N\in\mathbb N\), under the non-resonance hypothesis in the
statement. In the infinite case the construction remains degree-local, but
necessity requires the global argument given in Step~11.

\medskip
\noindent{\em Step 1:}
Define
\[
\psi^{[n]}(x)=d_{2n}(x-c)+e_n,
\]
where
\[
d_n=\phi_2\gamma_n+\psi_1\alpha_n,
\quad
e_n=(2\phi_2 c+\phi_1)\gamma_n+(\psi_1 c+\psi_0)\alpha_n.
\]
Also define \(\phi^{[n]}\) by
\begin{align}
\phi^{[n]}(x)
&=
\bigl(\psi_1(\alpha^2-1)\gamma_{2n}+\phi_2\alpha_{2n}\bigr)
\bigl((x-c)^2-2\Pi\bigr)
\notag\\[7pt]
&\quad
+\bigl(\phi'(c)\alpha_n+\psi(c)(\alpha^2-1)\gamma_n\bigr)(x-c)
+\phi(c)+2\phi_2\Pi.
\label{eq:phi-n-present-proof-solid2}
\end{align}
These are exactly the transformed coefficients appearing in
\cite[Proposition~9.2]{CP25}, written in the notation of the present paper.

\medskip
\noindent{\em Step 2:}
The two elementary symmetric functions of the half-step values give
\[
U_1(x)=(\alpha^2-1)(x-c),
\quad
U_2(x)=\frac{(Y-Z)^2}{4}
=(\alpha^2-1)\bigl((x-c)^2-4\Pi\bigr).
\]
The factor \(-4\Pi\) in \(U_2\) should not be confused with the factor
\(-2\Pi\) occurring in \(\phi^{[n]}\); they belong to different polynomial
identities.
Define recursively functionals \(\mathbf u^{[n]}\in\mathcal P'\) by
\[
\mathbf u^{[0]}=\mathbf u,
\]
and, for every \(n\in\mathbb N\),
\begin{equation}\label{eq:u-n-present-proof-solid2}
\mathbf u^{[n+1]}
=
\mathbf D\!\Bigl((\alpha^2-1)\bigl((x-c)^2-4\Pi\bigr)\,\psi^{[n]}\mathbf u^{[n]}\Bigr)
-
\mathbf S\!\bigl(\phi^{[n]}\mathbf u^{[n]}\bigr).
\end{equation}
For \(n=0\), the identity
\[
\mathbf D\!\bigl(\phi^{[0]}\mathbf u^{[0]}\bigr)
=
\mathbf S\!\bigl(\psi^{[0]}\mathbf u^{[0]}\bigr)
\]
is precisely the assumed relation, since
\[
\mathbf u^{[0]}=\mathbf u,
\quad
\phi^{[0]}=\phi,
\quad
\psi^{[0]}=\psi.
\]
By \cite[Lemma~9.1, equation~(9.27)]{CP25}, together with the explicit
coefficient formulae in \cite[Proposition~9.2]{CP25}, the validity of the
identity at level \(n\) implies the corresponding identity at level \(n+1\).
Here both results have simply been rewritten in the present notation and used
with the definition of \(\mathbf u^{[n+1]}\). Hence, by induction,
\begin{equation}\label{eq:shifted-relation-present-proof-solid2}
\mathbf D\!\bigl(\phi^{[n]}\mathbf u^{[n]}\bigr)
=
\mathbf S\!\bigl(\psi^{[n]}\mathbf u^{[n]}\bigr),
\end{equation}
for every \(n\in\mathbb N\).

\medskip
\noindent{\em Step 3:}
Assume that
\(
d_j\neq0
\)
for every \(j\in J_I\), and that
\[
\phi^{[n]}\!\left(c-\frac{e_n}{d_{2n}}\right)\neq0
\]
for every \(n\in K_I\). We claim that, for each \(n\in I\), there exists a polynomial \(R_n\) such that
\begin{equation}\label{eq:Rn-rodrigues-present-proof-solid2}
R_n\,\mathbf u=\mathbf D^n\mathbf u^{[n]}.
\end{equation}
Moreover, these polynomials satisfy
\[
R_{-1}=0,
\quad
R_0=1,
\quad
R_1(x)=-\alpha\,\psi^{[0]}(x)
=
-\alpha\bigl(d_0(x-c)+e_0\bigr),
\]
and, for every \(n\in K_I^\times\),
\begin{equation}\label{eq:Rn-ttrr-present-proof-solid2}
R_{n+1}(x)=(a_n x-s_n)R_n(x)-t_nR_{n-1}(x),
\end{equation}
where
\begin{align}
a_n&=-\alpha\,\frac{d_{2n}d_{2n-1}}{d_{n-1}},
\label{eq:an-present-proof-solid2}\\[7pt]
s_n&=
a_n\left(
c+\frac{\gamma_n e_{n-1}}{d_{2n-2}}
-\frac{\gamma_{n+1}e_n}{d_{2n}}
\right),
\label{eq:sn-present-proof-solid2}\\[7pt]
t_n&=
a_n\,\frac{\alpha\gamma_n d_{2n-2}}{d_{2n-1}}\,
\phi^{[n-1]}\!\left(c-\frac{e_{n-1}}{d_{2n-2}}\right).
\label{eq:tn-present-proof-solid2}
\end{align}
This is the degree-local \(q\)-exponential Rodrigues construction given in the
proof of \cite[Theorem~9.1]{CP25}, rewritten in the notation of the present
paper. The locality of that argument is worth spelling out, since the theorem
there is stated for \(q>0\) and under a global non-vanishing hypothesis. If
\(q\) is not a root of unity, then
\[
\alpha_m\gamma_m\neq0,
\quad m\in\mathbb N^\times,
\]
for \(\alpha_m=0\) would give \(q^m=-1\), while \(\gamma_m=0\) would give
\(q^m=1\). In the inductive calculation which produces \(R_{n+1}\), the only
divisions, apart from these automatically non-zero factors and \(\alpha\), are
by the coefficients
\(
d_j\neq0
\)
for
\(
j\in\{0,1,\dots,2n\},
\)
and the preceding evaluation factors required are
\[
\phi^{[j]}\!\left(c-\frac{e_j}{d_{2j}}\right)\neq0
\]
for
\(
j\in\{0,1,\dots,n-1\}.
\)
Thus the initial step requires only \(d_0\neq0\), and the passage from level
\(n\) to level \(n+1\) uses no non-vanishing assumption from a later level.
The additional factors \(d_{2n+1}\) and
\(
\phi^{[n]}(c-e_n/d_{2n})
\)
enter only when the candidate terminal norm coefficient \(C_{n+1}\) is
evaluated; they are not used to construct \(R_{n+1}\). Every identity involved
is algebraic once these displayed divisors have been cleared. Consequently,
the same inductive construction remains valid for complex non-torsion \(q\)
and in the present finite setting. In particular, under the assumptions above,
it yields the existence of \((R_n)_{n\in I}\). We now normalise these
polynomials so as to obtain a monic family.

\medskip
\noindent{\em Step 4:}
For every \(n\in I\), set
\[
k_0=1,
\quad
k_n=(-\alpha)^{-n}\prod_{j=1}^{n}d_{n+j-2}^{-1},
\]
and define
\[
P_n=k_nR_n.
\]
By the leading-coefficient computation contained in
\cite[Theorem~9.1]{CP25}, this choice of \(k_n\) makes \(P_n\) monic. Multiplying
\eqref{eq:Rn-ttrr-present-proof-solid2} by \(k_{n+1}\), and using
\eqref{eq:an-present-proof-solid2}--\eqref{eq:tn-present-proof-solid2}, one obtains
\[
P_{n+1}(x)=(x-B_n)P_n(x)-C_nP_{n-1}(x),
\]
for \(n=0,1,\dots,N\), where
\[
B_0=c-\frac{\gamma_1e_0}{d_0},
\quad
C_1=-\frac{\gamma_1}{d_1}\,
\phi^{[0]}\!\left(c-\frac{e_0}{d_0}\right),
\]
and, for \(n=1,2,\dots,N\),
\[
B_n
=
c+\frac{\gamma_n e_{n-1}}{d_{2n-2}}
-\frac{\gamma_{n+1}e_n}{d_{2n}}.
\]
For \(n=1,2,\dots,N-1\), when this range is non-empty, the coefficients
already used in the constructed recurrence are
\[
C_{n+1}
=
-\frac{\gamma_{n+1}d_{n-1}}{d_{2n-1}d_{2n+1}}
\,
\phi^{[n]}\!\left(c-\frac{e_n}{d_{2n}}\right).
\]
The same closed expression at \(n=N\) will provisionally be denoted by
\(C_{N+1}\) when \(N\ge1\); for \(N=0\), the provisional terminal quantity
is the displayed value of \(C_1\). At this stage it is only the candidate
terminal coefficient: it has not been obtained by constructing
\(P_{N+2}\). Under the present assumptions all the displayed quantities are
well defined: every displayed denominator is non-zero, and every displayed
\(C\)-quantity is algebraically non-zero. The argument below proves, without
extending the recurrence, that the terminal candidate is precisely
\(h_{N+1}/h_N\).

\medskip
\noindent{\em Step 5:}
Let \(0\le m<n\le N+1\). Since \(\deg P_m=m<n\), one has
\(
D^nP_m=0.
\)
Using \eqref{eq:Rn-rodrigues-present-proof-solid2} and the definition of the transpose
\(\mathbf D\), we obtain
\begin{align*}
\langle \mathbf u,P_nP_m\rangle
&=
k_n\langle \mathbf u,R_nP_m\rangle
=
k_n\langle R_n\mathbf u,P_m\rangle
=
k_n\langle \mathbf D^n\mathbf u^{[n]},P_m\rangle
\\[7pt]
&=
(-1)^n k_n\langle \mathbf u^{[n]},D^nP_m\rangle
=
0.
\end{align*}
Thus
\[
\langle \mathbf u,P_nP_m\rangle=0
\]
whenever \(0\le m<n\le N+1\).
It remains to show that the squared norms do not vanish.
Since \(P_0=1\), we have
\[
\langle \mathbf u,P_0^2\rangle=\langle \mathbf u,1\rangle\neq0.
\]
Now let
\(
n=1,\dots,N.
\)
Pairing
\[
P_n(x)=(x-B_{n-1})P_{n-1}(x)-C_{n-1}P_{n-2}(x)
\]
with \(P_n\), and using orthogonality, gives
\[
\langle \mathbf u,P_n^2\rangle
=
\langle \mathbf u,xP_{n-1}P_n\rangle.
\]
Pairing
\[
P_{n+1}(x)=(x-B_n)P_n(x)-C_nP_{n-1}(x)
\]
with \(P_{n-1}\), and again using orthogonality, yields
\[
\langle \mathbf u,xP_nP_{n-1}\rangle
=
C_n\langle \mathbf u,P_{n-1}^2\rangle.
\]
Hence
\[
\langle \mathbf u,P_n^2\rangle
=
C_n\langle \mathbf u,P_{n-1}^2\rangle.
\]
This proves the non-vanishing of the norms through \(h_N\), where
\[
h_j=\langle\mathbf u,P_j^2\rangle.
\]
It remains to justify the terminal norm: the recurrence which constructs
\(P_0,\dots,P_{N+1}\) does not use \(C_{N+1}\), and therefore that coefficient
must not simply be presumed to be \(h_{N+1}/h_N\).

For \(0\le k\le N\) and \(0\le j\le N+1-k\), put
\begin{equation}\label{eq:derived-present-proof-solid2}
P_j^{[k]}(x)
=
\frac{D^kP_{j+k}(x)}{\displaystyle\prod_{\ell=1}^{k}\gamma_{j+\ell}},
\end{equation}
with the empty product interpreted as \(1\), and set
\[
H_j^{[k]}
=
\big\langle\mathbf u^{[k]},(P_j^{[k]})^2\big\rangle.
\]
We now make the degree-local norm calculation explicit. The product rules for
the fixed \(q\)-exponential calculus give, for arbitrary \(f,g\in\mathcal P\),
\begin{align*}
fDg
&=
D\!\left(\left(Sf-\alpha^{-1}U_1Df\right)g\right)
-\alpha^{-1}S(gDf),
\\[7pt]
fSg
&=
S\!\left(\left(Sf-\alpha^{-1}U_1Df\right)g\right)
-\alpha^{-1}U_2D(gDf).
\end{align*}
Let \(\mathbf v\) satisfy
\(
\mathbf D(\varphi\mathbf v)=\mathbf S(\vartheta\mathbf v)
\), and put
\[
\mathbf v^+
=
\mathbf D(U_2\vartheta\mathbf v)-\mathbf S(\varphi\mathbf v).
\]
Transposition of the preceding two identities, followed by the functional
equation for \(\mathbf v\), gives the local integration formula
\begin{equation}\label{eq:q-local-integration-by-parts}
\langle\mathbf v^+,fDg\rangle
=
\alpha\langle\mathbf v,
g(\varphi Df+\vartheta Sf)\rangle.
\end{equation}
Starting from the off-diagonal orthogonality at level \(k=0\) proved above,
assume inductively that the derived polynomials are orthogonal off the
diagonal at level \(k\). Apply
\eqref{eq:q-local-integration-by-parts} with
\[
\mathbf v=\mathbf u^{[k]},
\quad
\mathbf v^+=\mathbf u^{[k+1]},
\quad
\varphi=\phi^{[k]},
\quad
\vartheta=\psi^{[k]},
\]
and with
\(
f=P_m^{[k+1]}
\),
\(
g=P_{j+1}^{[k]}
\).
The polynomial
\[
\phi^{[k]}DP_m^{[k+1]}
+\psi^{[k]}SP_m^{[k+1]}
\]
has degree at most \(m+1\), and its leading coefficient is
\(
d_{m+2k}
\).
The induction hypothesis at level \(k\) therefore gives
\[
\gamma_{j+1}
\big\langle
\mathbf u^{[k+1]},
P_m^{[k+1]}P_j^{[k+1]}
\big\rangle
=
\alpha d_{j+2k}H_{j+1}^{[k]}\delta_{mj},
\]
for \(0\le m\le j\le N-k\). This establishes both the off-diagonal
orthogonality at level \(k+1\) and the norm identity
\begin{equation}\label{eq:local-norm-terminal-q}
H_j^{[k+1]}
=
\alpha\,\frac{d_{j+2k}}{\gamma_{j+1}}H_{j+1}^{[k]}
\end{equation}
whenever \(0\le k\le N-1\) and \(0\le j\le N-k\). Induction gives these
relations throughout the stated range. Notice that no non-vanishing of the
last norm has entered this calculation.

For \(N\ge1\), iteration along the two adjacent diagonals gives
\[
H_0^{[N]}
=
\alpha^N
\frac{\displaystyle\prod_{r=N-1}^{2N-2}d_r}
     {\displaystyle\prod_{r=1}^{N}\gamma_r}\,h_N,
\quad
H_1^{[N]}
=
\alpha^N
\frac{\displaystyle\prod_{r=N}^{2N-1}d_r}
     {\displaystyle\prod_{r=2}^{N+1}\gamma_r}\,h_{N+1}.
\]
At level \(N\), put
\[
r_N=c-\frac{e_N}{d_{2N}}.
\]
Testing the shifted functional equation first against \(1\) and then against
\(x\) gives, respectively,
\begin{align*}
\big\langle\mathbf u^{[N]},x-r_N\big\rangle
&=0,
\\[7pt]
\big\langle\mathbf u^{[N]},
\phi^{[N]}+\psi^{[N]}Sx\big\rangle
&=0.
\end{align*}
Since
\(
Sx=c+\alpha(x-c)
\),
\(
\psi^{[N]}(x)=d_{2N}(x-r_N)
\), and the coefficient of \(x^2\) in \(\phi^{[N]}\), say
\(\Phi_{2,N}\), satisfies
\[
\Phi_{2,N}+\alpha d_{2N}=d_{2N+1},
\]
the two displayed moment identities yield directly
\[
\frac{H_1^{[N]}}{H_0^{[N]}}
=
-\frac{1}{d_{2N+1}}\,
\phi^{[N]}\!\left(c-\frac{e_N}{d_{2N}}\right).
\]
Consequently,
\begin{equation}\label{eq:terminal-norm-q}
\frac{h_{N+1}}{h_N}
=
-\frac{\gamma_{N+1}d_{N-1}}{d_{2N-1}d_{2N+1}}\,
\phi^{[N]}\!\left(c-\frac{e_N}{d_{2N}}\right)
=C_{N+1}.
\end{equation}
For \(N=0\), this is precisely the direct first-coefficient calculation for
\(C_1=h_1/h_0\). Thus the terminal evaluation condition is exactly the
non-vanishing condition for \(h_{N+1}\), and
\[
\langle \mathbf u,P_n^2\rangle\neq0
\]
for \(n=0,1,\dots,N+1\). Hence
\(
(P_n)_{n\in I}
\)
is a monic orthogonal polynomial sequence with respect to \(\mathbf u\). This proves the sufficiency of the stated non-vanishing conditions in the case
\(
I=\{0,1,\dots,N+1\}.
\)

\medskip
\noindent{\em Step 6:}
Assume now that there exists a monic orthogonal polynomial sequence
\(
(P_n)_{n\in I}
\)
with respect to \(\mathbf u\), where
\(
I=\{0,1,\dots,N+1\}.
\)
For integers \(k,j\in\mathbb N\) with \(j+k\le N+1\), retain the notation
\eqref{eq:derived-present-proof-solid2}.
Since \(q\) is not a root of unity, one has
\(
\gamma_{j+\ell}\neq0
\)
for every index appearing in the product, so this is well defined. Moreover, \(D\) lowers degree by one, and its action on the leading term is governed by the coefficient \(\gamma_n\); hence each \(P_j^{[k]}\) is monic of degree \(j\).

\medskip
\noindent{\em Step 7:}
Let \(\phi_2^{[k]}\) denote the quadratic coefficient of \(\phi^{[k]}\), and \(\psi_1^{[k]}\) the linear coefficient of \(\psi^{[k]}\). By
\eqref{eq:phi-n-present-proof-solid2} and the definition of \(\psi^{[k]}\),
\begin{align*}
\phi_2^{[k]}=\psi_1(\alpha^2-1)\gamma_{2k}+\phi_2\alpha_{2k},\quad
\psi_1^{[k]}=\phi_2\gamma_{2k}+\psi_1\alpha_{2k}.
\end{align*}
Define
\[
d_n^{[k]}=\phi_2^{[k]}\gamma_n+\psi_1^{[k]}\alpha_n.
\]
A direct computation from the definitions of \(\alpha_n\) and \(\gamma_n\) gives
\[
\alpha_{2k}\gamma_n+\gamma_{2k}\alpha_n=\gamma_{n+2k},
\quad
\alpha_{2k}\alpha_n+(\alpha^2-1)\gamma_{2k}\gamma_n=\alpha_{n+2k},
\]
and therefore
\begin{equation}\label{eq:shifted-d-present-proof-solid2}
d_n^{[k]}=\phi_2\gamma_{n+2k}+\psi_1\alpha_{n+2k}=d_{n+2k}.
\end{equation}

\medskip
\noindent{\em Step 8:}
The local integration formula
\eqref{eq:q-local-integration-by-parts} applies verbatim at each shifted
level. Assuming orthogonality at level \(k\), take
\(
f=P_m^{[k+1]}
\)
and
\(
g=P_{n+1}^{[k]}
\).
For \(0\le m\le n\), the same degree and leading-coefficient calculation used
above gives
\begin{equation}\label{eq:norm-present-proof-solid2}
\big\langle \mathbf u^{[k+1]},P_n^{[k+1]}P_m^{[k+1]}\big\rangle
=
\alpha\,\frac{d_n^{[k]}}{\gamma_{n+1}}\,
\big\langle \mathbf u^{[k]},(P_{n+1}^{[k]})^2\big\rangle
\delta_{n,m}.
\end{equation}
This calculation uses only the transformed relation at level \(k\), the
derived polynomials up to index \(n+1\), and the orthogonality relations whose
indices do not exceed \(n+k+1\). It is therefore valid in the finite setting
whenever
\[
n+k+1\le N+1.
\]

\medskip
\noindent{\em Step 9:}
We claim that, for \(k=0,1,\dots,N\), the sequence
\(
P_0^{[k]},P_1^{[k]},\dots,P_{N+1-k}^{[k]}
\)
is orthogonal with respect to \(\mathbf u^{[k]}\), and that
\[
\big\langle \mathbf u^{[k]},(P_n^{[k]})^2\big\rangle\neq0
\]
for \(n=0,1,\dots,N+1-k\).
For \(k=0\), this is precisely the assumption that
\(
(P_n)_{n\in I}
\)
is orthogonal with respect to \(\mathbf u\). Assume it holds for some
\(k\le N-1\). Equation~\eqref{eq:norm-present-proof-solid2}, applied at level
\(k\), gives, for
\(0\le n,m\le N-k\),
\[
\big\langle \mathbf u^{[k+1]},P_n^{[k+1]}P_m^{[k+1]}\big\rangle
=
\alpha\,\frac{d_n^{[k]}}{\gamma_{n+1}}\,
\big\langle \mathbf u^{[k]},(P_{n+1}^{[k]})^2\big\rangle
\delta_{n,m}.
\]
By \eqref{eq:shifted-d-present-proof-solid2} and the standing finite
non-resonance hypothesis, every factor
\(d_n^{[k]}=d_{n+2k}\) occurring here is non-zero. The squared norms at level
\(k\) are non-zero by the induction hypothesis, and
\(\alpha\gamma_{n+1}\neq0\). Hence the displayed identity proves
simultaneously that
\(
P_0^{[k+1]},P_1^{[k+1]},\dots,P_{N-k}^{[k+1]}
\)
is orthogonal with respect to \(\mathbf u^{[k+1]}\) and that all its squared
norms are non-zero. This proves the induction step, including the terminal
derived norm.

\medskip
\noindent{\em Step 10:}
Fix \(n\in K_I\). By Step 9, the sequence
\(
P_0^{[n]},P_1^{[n]},\dots,P_{N+1-n}^{[n]}
\)
is orthogonal with respect to \(\mathbf u^{[n]}\), with non-zero squared norms.
Hence its first recurrence coefficients \(B_0^{[n]}\) and \(C_1^{[n]}\) are well defined, with
\[
C_1^{[n]}\neq0.
\]
Put
\[
r_n
=
c-\frac{\gamma_1e_n}{d_0^{[n]}},
\]
where \(\gamma_1=1\). Testing the shifted functional equation against \(1\)
gives
\[
\big\langle\mathbf u^{[n]},x-r_n\big\rangle=0.
\]
Thus
\(
P_1^{[n]}=x-r_n
\).
Testing the same equation against \(x\) gives
\[
\big\langle\mathbf u^{[n]},
\phi^{[n]}+\psi^{[n]}Sx\big\rangle=0.
\]
Since
\(
Sx=c+\alpha(x-c)
\), the coefficient of \((x-r_n)^2\) in the last integrand is
\[
\phi_2^{[n]}+\alpha\psi_1^{[n]}=d_1^{[n]}.
\]
The term linear in \(x-r_n\) has zero moment, and hence
\begin{equation}\label{eq:C1n-present-proof-solid2}
C_1^{[n]}
=
-\frac{\gamma_1}{d_1^{[n]}}\,
\phi^{[n]}\!\left(c-\frac{e_n}{d_0^{[n]}}\right).
\end{equation}
Using \eqref{eq:shifted-d-present-proof-solid2}, we have
\[
d_0^{[n]}=d_{2n},
\quad
d_1^{[n]}=d_{2n+1}.
\]
The factors \(d_{2n}\) and \(d_{2n+1}\) are non-zero by the standing finite
non-resonance hypothesis. Since \(C_1^{[n]}\neq0\), the displayed formula
therefore gives
\[
\phi^{[n]}\!\left(c-\frac{e_n}{d_{2n}}\right)\neq0
\]
for every \(n\in K_I\).
This proves the necessity of the stated evaluation conditions when
\(
I=\{0,1,\dots,N+1\}.
\)

This establishes the theorem in the finite case. Indeed, under the stated
non-vanishing conditions, Step~4 yields a monic family satisfying the recurrence
with the explicit coefficients displayed in the statement. Conversely, Step~10
shows that the stated evaluation conditions hold, and therefore Step~4 may
be applied to the same data. The monic family thereby obtained satisfies the
same orthogonality conditions as the given one, and by the standard uniqueness
of the monic orthogonal polynomial sequence indexed by \(I\), the two families
coincide. Hence the recurrence coefficients are necessarily the coefficients
displayed in the statement.

\medskip
\noindent{\em Step 11:}
We now turn to the case \(I=\mathbb N\). The passage to the infinite case is
not an inference from finite initial segments: infinite regularity enters
precisely in the proof that no coefficient \(d_n\) can vanish. We therefore
reproduce, in the present notation, the global argument of
\cite[Lemma~3.4]{CMP22a}.

Assume first that \(\mathbf u\) is regular, let
\((P_n)_{n\in\mathbb N}\) be its monic orthogonal polynomial sequence, and put
\[
\mathfrak h_n=\langle\mathbf u,P_n^2\rangle.
\]
Since \(q\) is not a root of unity, one has
\[
\alpha_m\neq0\quad(m\in\mathbb N),
\qquad
\gamma_m\neq0\quad(m\in\mathbb N^\times).
\]
Indeed, \(\alpha_m=0\) would give \(q^m=-1\), whereas \(\gamma_m=0\)
would give \(q^m=1\). In particular, \(\alpha\neq0\), and
\(D:\mathcal P_{m+1}\to\mathcal P_m\) is onto for every \(m\), since
\[
D(x^{m+1})=\gamma_{m+1}x^m+\text{terms of lower degree}.
\]
We first note that
\[
d_0=\psi_1\neq0.
\]
If \(\psi_1=0\), testing the functional equation at \(1\) gives
\(\psi_0\langle\mathbf u,1\rangle=0\), and hence \(\psi=0\). The functional
equation would then give \(\mathbf D(\phi\mathbf u)=0\). Since \(D\) is onto,
this means \(\phi\mathbf u=0\). A regular functional has no non-zero
polynomial annihilator: if \(f=\sum_{j=0}^m a_jP_j\), with \(a_m\neq0\), then
\[
\langle\mathbf u,fP_m\rangle=a_m\mathfrak h_m\neq0.
\]
Thus \(\phi=0\), contrary to the hypothesis that \(\phi\) and \(\psi\) are
not both zero.

Set
\[
\mathcal U_1(x)=(\alpha^2-1)(x-c),
\qquad
\mathcal U_2(x)=(\alpha^2-1)\bigl((x-c)^2-4\Pi\bigr),
\]
so that \eqref{eq:u-n-present-proof-solid2}, at the first level, reads
\[
\mathbf u^{[1]}
=
\mathbf D\!\bigl(\mathcal U_2\psi\mathbf u\bigr)
-\mathbf S(\phi\mathbf u).
\]
For \(j\in\mathbb N\), define
\[
P_j^{[1]}=\frac{DP_{j+1}}{\gamma_{j+1}}.
\]
The product identities for \(D\) and \(S\), together with the functional
equation, give the degree-local identity
\begin{equation}\label{eq:q-global-derived-norm}
\big\langle\mathbf u^{[1]},P_j^{[1]}P_k^{[1]}\big\rangle
=
\alpha\,\frac{d_j}{\gamma_{j+1}}\,
\mathfrak h_{j+1}\delta_{j,k},
\qquad 0\le k\le j.
\end{equation}
No regularity of \(\mathbf u^{[1]}\), and no non-vanishing of \(d_j\), is
used in this calculation. For \(n\in\mathbb N^\times\), put
\begin{align*}
G_n
&=
\mathcal U_2\psi\,DP_n^{[1]}
+\phi\,SP_n^{[1]},
\\[7pt]
H_n
&=
SP_{n+2}
+\alpha^{-1}\mathcal U_1DP_{n+2}.
\end{align*}
The same product identities give
\begin{equation}\label{eq:q-global-bridge}
\langle\mathbf u,G_nP_{n+2}\rangle
=
-\big\langle\mathbf u^{[1]},H_nP_n^{[1]}\big\rangle.
\end{equation}
Write
\[
H_n=\sum_{j=0}^{n+2}c_{n,j}P_j^{[1]}.
\]
The coefficient of \(x^{n+2}\) in \(G_n\) is
\[
(\alpha^2-1)\psi_1\gamma_n+\phi_2\alpha_n
=
\alpha d_n-d_{n-1}.
\]
Consequently, orthogonality of \((P_n)_{n\in\mathbb N}\), followed by
\eqref{eq:q-global-derived-norm} and \eqref{eq:q-global-bridge}, yields
\[
(\alpha d_n-d_{n-1})\mathfrak h_{n+2}
=
-\alpha\,\frac{c_{n,n}d_n}{\gamma_{n+1}}\,
\mathfrak h_{n+1}.
\]
Since
\[
C_{n+2}=\frac{\mathfrak h_{n+2}}{\mathfrak h_{n+1}}\neq0,
\]
we obtain
\begin{equation}\label{eq:q-global-d-induction}
\alpha\left(
1+\frac{c_{n,n}}{\gamma_{n+1}C_{n+2}}
\right)d_n
=d_{n-1}.
\end{equation}
Starting from \(d_0\neq0\), equation~\eqref{eq:q-global-d-induction} proves
inductively that
\[
d_n\neq0,
\qquad n\in\mathbb N.
\]
Equation~\eqref{eq:q-global-derived-norm} now shows that
\((P_j^{[1]})_{j\in\mathbb N}\) is the monic orthogonal polynomial sequence
associated with \(\mathbf u^{[1]}\). Applying the same argument to the
transformed functional equation and using
\[
d_j^{[k]}=d_{j+2k},
\]
we conclude that every \(\mathbf u^{[k]}\) is regular and that
\((P_j^{[k]})_{j\in\mathbb N}\) is its monic orthogonal polynomial sequence.
The first recurrence coefficient at level \(n\) is therefore non-zero, and the
calculation of Step~10 gives
\[
0\neq C_1^{[n]}
=
-\frac{\gamma_1}{d_{2n+1}}\,
\phi^{[n]}\!\left(c-\frac{e_n}{d_{2n}}\right).
\]
Hence
\[
\phi^{[n]}\!\left(c-\frac{e_n}{d_{2n}}\right)\neq0,
\qquad n\in\mathbb N.
\]
This proves necessity in the infinite case.

Conversely, assume that the two non-vanishing conditions in the statement
hold for every relevant index. The Rodrigues construction of Steps~1--4 is
degree-local and therefore produces \(P_n\) for every \(n\). Step~5 gives
\[
\langle\mathbf u,P_nP_m\rangle=0,
\qquad 0\le m<n,
\]
while the recurrence and the displayed formula for \(C_n\) give
\[
\langle\mathbf u,P_n^2\rangle
=
C_n\langle\mathbf u,P_{n-1}^2\rangle\neq0,
\qquad n\in\mathbb N^\times.
\]
Thus \((P_n)_{n\in\mathbb N}\) is the monic orthogonal polynomial sequence
associated with \(\mathbf u\), and the proof is complete. Although
\cite{CMP22a} states the result for positive \(q\), the preceding argument is
algebraic in the compatible root \(q^{1/2}\): no order, positivity, or
conjugation has been used.
\end{proof}

\begin{remark}\label{rem:q-exponential-global}
Remark~\ref{rem:degree-three-global-compatibility} has already shown, directly
from admissibility, that
\[
\Pi=a_{s_0}b_{s_0}
\]
is global in the non-torsion \(q\)-exponential branch. It is nevertheless
instructive to see that the recurrence coefficients remember the same
invariant. Whenever \(C_2\) belongs to the available orthogonal segment and
the non-resonance conditions include \(d_0d_3\neq0\), taking \(n=1\) in
Theorem~\ref{thm:regularity-admissible-orbit} gives
\[
C_2
=
-\frac{\gamma_2d_0}{d_1d_3}
\,
\phi^{[1]}\!\left(
c-\frac{e_1}{d_2}
\right).
\]
The dependence on \(\Pi\) is made explicit by
\begin{align*}
\phi^{[1]}(x)
&=
\bigl(\psi_1(\alpha^2-1)\gamma_2+\phi_2\alpha_2\bigr)
\bigl((x-c)^2-2\Pi\bigr)
\\[7pt]
&\quad
+\bigl(\phi'(c)\alpha_1+\psi(c)(\alpha^2-1)\gamma_1\bigr)(x-c)
+\phi(c)+2\phi_2\Pi.
\end{align*}
Thus the coefficient multiplying \(\Pi\) in \(\phi^{[1]}\) is
\[
-2\bigl(\psi_1(\alpha^2-1)\gamma_2+\phi_2\alpha_2\bigr)+2\phi_2.
\]
Using
\[
\gamma_2=2\alpha,
\quad
\alpha_2=2\alpha^2-1,
\]
this coefficient becomes
\[
-4(\alpha^2-1)(\phi_2+\alpha\psi_1)
=
-4(\alpha^2-1)d_1.
\]
Consequently, the coefficient multiplying \(\Pi\) in \(C_2\) is
\[
-\frac{\gamma_2d_0}{d_1d_3}
\bigl(-4(\alpha^2-1)d_1\bigr)
=
\frac{4\gamma_2d_0(\alpha^2-1)}{d_3}.
\]
In the present range one has
\[
\gamma_2\neq0,
\quad
\alpha^2-1\neq0.
\]
Therefore \(C_2\) detects \(\Pi\) through a non-zero linear coefficient. It
does not create its independence of the progression: that has already been imposed by
the global admissible structure. Nor does this calculation imply that
\(a_{s_0}\) and \(b_{s_0}\) are separately independent of the progression, or
that an equality of their products would suffice to assemble arbitrary local
descriptions into a single functional and a single orthogonal polynomial
sequence.

The role of \(U\) should be understood in this operator-theoretic sense. The
functional \(\mathbf u\) is an element of \(\mathcal P'\), not a functional on
\(U\). The set \(U\), together with \(X\) and \(h\), enters through the
definition of the operators \(D\) and \(S\). It becomes visible at the level of
support only when \(\mathbf u\) is represented by a measure, or by a discrete
sum, involving the points \(X(s)\). This is the situation alluded to in the introduction: different full-step
progressions in the parameter set may contribute to one and the same
representing support, without producing different orthogonal polynomial
sequences.
\end{remark}

\begin{remark}\label{rem:q-exponential-no-degenerate-pair}
In the \(q\)-exponential case, the regularity criterion of
Theorem~\ref{thm:regularity-admissible-orbit} shows that the degenerate
possibilities \(\phi=0\) and \(\psi=0\), although not excluded a priori by
Definition~\ref{def:classical}, cannot occur either in the infinite situation
or in a finite non-resonant segment covered by that theorem. Indeed, in either
of those two settings, if there exists a monic orthogonal polynomial sequence
\(
(P_n)_{n\in I}
\)
with respect to \(\mathbf u\), then the theorem gives in particular
\[
d_0\neq0.
\]
Since
\[
d_0=\phi_2\gamma_0+\psi_1\alpha_0=\psi_1,
\]
it follows that
\(
\psi_1\neq0.
\)
Hence \(\deg\psi=1\). Writing
\[
\psi(x)=\psi_1\left(x+\frac{\psi_0}{\psi_1}\right),
\]
one also has
\(
e_0=\psi_1c+\psi_0,
\)
and therefore
\[
c-\frac{e_0}{d_0}=-\frac{\psi_0}{\psi_1}.
\]
For \(n=0\), the transformed polynomial \(\phi^{[0]}\) reduces to \(\phi\), since
\[
\phi^{[0]}(x)
=
\phi_2(x-c)^2+\phi'(c)(x-c)+\phi(c)
=
\phi(x).
\]
Thus the regularity criterion gives
\[
0\neq \phi^{[0]}\!\left(c-\frac{e_0}{d_0}\right)=\phi\left(-\frac{\psi_0}{\psi_1}\right).
\]
In particular, \(\phi\) is not the zero polynomial. Therefore, in the
\(q\)-exponential case, any classical functional in the sense of
Definition~\ref{def:classical} to which
Theorem~\ref{thm:regularity-admissible-orbit} applies necessarily satisfies
\[
\deg\psi=1,
\quad
\phi\left(-\frac{\psi_0}{\psi_1}\right)\neq0.
\]
Consequently, neither
\(\phi=0\) nor \(\psi=0\) can occur.
No corresponding conclusion is being asserted here for a finite resonant
solution, which must instead be examined through
Theorem~\ref{thm:finite-resonant-regularity}.
\end{remark}

\subsection{Ordinary quadratic map}
The following theorem may, at a purely formal level, be viewed as the
degeneration \(q\to1\) of
Theorem~\ref{thm:regularity-admissible-orbit}. But why introduce a singular
family of auxiliary maps when the quadratic calculus already contains the
identities one needs? We work directly with the fixed operators \(D\) and
\(S\). This removes any ambiguity about a limiting functional and, at the same
time, makes visible the exact hypothesis on which the calculation rests.
Throughout this subsection, the expression \emph{quadratic case} refers to the
ordinary half-step branch. The reflected branch identified in
Corollary~\ref{cor:admissible-half-step-branches} is a genuine admissible
possibility, but the formulas below do not absorb it by a change of notation.
\begin{theorem}\label{thm:regularity-admissible-orbit-quadratic}
Fix \(h\in\mathbb C^\times\), let \(U\subseteq\mathbb C\) be half-step-invariant, and let
\(X:U\to\mathbb C\) be an admissible map.
Assume that \(X\) belongs to the ordinary quadratic half-step branch. Thus
there exists a global constant \(a\in\mathbb C\) such that
\[
Y+Z=2X+\frac a2
\]
throughout \(U\).
Let
\(
D,S:\mathcal P\to\mathcal P
\)
be the associated divided-difference and averaging operators, with transposes
\(
\mathbf D,\mathbf S:\mathcal P'\to\mathcal P'.
\)
Let \(\mathbf u\in\mathcal P'\) be such that
\(
\langle \mathbf u,1\rangle\neq0,
\)
and assume that there exist polynomials \(\phi\) and \(\psi\), where \(\phi\) has degree at most \(2\) and \(\psi\) has degree at most \(1\), not both identically zero, such that
\[
\mathbf D(\phi\,\mathbf u)=\mathbf S(\psi\,\mathbf u)
\]
in \(\mathcal P'\). Write
\[
\phi(x)=\phi_2 x^2+\phi_1 x+\phi_0,
\quad
\psi(x)=\psi_1 x+\psi_0.
\]
Fix \(s_0\in U\). The progression-wise representation furnished by
Theorem~\ref{thm:admissible-lattice-classification-U} may then be written as
\[
X(s_0+kh)=a k^2+b_{s_0}k+c_{s_0},
\]
for every \(k\in\mathbb Z\), with the same global coefficient \(a\). Set
\[
d_n=\phi_2 n+\psi_1,
\quad
e_n=\phi_1 n+\psi_0+\frac12\,a\psi_1 n^2,
\]
and define
\begin{align*}
\phi^{[n]}(x)
&=
\phi_2 x^2
+
\left(
\phi_1+\frac32\,a n\,d_n
\right)x
\\[7pt]
&\quad
+\phi\!\left(\frac14\,a n^2\right)
+\frac12\,a n\,\psi\!\left(\frac14\,a n^2\right)
+\frac n4\bigl(b_{s_0}^{\,2}-4ac_{s_0}\bigr)d_n.
\end{align*}
Let \(I\subseteq\mathbb N\) be either \(I=\mathbb N\), or
\(
I=\{0,1,\dots,N+1\}
\)
for some \(N\in\mathbb N\). In the finite case set
\[
J_I=\{0,1,\dots,2N+1\},
\quad
K_I=\{0,1,\dots,N\},
\]
whereas in the infinite case set
\[
J_I=K_I=\mathbb N.
\]
If \(I=\mathbb N\), then there exists a monic orthogonal polynomial sequence
with respect to \(\mathbf u\) if and only if
\[
d_j\neq0,
\quad j\in\mathbb N,
\]
and
\[
\phi^{[n]}\!\left(
-\frac14\,a n^2-\frac{e_n}{d_{2n}}
\right)\neq0,
\quad n\in\mathbb N.
\]

If \(I=\{0,1,\dots,N+1\}\), assume first that
\[
d_j\neq0,
\quad j=0,1,\dots,2N+1.
\]
Then there exists a monic orthogonal polynomial sequence
\((P_n)_{n\in I}\) with respect to \(\mathbf u\) if and only if
\[
\phi^{[n]}\!\left(
-\frac14\,a n^2-\frac{e_n}{d_{2n}}
\right)\neq0,
\quad n=0,1,\dots,N.
\]
The finite resonant solutions are governed by
Theorem~\ref{thm:finite-resonant-regularity}. Whenever the non-vanishing
conditions displayed in the present theorem hold, the corresponding monic orthogonal polynomial sequence is
uniquely determined by
\(P_{-1}=0\), \(P_0=1\),
and by the recurrence
\[
P_{n+1}(x)=(x-B_n)P_n(x)-C_nP_{n-1}(x),
\]
for every \(n\in K_I\), with the convention that the term involving
\(C_0P_{-1}\) is void. The coefficients are given by
\[
B_0=-\frac{e_0}{d_0},
\quad
C_1=-\frac{1}{d_1}\,
\phi^{[0]}\!\left(-\frac{e_0}{d_0}\right),
\]
and, for every \(n\in K_I^\times\),
\begin{align*}
B_n
&=
\frac{n\,e_{n-1}}{d_{2n-2}}
-\frac{(n+1)e_n}{d_{2n}}
-\frac12\,a n(n-1),\\[7pt]
C_{n+1}
&=
-\frac{(n+1)d_{n-1}}{d_{2n-1}d_{2n+1}}
\,
\phi^{[n]}\!\left(
-\frac14\,a n^2-\frac{e_n}{d_{2n}}
\right).
\end{align*}
In the finite case, the last coefficient \(C_{N+1}\) is the terminal norm
coefficient; it is not used to construct a polynomial \(P_{N+2}\).
\end{theorem}
\begin{proof}
Fix \(s_0\in U\), and abbreviate
\[
b=b_{s_0},
\quad
c=c_{s_0},
\quad
\Delta=b^2-4ac,
\quad
\delta=\frac a4.
\]
The half-step hypothesis is decisive. Applying
\[
Y+Z=2X+\frac a2
\]
successively along the half-step orbit through \(s_0\), and using the values
already prescribed on \(s_0+h\mathbb Z\), gives
\[
X(s_0+rh)=ar^2+br+c,
\quad
r\in\frac12\mathbb Z.
\]
Consequently, on that orbit,
\begin{align*}
Y+Z
&=2X+\frac a2,
\\[7pt]
YZ
&=X^2-\frac a2X+\frac{a^2}{16}-\frac{\Delta}{4}.
\end{align*}
By admissibility, \(YZ\) is a polynomial in \(X\). Moreover, the selected
half-step orbit has infinite image: if \(a\neq0\) this is immediate, whereas,
if \(a=0\), the parametrising condition forces \(b\neq0\). The preceding
identity therefore determines that polynomial throughout \(U\). In
particular, \(\Delta\) is independent of the selected full-step progression.
Thus the ordinary quadratic calculus is governed globally by
\[
Sx=x+\delta,
\quad
U_1=\frac a2,
\quad
U_2:=S(x^2)-(Sx)^2=ax+\frac{\Delta}{4}.
\]
Equivalently,
\[
U_2=\frac{(Y-Z)^2}{4}.
\]

For \(f,g\in\mathcal P\), direct substitution in the definitions of \(D\)
and \(S\) gives
\begin{align*}
D(fg)
&=(Df)(Sg)+(Sf)(Dg),
\\[7pt]
S(fg)
&=(Sf)(Sg)+U_2(Df)(Dg),
\\[7pt]
DSf
&=SDf+U_1D^2f.
\end{align*}
Two forms of these identities which will be used below are
\begin{align*}
fDg
&=
D\!\bigl((Sf-U_1Df)g\bigr)-S\!\bigl(gDf\bigr),
\\[7pt]
fSg
&=
S\!\bigl((Sf-U_1Df)g\bigr)-U_2D\!\bigl(gDf\bigr).
\end{align*}
By transposition,
\begin{align*}
\mathbf D(f\mathbf v)
&=
(Sf-U_1Df)\mathbf D\mathbf v
+(Df)\mathbf S\mathbf v,
\\[7pt]
\mathbf S(f\mathbf v)
&=
(Sf-U_1Df)\mathbf S\mathbf v
+(Df)\mathbf D(U_2\mathbf v),
\end{align*}
for every \(\mathbf v\in\mathcal P'\).

We now remain entirely within this fixed quadratic calculus. Set
\[
\phi^{[0]}=\phi,
\quad
\psi^{[0]}=\psi,
\quad
\mathbf u^{[0]}=\mathbf u,
\]
and define recursively
\begin{align*}
\phi^{[n+1]}
&=
S\phi^{[n]}+U_1S\psi^{[n]}+U_2D\psi^{[n]},
\\[7pt]
\psi^{[n+1]}
&=
D\phi^{[n]}+S\psi^{[n]}+U_1D\psi^{[n]},
\\[7pt]
\mathbf u^{[n+1]}
&=
\mathbf D\!\left(U_2\psi^{[n]}\mathbf u^{[n]}\right)
-\mathbf S\!\left(\phi^{[n]}\mathbf u^{[n]}\right).
\end{align*}
The preceding product identities, together with
\[
\mathbf D(\phi^{[0]}\mathbf u^{[0]})
=
\mathbf S(\psi^{[0]}\mathbf u^{[0]}),
\]
give, by induction,
\begin{align}
\mathbf D\!\left(\phi^{[n]}\mathbf u^{[n]}\right)
&=
\mathbf S\!\left(\psi^{[n]}\mathbf u^{[n]}\right),
\label{eq:quadratic-shifted-functional-equation}
\\[7pt]
\mathbf D\mathbf u^{[n+1]}
&=
-\psi^{[n]}\mathbf u^{[n]},
\label{eq:quadratic-derived-D}
\\[7pt]
\mathbf S\mathbf u^{[n+1]}
&=
-\bigl(\phi^{[n]}+U_1\psi^{[n]}\bigr)\mathbf u^{[n]},
\label{eq:quadratic-derived-S}
\\[7pt]
2U_1\mathbf u^{[n+1]}
&=
\mathbf S\!\left(U_2\psi^{[n]}\mathbf u^{[n]}\right)
-\mathbf D\!\left(U_2\phi^{[n]}\mathbf u^{[n]}\right).
\label{eq:quadratic-derived-U}
\end{align}
These are algebraic identities for the fixed operators \(D\) and \(S\); no
limiting functional is involved.

For completeness, write
\[
\phi^{[n]}(x)
=
\Phi_{2,n}x^2+\Phi_{1,n}x+\Phi_{0,n},
\quad
\psi^{[n]}(x)
=
\Psi_{1,n}x+\Psi_{0,n}.
\]
Using
\begin{align*}
Dx&=1, \quad Sx=x+\delta,
\\[7pt]
D(x^2)&=2x+2\delta, \quad 
S(x^2)=x^2+6\delta x+\delta^2+\frac{\Delta}{4},
\end{align*}
the preceding recursion becomes
\begin{align*}
\Phi_{2,n+1}
&=\Phi_{2,n},
\\[7pt]
\Psi_{1,n+1}
&=2\Phi_{2,n}+\Psi_{1,n},
\\[7pt]
\Phi_{1,n+1}
&=\Phi_{1,n}
+6\delta\bigl(\Phi_{2,n}+\Psi_{1,n}\bigr),
\\[7pt]
\Psi_{0,n+1}
&=\Psi_{0,n}+\Phi_{1,n}
+\delta\bigl(2\Phi_{2,n}+3\Psi_{1,n}\bigr),
\\[7pt]
\Phi_{0,n+1}
&=\Phi_{0,n}
+\delta\bigl(\Phi_{1,n}+2\Psi_{0,n}\bigr)
+\delta^2\Psi_{1,n}
\\[7pt]
&\quad
+\left(\delta^2+\frac{\Delta}{4}\right)
\bigl(\Phi_{2,n}+\Psi_{1,n}\bigr).
\end{align*}
Solving this elementary system, or simply checking the proposed solution by
induction, gives
\[
\psi^{[n]}(x)
=
d_{2n}\left(x+\frac14an^2\right)+e_n,
\]
where
\[
d_n=\phi_2n+\psi_1,
\quad
e_n=\phi_1n+\psi_0+\frac12a\psi_1n^2,
\]
and
\begin{align*}
\phi^{[n]}(x)
&=
\phi_2x^2
+\left(\phi_1+\frac32an\,d_n\right)x
\\[7pt]
&\quad
+\phi\!\left(\frac14an^2\right)
+\frac12an\,\psi\!\left(\frac14an^2\right)
+\frac n4\Delta d_n.
\end{align*}
In particular,
\[
d_j^{[n]}
:=
\phi_2^{[n]}j+\psi_1^{[n]}
=
d_{j+2n}.
\]
Set
\[
\zeta_n
=
-\frac14an^2-\frac{e_n}{d_{2n}},
\]
whenever \(d_{2n}\neq0\). Thus \(\zeta_n\) is the unique zero of
\(\psi^{[n]}\).

We next record the direct functional Rodrigues construction. Assume first that
the quantities \(d_j\) required below do not vanish. Set
\[
R_0=1,
\quad
R_1=-\psi,
\]
and, for \(n\in K_I^\times\), put
\begin{align*}
\mathfrak a_n
&=
-\frac{d_{2n}d_{2n-1}}{d_{n-1}},
\\[7pt]
\mathfrak s_n
&=
\mathfrak a_n
\left(
\frac{n e_{n-1}}{d_{2n-2}}
-\frac{(n+1)e_n}{d_{2n}}
-\frac12an(n-1)
\right),
\\[7pt]
\mathfrak t_n
&=
\mathfrak a_n
\frac{n d_{2n-2}}{d_{2n-1}}\,
\phi^{[n-1]}(\zeta_{n-1}),
\end{align*}
and
\[
R_{n+1}(x)
=
(\mathfrak a_nx-\mathfrak s_n)R_n(x)
-\mathfrak t_nR_{n-1}(x).
\]
The initial identity
\[
R_1\mathbf u
=
-\psi\mathbf u
=
\mathbf D\mathbf u^{[1]}
\]
is \eqref{eq:quadratic-derived-D} with \(n=0\). Suppose that
\[
R_{n-1}\mathbf u
=
\mathbf D^{\,n-1}\mathbf u^{[n-1]},
\quad
R_n\mathbf u
=
\mathbf D^{\,n}\mathbf u^{[n]}.
\]
We make the elimination in this induction step explicit. Put
\[
A=\mathbf D,
\quad
B=\mathbf S,
\quad
M\mathbf v=x\mathbf v,
\quad
\chi=\frac{\Delta}{4}-4\delta^2.
\]
The transposed product identities already established give
\begin{align*}
AM
&=(M-\delta)A+B,
\\[7pt]
AB
&=BA+2\delta A^2,
\\[7pt]
BM
&=(M+3\delta)B+4\delta MA+\chi A.
\end{align*}
Induction on \(m\) therefore yields, for \(m\ge1\),
\begin{align}
A^mM
&=
\bigl(M+\delta m(m-2)\bigr)A^m
+mBA^{m-1},
\label{eq:quadratic-AmM}
\\[7pt]
A^mM^2
&=
\bigl(
M^2+2\delta m^2M
+\delta^2m^2(m-2)^2+m\chi
\bigr)A^m
\notag\\[7pt]
&\quad
+2m\bigl(M+\delta(m^2-3m+3)\bigr)BA^{m-1}
+m(m-1)B^2A^{m-2}.
\label{eq:quadratic-AmM2}
\end{align}
In the second identity the last term is absent when \(m=1\).

We shall use the following immediate consequence, recorded here so that no
coefficient is hidden. Suppose that
\[
A(\varphi\mathbf v)=B(\vartheta\mathbf v),
\quad
\varphi(x)=A_2x^2+A_1x+A_0,
\quad
\vartheta(x)=px+r,
\]
where \(p\neq0\) and \(p-mA_2\neq0\). If
\[
T=A^m\mathbf v,
\quad
W=-A^m(\vartheta\mathbf v),
\]
then \eqref{eq:quadratic-AmM}--\eqref{eq:quadratic-AmM2} give
\begin{equation}\label{eq:quadratic-elimination-lemma}
A^m\bigl((\varphi+2\delta\vartheta)\mathbf v\bigr)
=
-\frac{A_2}{p-mA_2}\,xW
+\beta_mW
+\frac{p}{p-mA_2}\,
\varphi\!\left(-\frac rp\right)T,
\end{equation}
where
\[
\beta_m
=
-\frac{
A_1p-A_2\delta m(m+2)p-A_2r+2\delta p^2
}{p(p-mA_2)}.
\]
For \(m\ge1\), expand both sides by the two power identities and use
\[
A^m(\varphi\mathbf v)
=
BA^{m-1}(\vartheta\mathbf v)
+2\delta(m-1)A^m(\vartheta\mathbf v).
\]
The terms containing \(BA^{m-1}\mathbf v\) and
\(B^2A^{m-2}\mathbf v\) are thereby eliminated; the coefficients of
\(x^2T\) and \(xT\) cancel, and the three remaining coefficients are exactly
those in \eqref{eq:quadratic-elimination-lemma}. For \(m=0\), the identity is
simply the Euclidean division of \(\varphi+2\delta\vartheta\) by
\(\vartheta\); for \(m=1\), the term containing \(B^2\) is absent.

Now put
\[
W_j=A^j\mathbf u^{[j]}.
\]
Formula \eqref{eq:quadratic-AmM}, together with
\eqref{eq:quadratic-derived-S} at level \(n-1\), gives
\begin{equation}\label{eq:quadratic-W-first-reduction}
W_{n+1}
=
-d_{2n}xW_n-e_nW_n
+nd_{2n}A^{n-1}
\bigl((\phi^{[n-1]}+2\delta\psi^{[n-1]})
\mathbf u^{[n-1]}\bigr).
\end{equation}
Apply \eqref{eq:quadratic-elimination-lemma} with
\begin{align*}
m&=n-1,
&
\mathbf v&=\mathbf u^{[n-1]},
&
\varphi&=\phi^{[n-1]},
\\[7pt]
\vartheta&=\psi^{[n-1]},
&
p&=d_{2n-2},
&
r&=e_{n-1}+\delta(n-1)^2d_{2n-2}.
\end{align*}
Then
\[
p-mA_2=d_{n-1},
\quad
-\frac rp=\zeta_{n-1},
\quad
T=W_{n-1},
\quad
W=W_n.
\]
Substitution in \eqref{eq:quadratic-W-first-reduction} gives
\begin{align*}
-d_{2n}\left(1+\frac{n\phi_2}{d_{n-1}}\right)
&=
-\frac{d_{2n}d_{2n-1}}{d_{n-1}}
=\mathfrak a_n,
\\[7pt]
-e_n+nd_{2n}\beta_{n-1}
&=-\mathfrak s_n,
\\[7pt]
\frac{nd_{2n}d_{2n-2}}{d_{n-1}}
\phi^{[n-1]}(\zeta_{n-1})
&=-\mathfrak t_n.
\end{align*}
Consequently,
\[
W_{n+1}
=
(\mathfrak a_nx-\mathfrak s_n)W_n
-\mathfrak t_nW_{n-1}.
\]
Using the induction hypothesis, we conclude that
\[
R_{n+1}\mathbf u
=
\mathbf D^{\,n+1}\mathbf u^{[n+1]}.
\]
This proves, by induction,
\begin{equation}\label{eq:quadratic-Rodrigues-R}
R_n\mathbf u
=
\mathbf D^{\,n}\mathbf u^{[n]}.
\end{equation}
This is the fixed-operator algebra underlying
\cite[Theorems~3.1 and~4.2]{CMP22a} and
\cite[Theorem~9.3]{CP25}; the calculation above shows explicitly that no
passage from a \(q\)-functional is required.

For \(n\in I\), set
\[
k_0=1,
\quad
k_n=(-1)^n\prod_{j=1}^{n}d_{n+j-2}^{-1},
\quad
P_n=k_nR_n.
\]
The leading coefficient in the preceding recursion gives
\[
k_{n+1}\mathfrak a_n=k_n,
\]
for \(n\in K_I^\times\), while the initial case follows from
\(R_1=-\psi\). Thus \(P_n\) is monic. Moreover,
\eqref{eq:quadratic-Rodrigues-R} becomes
\begin{equation}\label{eq:quadratic-functional-Rodrigues}
P_n\mathbf u
=
k_n\mathbf D^{\,n}\mathbf u^{[n]}.
\end{equation}
Normalising the recurrence for \(R_n\) now gives
\[
P_{n+1}(x)
=
(x-B_n)P_n(x)-C_nP_{n-1}(x),
\]
for \(n\in K_I\), where
\[
B_0=-\frac{e_0}{d_0},
\quad
C_1=-\frac{1}{d_1}\phi^{[0]}(\zeta_0).
\]
In the finite case \(N=0\), the displayed value of \(C_1\) is only the
provisional terminal coefficient.
For \(n\in K_I^\times\), one has
\[
B_n
=
\frac{n e_{n-1}}{d_{2n-2}}
-\frac{(n+1)e_n}{d_{2n}}
-\frac12an(n-1).
\]
In the finite case, the remaining coefficients
already used in the constructed recurrence are
\[
C_{n+1}
=
-\frac{(n+1)d_{n-1}}{d_{2n-1}d_{2n+1}}\,
\phi^{[n]}(\zeta_n)
\]
for \(n=1,2,\dots,N-1\), when this range is non-empty; \(C_1\) is the
initial value displayed above. When \(N\ge1\), the same closed expression at
\(n=N\) is only the candidate terminal coefficient at this point; when
\(N=0\), that candidate is the displayed value of \(C_1\). It will be
justified below without constructing \(P_{N+2}\). In the infinite case the displayed formula
holds for every \(n\ge1\). Thus all coefficients actually used in the
recurrence have been obtained directly from the quadratic calculus.

We verify both orthogonality and the terminal norm. If \(m<n\), then
\(D^nP_m=0\), and \eqref{eq:quadratic-functional-Rodrigues} gives
\begin{align*}
\langle\mathbf u,P_nP_m\rangle
&=
k_n\langle\mathbf D^{\,n}\mathbf u^{[n]},P_m\rangle
\\[7pt]
&=
(-1)^nk_n\langle\mathbf u^{[n]},D^nP_m\rangle
=0.
\end{align*}
Write
\[
m_n=\langle\mathbf u^{[n]},1\rangle.
\]
Testing \eqref{eq:quadratic-shifted-functional-equation} against \(1\) gives
\[
\langle\mathbf u^{[n]},x-\zeta_n\rangle=0.
\]
Testing it against \(x\), and using \(Dx=1\) and \(Sx=x+\delta\), gives
\[
-\langle\mathbf u^{[n]},\phi^{[n]}\rangle
=
d_{2n}
\left\langle
\mathbf u^{[n]},
(x-\zeta_n)(x+\delta)
\right\rangle.
\]
Since
\[
(x-\zeta_n)(x+\delta)
=
(x-\zeta_n)^2+(\zeta_n+\delta)(x-\zeta_n),
\]
we obtain
\[
-\langle\mathbf u^{[n]},\phi^{[n]}\rangle
=
d_{2n}
\left\langle
\mathbf u^{[n]},
(x-\zeta_n)^2
\right\rangle.
\]
On the other hand,
\[
\phi^{[n]}(x)
=
\phi^{[n]}(\zeta_n)
+\ell_n(x-\zeta_n)
+\phi_2(x-\zeta_n)^2
\]
for a suitable \(\ell_n\in\mathbb C\). Hence
\[
d_{2n+1}
\left\langle\mathbf u^{[n]},(x-\zeta_n)^2\right\rangle
=
-\phi^{[n]}(\zeta_n)m_n.
\]
Finally, the definition of \(\mathbf u^{[n+1]}\), tested against \(1\),
gives
\begin{align*}
m_{n+1}
&=
-\langle\mathbf u^{[n]},\phi^{[n]}\rangle
\\[7pt]
&=
-\frac{d_{2n}}{d_{2n+1}}\,
\phi^{[n]}(\zeta_n)m_n.
\end{align*}
Starting from
\(
m_0=\langle\mathbf u,1\rangle\neq0
\)
and proceeding inductively, the last identity proves that all \(m_n\) in the
required range are non-zero under the stated non-vanishing assumptions.

Let
\[
h_n=\langle\mathbf u,P_n^2\rangle.
\]
Since \(D^nP_n=n!\), the functional Rodrigues formula gives
\[
h_n=(-1)^nn!k_nm_n.
\]
For \(n\in K_I^\times\),
\[
\frac{k_{n+1}}{k_n}
=
-\frac{d_{n-1}}{d_{2n-1}d_{2n}},
\]
and consequently
\begin{align*}
\frac{h_{n+1}}{h_n}
&=
-(n+1)\frac{k_{n+1}}{k_n}\frac{m_{n+1}}{m_n}
\\[7pt]
&=
-\frac{(n+1)d_{n-1}}{d_{2n-1}d_{2n+1}}\,
\phi^{[n]}(\zeta_n)
=C_{n+1}.
\end{align*}
For \(n=0\), the preceding first-norm computation gives directly
\[
\frac{h_1}{h_0}
=
-\frac{\phi^{[0]}(\zeta_0)}{d_1}
=C_1.
\]
Thus every norm in the required range is non-zero. In particular, in the
finite case,
\[
\frac{\langle\mathbf u,P_{N+1}^2\rangle}
     {\langle\mathbf u,P_N^2\rangle}
=C_{N+1}.
\]
The coefficient \(C_{N+1}\) is therefore genuinely the terminal norm
coefficient; no recurrence involving a non-existent \(P_{N+2}\) has been
used. This proves sufficiency, both in the infinite case and in every finite
non-resonant segment.

It remains to discuss necessity. We begin with the infinite case, for here one
must not replace a global argument by an arbitrary collection of finite ones.
Suppose that \(\mathbf u\) is regular, and let \((P_n)_{n\ge0}\) be its monic
orthogonal polynomial sequence. Set
\[
h_n=\langle\mathbf u,P_n^2\rangle.
\]
First, \(d_0=\psi_1\neq0\). Indeed, if \(\psi_1=0\), testing the functional
equation against \(1\) gives \(\psi_0=0\), and hence \(\psi=0\). Since \(D\)
is surjective on \(\mathcal P\) in the ordinary quadratic branch,
\(\mathbf D(\phi\mathbf u)=0\) would imply \(\phi\mathbf u=0\). A regular
functional has no non-zero polynomial annihilator, so this would force
\(\phi=0\), contrary to the hypotheses.

Set
\[
\mathbf u^{[1]}
=
\mathbf D(U_2\psi\mathbf u)-\mathbf S(\phi\mathbf u),
\quad
Q_n=\frac{1}{n+1}DP_{n+1}.
\]
The polynomials \(Q_n\) are monic. From the identity for \(fDg\), together
with \eqref{eq:quadratic-derived-D} and
\eqref{eq:quadratic-derived-S}, one obtains
\begin{equation}\label{eq:quadratic-integration-by-parts}
\langle\mathbf u^{[1]},fDg\rangle
=
\langle\mathbf u,g(\phi Df+\psi Sf)\rangle
\end{equation}
for all \(f,g\in\mathcal P\). Taking \(f=Q_m\) and \(g=P_{n+1}\), and using
orthogonality, we find
\begin{equation}\label{eq:quadratic-derived-norm}
\langle\mathbf u^{[1]},Q_nQ_m\rangle
=
\frac{d_n}{n+1}h_{n+1}\delta_{n,m}.
\end{equation}
Indeed, if \(m<n\), then \(\phi DQ_m+\psi SQ_m\) has degree at most
\(m+1<n+1\); if \(m=n\), its leading coefficient is
\[
\phi_2n+\psi_1=d_n.
\]

We now show that no \(d_n\) can vanish. For \(n\ge1\), expand
\[
SP_{n+2}+U_1DP_{n+2}
=
\sum_{j=0}^{n+2}\chi_{n,j}Q_j.
\]
The two polynomial identities for \(fDg\) and \(fSg\), together with
\eqref{eq:quadratic-derived-U}, give
\begin{align*}
&\left\langle
\mathbf u,
\bigl(U_2\psi DQ_n+\phi SQ_n\bigr)P_{n+2}
\right\rangle
\\[7pt]
&\qquad
=
-\left\langle
\mathbf u^{[1]},
\bigl(SP_{n+2}+U_1DP_{n+2}\bigr)Q_n
\right\rangle.
\end{align*}
The polynomial multiplying \(P_{n+2}\) on the left has leading coefficient
\[
\phi_2=d_n-d_{n-1},
\]
whereas \eqref{eq:quadratic-derived-norm} reduces the right-hand side to its
\(\chi_{n,n}\)-term. Hence
\[
(d_n-d_{n-1})h_{n+2}
=
-\frac{\chi_{n,n}d_n}{n+1}h_{n+1}.
\]
If \(d_n=0\), this identity gives
\[
d_{n-1}h_{n+2}=0.
\]
Since \(h_{n+2}\neq0\), induction from \(d_0\neq0\) yields
\[
d_n\neq0,
\quad n\in\mathbb N.
\]
Notice that this argument uses \(P_{n+2}\) and its non-zero norm. It is
therefore genuinely global and cannot be replaced by a finite initial segment
at its terminal index.

Now \eqref{eq:quadratic-derived-norm} shows that \((Q_n)_{n\ge0}\) is the
monic orthogonal polynomial sequence with respect to \(\mathbf u^{[1]}\), so
\(\mathbf u^{[1]}\) is regular. Repeating the argument and using
\[
d_j^{[n]}=d_{j+2n},
\]
we conclude that every \(\mathbf u^{[n]}\) is regular. Its first recurrence
coefficient is
\[
C_1^{[n]}
=
-\frac{\phi^{[n]}(\zeta_n)}{d_{2n+1}}.
\]
Since \(C_1^{[n]}\neq0\), it follows that
\[
\phi^{[n]}(\zeta_n)\neq0,
\quad n\in\mathbb N.
\]
This proves infinite necessity.

Finally, let
\[
I=\{0,1,\dots,N+1\}
\]
and assume the finite non-resonance conditions
\[
d_j\neq0,
\quad j=0,1,\dots,2N+1.
\]
All the identities above are degree-local. In particular,
\eqref{eq:quadratic-derived-norm} shows successively that
\[
P_0^{[n]},P_1^{[n]},\dots,P_{N+1-n}^{[n]}
\]
is a finite monic orthogonal polynomial sequence with respect to
\(\mathbf u^{[n]}\), where
\[
P_j^{[n]}
=
\frac{j!}{(j+n)!}D^nP_{j+n}.
\]
For \(n=0,1,\dots,N\), its first norm coefficient is therefore non-zero, and
the direct first-moment computation gives
\[
0\neq C_1^{[n]}
=
-\frac{\phi^{[n]}(\zeta_n)}{d_{2n+1}}.
\]
Hence
\[
\phi^{[n]}(\zeta_n)\neq0,
\quad n=0,1,\dots,N.
\]
This proves necessity within the finite non-resonant sector and completes the
proof.
\end{proof}

\begin{remark}\label{rem:quadratic-global}
The discriminant-like quantity
\[
\Delta=b_{s_0}^{\,2}-4ac_{s_0}
\]
is already global by
Remark~\ref{rem:degree-three-global-compatibility}. Once again, however, the
recurrence coefficient \(C_2\) provides a useful internal reflection of that
fact. Whenever \(C_2\) belongs to the available segment and the
non-resonance conditions include \(d_0d_3\neq0\), its formula
in Theorem~\ref{thm:regularity-admissible-orbit-quadratic} gives
\[
C_2
=
-\frac{2d_0}{d_1d_3}
\,
\phi^{[1]}\!\left(
-\frac14\,a-\frac{e_1}{d_2}
\right).
\]
Its dependence on \(\Delta\) comes from the final term in \(\phi^{[1]}\),
namely
\[
\frac14\Delta d_1.
\]
Therefore the coefficient of \(\Delta\) in \(C_2\) is
\[
-\frac{2d_0}{d_1d_3}\cdot\frac14d_1
=
-\frac{d_0}{2d_3}
\]
which is non-zero in the stated range. Thus \(C_2\) faithfully detects the
same invariant. It does not create its independence of the progression: that has
already been imposed by admissibility. Here \(a\) is global, whereas
\(b_{s_0}\) and \(c_{s_0}\) may separately depend on the chosen progression.

As in Remark~\ref{rem:q-exponential-global}, this statement should be
understood in the operator-theoretic sense. The functional \(\mathbf u\) is an
element of \(\mathcal P'\), not a functional on \(U\). The different full-step
progressions enter through the same admissible structure \((U,X,h)\), and hence
through the same operators \(D\) and \(S\). They may become visible at the level
of support only when \(\mathbf u\) is represented by a measure, or by a
discrete sum, involving the points \(X(s)\).
\end{remark}

\begin{remark}\label{rem:quadratic-no-degenerate-pair}
In the ordinary quadratic branch, the regularity criterion of
Theorem~\ref{thm:regularity-admissible-orbit-quadratic} likewise shows that the
degenerate possibilities \(\phi=0\) and \(\psi=0\), although not excluded
a priori by Definition~\ref{def:classical}, cannot occur either in the
infinite situation or in a finite non-resonant segment covered by that
theorem. Indeed, in either of those settings, if there exists a monic orthogonal polynomial sequence
\(
(P_n)_{n\in I}
\)
with respect to \(\mathbf u\), then the theorem gives in particular
\[
d_0\neq0.
\]
Since
\[
d_0=\phi_2\cdot 0+\psi_1=\psi_1,
\]
it follows that
\(
\psi_1\neq0.
\)
Hence \(\deg\psi=1\). Writing
\[
\psi(x)=\psi_1\left(x+\frac{\psi_0}{\psi_1}\right),
\]
one has
\(
e_0=\psi_0,
\)
and therefore
\[
-\frac{e_0}{d_0}=-\frac{\psi_0}{\psi_1}.
\]
Moreover, for \(n=0\), the transformed polynomial \(\phi^{[0]}\) is exactly \(\phi\), since
\[
\phi^{[0]}(x)=\phi_2x^2+\phi_1x+\phi_0=\phi(x).
\]
Hence the regularity criterion yields
\[
0\neq \phi^{[0]}\!\left(-\frac{e_0}{d_0}\right)=\phi\left(-\frac{\psi_0}{\psi_1}\right).
\]
In particular, \(\phi\) is not the zero polynomial. Therefore, in the quadratic case, any classical functional in the sense of
Definition~\ref{def:classical} to which
Theorem~\ref{thm:regularity-admissible-orbit-quadratic} applies necessarily
satisfies
\[
\deg\psi=1,
\quad
\phi\left(-\frac{\psi_0}{\psi_1}\right)\neq0.
\]
Consequently, neither
\(\phi=0\) nor \(\psi=0\) can occur.
No corresponding conclusion is asserted for a finite resonant solution.
\end{remark}

\subsection{\texorpdfstring{\(q\)}{q}-exponential map: \texorpdfstring{\(q\)}{q} is a root of unity}\label{rootunit}

The following corollary gives the finite non-resonant part of the torsion
\(q\)-exponential theory. The positive-real assumption in
\cite{CMP22a,CP25} excludes this regime, but the identities used there are
algebraic and may be specialised once the denominators which remain in the
finite statement have been identified. The degenerate value \(q=1\) splits,
according to its compatible half-step root, into the ordinary and reflected
quadratic branches; the theorem above treats the ordinary one. The value
\(q=-1\) belongs instead to the alternating regime discussed later. Repeated
torsion resonances, when they occur, are not to be removed by this corollary:
they remain under the exact finite moment criterion of
Theorem~\ref{thm:finite-resonant-regularity}.

\begin{corollary}\label{cor:torsion-q-exponential-finite}
Retain the hypotheses and progression-wise notation of
Theorem~\ref{thm:regularity-admissible-orbit}, except for the assumption that
\(q\) is not a root of unity. Assume instead that \(q\) is a primitive
\(\nu\)-th root of unity, with \(\nu\ge3\), use the compatible root
\(q^{1/2}\), and put
\[
\Pi=a_{s_0}b_{s_0}.
\]
This quantity is independent of the chosen progression by
Remark~\ref{rem:degree-three-global-compatibility}. Let
\(
I=\{0,\dots,N+1\}
\)
for some \(N\in\mathbb N\) such that
\[
N+1<\nu.
\]
Assume that
\[
d_j\neq0,
\quad j=0,1,\dots,2N+1.
\]
Then \(\mathbf u\) is regular of order \(N+2\) if and only if
\[
\phi^{[n]}\!\left(c-\frac{e_n}{d_{2n}}\right)\neq0,
\quad n=0,1,\dots,N.
\]
Whenever this holds, the recurrence coefficients are those displayed in
Theorem~\ref{thm:regularity-admissible-orbit}. If one of the stated
coefficients \(d_j\) vanishes, the finite problem is instead governed by
Theorem~\ref{thm:finite-resonant-regularity}.
\end{corollary}

\begin{proof}
Write \(\varrho=q^{1/2}\) for the compatible root. Since \(q\neq1\), one
has \(\varrho-\varrho^{-1}\neq0\), and
\[
\gamma_m
=
\frac{\varrho^m-\varrho^{-m}}
     {\varrho-\varrho^{-1}}.
\]
Thus \(\gamma_m=0\) if and only if \(q^m=1\). The inequalities
\[
1\le m\le N+1<\nu
\]
therefore give
\[
\gamma_m\neq0,
\quad m=1,\dots,N+1.
\]
Moreover, \(\alpha\neq0\), since \(\alpha=0\) would give \(q=-1\), which
cannot be a primitive root of order at least \(3\).

One might now be tempted simply to repeat every intermediate division in the
non-torsion proof. That is unnecessary, and in even order it may conceal a
vanishing \(\alpha_m\) which does not occur in the final formulas. We argue
instead at the level at which the finite statement lives. Normalise
\(\mu_0=1\). Under the standing assumptions
\[
d_j\neq0,
\quad j=0,1,\dots,2N+1,
\]
the triangular system \eqref{eq:finite-moment-system} determines successively
\(\mu_1,\dots,\mu_{2N+2}\); the only divisors introduced by this elimination
are the displayed \(d_j\). Consequently these moments, their Hankel
determinants, and the recurrence coefficients obtained from them are rational
functions of the structure constants.

For non-torsion values of \(\varrho\), the identities proved in
Theorem~\ref{thm:regularity-admissible-orbit} identify those rational
functions with the displayed expressions for \(B_n\), \(C_{n+1}\), and the
terminal norm. Let us make the specialisation step precise. Fix \(N\), regard
\(\varrho\) and the remaining structure constants as indeterminates, and
perform the preceding triangular elimination in the corresponding Laurent
polynomial ring, localised at the finitely many divisors which occur. For
suitable non-negative integers \(M,M_j,L_m\), multiplication by
\[
(\varrho-\varrho^{-1})^M
\prod_{j=0}^{2N+1}d_j^{M_j}
\prod_{m=1}^{N+1}(\varrho^m-\varrho^{-m})^{L_m}
\]
and by the remaining denominators explicitly present in the formulas turns
the difference between either side of each identity into a Laurent polynomial
in \(\varrho\), with polynomial dependence on the other structure constants.
It vanishes for the generic non-torsion values covered by
Theorem~\ref{thm:regularity-admissible-orbit}; those values are Zariski dense,
so the Laurent polynomial is identically zero. At the present torsion value,
\(\varrho-\varrho^{-1}\neq0\), the factors \(d_0,\dots,d_{2N+1}\) are non-zero
by hypothesis, and \(\varrho^m-\varrho^{-m}\neq0\) for
\(1\le m\le N+1\). The cleared identities may therefore be specialised and
then divided back by precisely the factors which occur in the final formulas.
No analytic limiting argument, and no division by a possibly vanishing
\(\alpha_m\), is involved.

Let
\[
h_j=\langle\mathbf u,P_j^2\rangle.
\]
With the normalisation \(h_0=\mu_0=1\), the elementary determinant identities
\[
\Delta_m
=
\prod_{j=0}^{m}h_j
=
\prod_{j=1}^{m}C_j^{\,m+1-j},
\quad m=0,1,\dots,N+1,
\]
show exactly what survives the specialisation. The factors
\(\gamma_1,\dots,\gamma_{N+1}\) do not vanish, as shown above, and every
remaining denominator in the formulas for \(C_1,\dots,C_{N+1}\) is non-zero
by hypothesis. Thus the displayed Hankel determinants are non-zero exactly
when
\[
\phi^{[n]}\!\left(c-\frac{e_n}{d_{2n}}\right)\neq0,
\quad n=0,1,\dots,N.
\]
The recurrence and terminal-norm formulas specialise without alteration.
This proves the corollary.
\end{proof}

\begin{remark}\label{rem:torsion-q-exponential-global}
The globality of
\[
\Pi=a_{s_0}b_{s_0}
\]
in the torsion regime is already a consequence of admissibility, not of
regularity. When \(N\ge1\), the formula for \(C_2\) in
Corollary~\ref{cor:torsion-q-exponential-finite} still records that invariant,
exactly as in Remark~\ref{rem:q-exponential-global}; it does not create it.
\end{remark}

\begin{remark}\label{rem:torsion-q-exponential-no-degenerate-pair}
Within the finite non-resonant torsion \(q\)-exponential regime covered by
Corollary~\ref{cor:torsion-q-exponential-finite}, the degenerate possibilities
\(\phi=0\) and \(\psi=0\), although not excluded a priori by
Definition~\ref{def:classical}, cannot occur for a classical functional.
Indeed, by Definition~\ref{def:classical}, \(\mathbf u\) is regular of order at
least \(3\). Hence, in the finite torsion setting under consideration, we may
work on a finite initial segment
\(
I=\{0,1,\dots,N+1\}
\)
with \(N\in\mathbb N^\times\) and \(N+1<\nu\), for instance \(N=1\). Whenever
the hypotheses of Corollary~\ref{cor:torsion-q-exponential-finite} are
satisfied for such a segment, the corollary gives the same non-vanishing
conditions as in Theorem~\ref{thm:regularity-admissible-orbit}. Therefore, by
the same argument used in the non-torsion case, neither \(\phi=0\) nor
\(\psi=0\) can occur.
\end{remark}

The next proposition records a standard finite discrete orthogonality consequence
associated with the simple zeros of the truncating polynomial \(P_{N+1}\). The
reader should keep in mind that this is a rather special situation within the
present theory. Nevertheless, since even this restricted case gives rise to
apparently new consequences, it is useful to pause and treat it explicitly.

\begin{proposition}\label{prop:simple-zeros-discrete-orthogonality}
Let \(N\in\mathbb N\), and let \(\mathbf u\in\mathcal P'\) be regular of
order \(N+1\), with monic
orthogonal polynomial sequence \(P_0,P_1,\dots,P_N\), and set
\[
h_n=\langle \mathbf u,P_n^2\rangle,
\]
for \(n=0,1,\dots,N\).
Assume that there exists a monic polynomial \(P_{N+1}\) of degree \(N+1\)
such that
\[
\langle\mathbf u,P_{N+1}p\rangle=0,
\quad p\in\mathcal P_N,
\]
and that \(P_{N+1}\) has \(N+1\) simple zeros
\(
x_0,x_1,\dots,x_N.
\)
Then the following statements hold.
\begin{itemize}
\item[\emph{(a)}]
\[
P_N(x_s)\neq0
\]
for every \(s=0,1,\dots,N\).

\item[\emph{(b)}]
The Christoffel numbers
\[
\lambda_s=\frac{h_N}{P_N(x_s)\,P'_{N+1}(x_s)},
\]
for every \(s=0,1,\dots,N\), are well defined.

\item[\emph{(c)}]
For every polynomial \(p\) with \(\deg p\le 2N+1\), one has
\[
\langle \mathbf u,p\rangle
=
\sum_{s=0}^{N}\lambda_s\,p(x_s).
\]

\item[\emph{(d)}]
In particular,
\[
\sum_{s=0}^{N}\lambda_s\,P_n(x_s)P_m(x_s)
=
h_n\,\delta_{nm}
\]
for every \(0\le n,m\le N\).
\end{itemize}
\end{proposition}
\begin{proof}
This is the standard finite Christoffel--Darboux argument, but it is worth
noting that no non-zero norm for \(P_{N+1}\) is being assumed. If
\(N\in\mathbb N^\times\), orthogonality gives
\[
P_{N+1}(x)
=
(x-B_N)P_N(x)-C_NP_{N-1}(x),
\quad
C_N=\frac{h_N}{h_{N-1}}\neq0.
\]
If \(P_N\) and \(P_{N+1}\) had a common zero, this recurrence would first
force \(P_{N-1}\) to vanish there; descending through the preceding
three-term recurrences would then give \(P_0=0\), which is impossible. If
\(N=0\), the same conclusion is immediate because \(P_0=1\). Thus
\(P_N(x_s)\neq0\), and the numbers \(\lambda_s\) are well defined because the
zeros \(x_s\) are simple.

For each \(s=0,1,\dots,N\), put
\[
\ell_s(x)
=
\frac{P_{N+1}(x)}
     {(x-x_s)P'_{N+1}(x_s)}.
\]
The polynomial \(P_{N+1}(x)/(x-x_s)\) is monic of degree \(N\); hence its
inner product with \(P_N\) is \(h_N\). On the other hand,
\[
P_N(x)-P_N(x_s)=(x-x_s)r_s(x),
\quad
r_s=0\ \hbox{if }N=0,
\quad
\deg r_s\le N-1\ \hbox{if }N\ge1,
\]
and therefore the assumed orthogonality of \(P_{N+1}\) gives
\[
\left\langle
\mathbf u,
\frac{P_{N+1}(x)}{x-x_s}P_N(x)
\right\rangle
=
P_N(x_s)
\left\langle
\mathbf u,
\frac{P_{N+1}(x)}{x-x_s}
\right\rangle.
\]
Consequently,
\[
\langle\mathbf u,\ell_s\rangle
=
\frac{h_N}{P_N(x_s)P'_{N+1}(x_s)}
=
\lambda_s.
\]

Finally, if \(\deg p\le2N+1\), division by \(P_{N+1}\) gives
\[
p=P_{N+1}q+r,
\quad q,r\in\mathcal P_N.
\]
The first term is annihilated by the hypothesis, while Lagrange interpolation
at the simple zeros of \(P_{N+1}\) gives
\[
r(x)=\sum_{s=0}^{N}r(x_s)\ell_s(x)
=\sum_{s=0}^{N}p(x_s)\ell_s(x).
\]
It follows that the Gaussian quadrature formula
\[
\langle \mathbf u,p\rangle
=
\sum_{s=0}^{N}\lambda_s p(x_s),
\quad \deg p\le 2N+1,
\]
holds. Taking \(p=P_nP_m\) gives the displayed discrete orthogonality
relation.
\end{proof}

The following example illustrates the preceding root-of-unity theory for the
\(q\)-exponential map.

\begin{example}[\(q\)-exponential map for \(q\) a root of unity]\label{AW}
Let
\[
q=e^{2\pi iM/\nu},
\quad
1\le M<\nu,
\quad
\gcd(M,\nu)=1,
\quad
\nu\ge3,
\]
let \(q^{1/2}\) denote the half-step-compatible root belonging to the ambient
admissible structure. Let \(A,B\in\mathbb C^\times\), and fix once and for
all
\[
\kappa=\sqrt{AB}.
\]
Consider a full-step arithmetic progression \(s_0+h\mathbb Z\subseteq U\) on which
\[
X(s_0+kh)=Aq^{-k}+Bq^k+C,
\]
with \(C\in\mathbb C\). Let \(\mathbf u\in\mathcal P'\) be such that
\(
\langle \mathbf u,1\rangle\neq0,
\)
and satisfy
\[
\mathbf D(\phi\,\mathbf u)=\mathbf S(\psi\,\mathbf u),
\]
where
\begin{align*}
\phi(x)
&=
2(1+abcd)(x-C)^2-2\kappa\,
(a+b+c+d+abc+abd+acd+bcd)(x-C)
\\[7pt]
&\quad
+4AB\,(ab+ac+ad+bc+bd+cd-abcd-1),
\\[7pt]
\psi(x)
&=
\frac{4q^{1/2}}{q-1}
\Bigl(
(abcd-1)(x-C)+\kappa\,(a+b+c+d-abc-abd-acd-bcd)
\Bigr),
\end{align*}
with \(a,b,c,d\in\mathbb C\). Set
\[
g=abcd.
\]
For this case, the quantity \(d_n\) appearing in
Theorem~\ref{thm:regularity-admissible-orbit} is
\[
d_n=
-\frac{4}{q^{1/2}-q^{-1/2}}\,
q^{-n/2}(1-gq^n).
\]
Fix \(N\) with \(0\le N\le\nu-2\). In the non-resonant sector
\[
1-gq^j\neq0
\]
for every
\(
j=0,1,\dots,2N+1
\), Corollary~\ref{cor:torsion-q-exponential-finite} says that
\(\mathbf u\) is regular of order \(N+2\) if and only if
\[
AB(1-abq^n)(1-acq^n)(1-adq^n)
(1-bcq^n)(1-bdq^n)(1-cdq^n)\neq0
\]
for \(
n=0,1,\dots,N.
\)
At the resonant values \(gq^j=1\), finite regularity is governed instead by
the moment system of Theorem~\ref{thm:finite-resonant-regularity}; resonance
alone is no obstruction.
Moreover, the recurrence coefficients furnished by
Theorem~\ref{thm:regularity-admissible-orbit} are as follows.
The coefficients \(B_0\) and \(C_1\) are given by the corresponding initial
formulas in Theorem~\ref{thm:regularity-admissible-orbit}. For
\(n=1,2,\dots,N\), one has
\begin{align*}
B_n
=
C+\kappa
&\left(
a+\frac1a
-\frac{(1-abq^n)(1-acq^n)(1-adq^n)(1-gq^{n-1})}
       {a(1-gq^{2n-1})(1-gq^{2n})}
\right.
\\[7pt]
&\quad\left.
-\frac{a(1-q^n)(1-bcq^{n-1})(1-bdq^{n-1})(1-cdq^{n-1})}
       {(1-gq^{2n-1})(1-gq^{2n-2})}
\right),
\end{align*}
with the usual limiting interpretation when \(a=0\), and
\begin{align*}
C_{n+1}
&=
AB(1-q^{n+1})(1-gq^{n-1})
\\[7pt]
&\quad\times
\frac{
(1-abq^n)(1-acq^n)(1-adq^n)(1-bcq^n)(1-bdq^n)(1-cdq^n)
}{
(1-gq^{2n-1})(1-gq^{2n})^2(1-gq^{2n+1})
}.
\end{align*}
Assume now that \(\mathbf u\) is regular of order \(\nu\) and lies in the
non-resonant sector
\[
g\notin\{q^k:\ k=0,1,\dots,\nu-1\}.
\]
To keep the four-parameter calculation away from its exceptional algebraic
loci, assume in addition that
\[
a,b,c,d\in\mathbb C^\times,
\quad
ab,ac,ad,bc,bd,cd\notin\{q^k:\ k=0,1,\dots,\nu-1\}.
\]
Let
\(
P_0,P_1,\dots,P_{\nu-1}
\)
be the corresponding monic orthogonal polynomial sequence. In particular,
\(
C_n\neq0
\)
for \(
n=1,\dots,\nu-1.
\)
These coefficients are precisely the recurrence coefficients of the monic
Askey--Wilson polynomial sequence. Hence, by uniqueness of the monic polynomial sequence determined by the three-term recurrence, one has
\[
P_n(x)=
2^n\kappa^n
Q_n\!\left(
\frac{x-C}{2\kappa};
a,b,c,d\Bigm| q
\right),
\]
for \(n=0,1,\dots,\nu-1\),
where \((Q_n)_{n\in\mathbb N}\) denotes the monic Askey--Wilson polynomial sequence; see
\cite[(14.1.5)]{KLS10}. In particular,
if one writes
\[
x=C+\kappa(t+t^{-1}),
\]
then
\[
\frac{x-C}{2\kappa}=\frac{t+t^{-1}}{2}.
\]
The denominators involving powers of \(g\) in the displayed rational factors
are non-zero under the preceding assumptions. We now determine the terminal
polynomial directly from the functional equation. This is the proper point at
which to proceed internally: the regular sequence has only been constructed
through degree \(\nu-1\), and no formula lying beyond that range need be
imported.

Put
\[
\varrho=q^{1/2},
\quad
\varepsilon=\varrho^\nu\in\{1,-1\}.
\]
In the variable
\[
x=C+\kappa(t+t^{-1}),
\]
the two half-step neighbours are
\[
Y=C+\kappa(\varrho t+\varrho^{-1}t^{-1}),
\quad
Z=C+\kappa(\varrho^{-1}t+\varrho t^{-1}).
\]
The expression
\[
\Omega_\nu(x)=\kappa^\nu(t^\nu+t^{-\nu})
\]
is a monic polynomial of degree \(\nu\) in \(x\), and
\[
\Omega_\nu(Y)=\Omega_\nu(Z)=\varepsilon\Omega_\nu(X).
\]
Consequently,
\[
D(\Omega_\nu p)=\varepsilon\Omega_\nu Dp,
\quad
S(\Omega_\nu p)=\varepsilon\Omega_\nu Sp.
\]
Write
\[
L_{a,b,c,d}p=\phi Dp+\psi Sp.
\]
The functional equation is precisely
\[
\langle\mathbf u,L_{a,b,c,d}p\rangle=0,
\quad
p\in\mathcal P.
\]
The preceding identities show that
\[
L_{a,b,c,d}(\Omega_\nu p)
=
\varepsilon\Omega_\nu L_{a,b,c,d}p,
\]
and hence that \(\Omega_\nu\mathbf u\) satisfies the same functional equation
as \(\mathbf u\).

There is no hidden question of uniqueness here. Since
\[
g\notin\{q^k:\ k=0,1,\dots,\nu-1\},
\]
the displayed formula for \(d_j\) gives
\[
d_j\neq0,
\quad
j\in\mathbb N.
\]
The triangular moment system therefore determines every solution from its
zeroth moment. Its solution space is one-dimensional, and there exists a
scalar \(\varkappa_\nu\) such that
\[
\Omega_\nu\mathbf u=\varkappa_\nu\mathbf u.
\]
It remains only to determine that scalar.

For \(\zeta\in\mathbb C\), set
\[
\chi_\zeta(x)
=
1-\frac{\zeta}{\kappa}(x-C)+\zeta^2
=
(1-\zeta t)(1-\zeta/t).
\]
A direct factorisation of the displayed pair \((\phi,\psi)\) gives
\begin{align*}
\phi+\frac{Y-Z}{2}\psi
&=
4\kappa^2t^{-2}
\prod_{\zeta\in\{a,b,c,d\}}(1-\zeta t),
\\[7pt]
\phi-\frac{Y-Z}{2}\psi
&=
4\kappa^2t^2
\prod_{\zeta\in\{a,b,c,d\}}(1-\zeta/t).
\end{align*}
These two identities yield the exact intertwining relation
\[
L_{a,b,c,d}(\chi_{a\varrho}p)
=
\chi_aL_{aq,b,c,d}p.
\]
Thus, if
\[
\Xi_m(x)
=
(at,a/t;q)_m
=
\prod_{j=0}^{m-1}\chi_{aq^j}(x),
\quad
\mathbf u_m=\Xi_m\mathbf u,
\]
then \(\mathbf u_m\) satisfies the functional equation obtained by replacing
\(a\) with \(aq^m\).

Set
\[
F_m=\langle\mathbf u,\Xi_m\rangle
=\langle\mathbf u_m,1\rangle.
\]
Testing the shifted functional equation against the constant polynomial \(1\)
gives
\[
\frac{\langle\mathbf u_m,x-C\rangle}{\kappa}
=
\frac{
aq^m+b+c+d-abcq^m-abdq^m-acdq^m-bcd
}{
1-gq^m
}
F_m.
\]
On the other hand,
\[
F_{m+1}
=
(1+a^2q^{2m})F_m
-\frac{aq^m}{\kappa}\langle\mathbf u_m,x-C\rangle.
\]
Combining the last two identities, we obtain
\[
(1-gq^m)F_{m+1}
=
(1-abq^m)(1-acq^m)(1-adq^m)F_m.
\]
Multiplication over one complete \(q\)-cycle therefore gives
\[
\frac{F_\nu}{F_0}
=
\frac{
(1-(ab)^\nu)(1-(ac)^\nu)(1-(ad)^\nu)
}{
1-g^\nu
}.
\]
Moreover,
\[
\Xi_\nu(x)
=
1-a^\nu(t^\nu+t^{-\nu})+a^{2\nu}.
\]
It follows that
\[
\frac{\langle\mathbf u,t^\nu+t^{-\nu}\rangle}{F_0}
=
\frac{1+a^{2\nu}-F_\nu/F_0}{a^\nu}
=
\mathcal E_\nu,
\]
where
\[
\mathcal E_\nu=
\frac{
a^\nu+b^\nu+c^\nu+d^\nu
-(abc)^\nu-(abd)^\nu-(acd)^\nu-(bcd)^\nu
}{
1-g^\nu
}.
\]
Pairing
\(
\Omega_\nu\mathbf u=\varkappa_\nu\mathbf u
\)
with \(1\) now gives
\[
\varkappa_\nu=\kappa^\nu\mathcal E_\nu.
\]
Hence
\[
P_\nu(x)
=
\Omega_\nu(x)-\varkappa_\nu
=
\kappa^\nu
\bigl(t^\nu+t^{-\nu}-\mathcal E_\nu\bigr)
\]
satisfies
\[
P_\nu\mathbf u=0.
\]
Thus \(P_\nu\) is the unique monic polynomial of degree \(\nu\) orthogonal to
\(\mathcal P_{\nu-1}\); indeed, it annihilates the functional after
multiplication by any polynomial. In particular,
\[
\langle\mathbf u,P_\nu^2\rangle=0.
\]
This is the precise meaning of the terminal identity \(C_\nu=0\): it is a
vanishing terminal norm coefficient, not an invocation of the regular
recurrence beyond degree \(\nu-1\).

Assume, in addition, that
\[
\mathcal E_\nu\neq\pm2.
\]
Then the zeros of \(P_\nu\) are simple. More precisely, if \(r\) is chosen so
that
\[
r^\nu=\frac{\mathcal E_\nu}{2}+\sqrt{\frac{\mathcal E_\nu^2}{4}-1},
\]
then the zeros of \(P_\nu\) are
\[
\xi_s=
C+\kappa(rq^s+r^{-1}q^{-s}),
\]
for every \(s=0,1,\dots,\nu-1\).
We have already proved that
\[
P_\nu\mathbf u=0.
\]
Since \(\mathbf u\) is regular of order \(\nu\), the hypotheses of
Proposition~\ref{prop:simple-zeros-discrete-orthogonality} are therefore
satisfied with \(N=\nu-1\).
Writing
\[
h_n=\langle \mathbf u,P_n^2\rangle,
\]
for \(n=0,1,\dots,\nu-1\), one obtains
\[
\sum_{s=0}^{\nu-1}\lambda_s\,P_n(\xi_s)P_m(\xi_s)
=
h_n\,\delta_{nm}
\]
for every \(0\le n,m\le \nu-1\), where
\[
\lambda_s=
\frac{h_{\nu-1}}{P_{\nu-1}(\xi_s)\,P_\nu'(\xi_s)}.
\]
To compute \(P_\nu'(\xi_s)\), differentiate the identity
\[
P_\nu(x)=\kappa^\nu\bigl(t^\nu+t^{-\nu}-\mathcal E_\nu\bigr),
\quad
x=C+\kappa(t+t^{-1}).
\]
One has
\[
P_\nu'(x)
=
\frac{dP_\nu/dt}{dx/dt}
=
\nu\kappa^{\nu-1}\frac{t^\nu-t^{-\nu}}{t-t^{-1}}.
\]
Evaluating at \(t=rq^s\), one finds
\[
P_\nu'(\xi_s)
=
\nu\kappa^{\nu-1}\frac{r^\nu-r^{-\nu}}{rq^s-r^{-1}q^{-s}}.
\]
The terminal identity also extends the preceding quadrature formula to an
identity of functionals. Indeed, division by \(P_\nu\) shows that
\[
\langle\mathbf u,p\rangle
=
\sum_{s=0}^{\nu-1}\lambda_s p(\xi_s),
\quad
p\in\mathcal P,
\]
because both sides annihilate every multiple of \(P_\nu\), while they agree on
polynomials of degree at most \(\nu-1\).

We may therefore substitute this functional representation directly into the
functional equation. Put
\[
t_s=rq^s,
\]
and define the intermediate half-step points
\[
y_s
=
C+\kappa(\varrho t_s+\varrho^{-1}t_s^{-1}),
\quad
z_s
=
C+\kappa(\varrho^{-1}t_s+\varrho t_s^{-1}).
\]
The indices are read modulo \(\nu\), and
\[
y_s=z_{s+1}.
\]
The hypothesis
\(
\mathcal E_\nu\neq\pm2
\)
also ensures that these intermediate points are distinct. Using the two
factorisations obtained above and collecting the coefficient of \(p(y_s)\) in
\[
\sum_{j=0}^{\nu-1}\lambda_jL_{a,b,c,d}p(\xi_j)=0,
\]
we obtain
\[
\frac{\lambda_sR_+(t_s)}{y_s-z_s}
=
\frac{\lambda_{s+1}R_-(t_{s+1})}{y_{s+1}-z_{s+1}},
\]
where
\[
R_+(t)
=
4\kappa^2t^{-2}
\prod_{\zeta\in\{a,b,c,d\}}(1-\zeta t),
\quad
R_-(t)
=
4\kappa^2t^2
\prod_{\zeta\in\{a,b,c,d\}}(1-\zeta/t).
\]
The Christoffel formula above also gives \(\lambda_s\neq0\): the terminal
zeros are simple, and consecutive polynomials cannot have a common zero while
the recurrence coefficients \(C_1,\dots,C_{\nu-1}\) are non-zero. No division
by a vanishing factor is therefore concealed here. Indeed, put
\[
\mathsf a=a^\nu,
\quad
\mathsf b=b^\nu,
\quad
\mathsf c=c^\nu,
\quad
\mathsf d=d^\nu,
\quad
R=r^\nu.
\]
If, for instance, \(R=\mathsf a\) or \(R=\mathsf a^{-1}\), then the defining
identity for \(\mathcal E_\nu=R+R^{-1}\), after multiplication by
\(\mathsf a\) and collection of factors, gives
\[
0
=
(1-\mathsf a\mathsf b)
(1-\mathsf a\mathsf c)
(1-\mathsf a\mathsf d).
\]
Since \(q\) is primitive of order \(\nu\), one of \(ab,ac,ad\) would then be
a power of \(q\), contrary to the hypotheses. The same argument applies to
\(b,c,d\). Hence
\[
R\notin
\{\zeta^\nu,\zeta^{-\nu}:\ \zeta\in\{a,b,c,d\}\}.
\]
Moreover, \(\mathcal E_\nu\neq\pm2\) gives \(R^2\neq1\). Raising a
potentially vanishing factor in the balance to the power \(\nu\) now proves
that every factor by which we divide below is non-zero.
Since
\[
y_s-z_s
=
\kappa(\varrho-\varrho^{-1})(t_s-t_s^{-1}),
\]
the preceding balance reduces to
\[
\frac{\lambda_{s+1}}{\lambda_s}
=
\frac{q}{g}
\frac{1-r^2q^{2s+2}}{1-r^2q^{2s}}
\frac{(1-arq^s)(1-brq^s)(1-crq^s)(1-drq^s)}
     {(1-rq^{s+1}/a)(1-rq^{s+1}/b)(1-rq^{s+1}/c)(1-rq^{s+1}/d)}.
\]
Iterating from \(s=0\), we find
\[
\lambda_s
=
\lambda_0
\left(\frac{q}{g}\right)^s
\frac{
(1-r^2q^{2s})(ar,br,cr,dr;q)_s
}{
(1-r^2)(qr/a,qr/b,qr/c,qr/d;q)_s
},
\]
for every \(s=0,1,\dots,\nu-1\). To see the cyclic compatibility without
leaving anything implicit, put
\(
\mathsf g=\mathsf a\mathsf b\mathsf c\mathsf d=g^\nu
\).
A direct expansion gives
\[
\prod_{\zeta\in\{\mathsf a,\mathsf b,\mathsf c,\mathsf d\}}
(1-\zeta R)
-
\prod_{\zeta\in\{\mathsf a,\mathsf b,\mathsf c,\mathsf d\}}
(\zeta-R)
=
(R^2-1)(\mathsf g-1)
(R^2+1-\mathcal E_\nu R)
=0.
\]
The product of the displayed weight ratios over one complete cycle is
therefore, since the factor involving \(r^2\) telescopes and
\(q^\nu=1\),
\[
\left(\frac qg\right)^\nu
\frac{
\prod_{\zeta\in\{\mathsf a,\mathsf b,\mathsf c,\mathsf d\}}(1-\zeta R)
}{
\prod_{\zeta\in\{\mathsf a,\mathsf b,\mathsf c,\mathsf d\}}(1-R/\zeta)
}
=
\frac{
\prod_{\zeta\in\{\mathsf a,\mathsf b,\mathsf c,\mathsf d\}}(1-\zeta R)
}{
\prod_{\zeta\in\{\mathsf a,\mathsf b,\mathsf c,\mathsf d\}}(\zeta-R)
}
=1.
\]
Thus the finite root-of-unity orthogonality formula appears here as a
special instance of the general framework developed above, subject to the
additional requirement that the truncating polynomial have simple zeros.
\end{example}

\begin{remark}\label{AWremark}
The explicit support and weights obtained above may naturally be compared with
the root-of-unity formulas in \cite{SZ97} and with their
\(q\)-ultraspherical specialisation in \cite{SZ96}. The direction of
the present argument should nevertheless be kept clear. We have not identified
the functional by importing a previously known relation between weights. The
terminal identity
\[
P_\nu\mathbf u=0
\]
was derived from the functional equation and the one-dimensional moment
recursion, while the ratio of successive weights followed from substituting
the resulting functional representation into that same equation.

The cited works provide particular representations under additional
family-specific assumptions. Here the finite regular family is obtained first,
the terminal polynomial is then determined internally, and only afterwards do
its simple zeros yield the discrete functional representation. This order
separates the structural mechanism from the explicit evaluation of the support
and its Christoffel numbers, and shows precisely which part of the
root-of-unity formula belongs to the general theory.
\end{remark}
\section{Normalised alternating map}

Although Definition~\ref{def:classical} is formally uniform across all admissible map types, its degenerate subcases behave differently in the three regimes. In the infinite and finite non-resonant ordinary quadratic and \(q\)-exponential settings considered above, regularity forces \(\deg\psi=1\), whereas in the alternating case additional degenerate phenomena may still occur. We now turn to a distinguished globally normalised subcase of
case~\emph{(ii)} in
Theorem~\ref{thm:admissible-lattice-classification-U}, namely the alternating
map
\[
X(s)=e^{\pi i s/h}.
\]
This map is not merely formal: as observed above, there are natural
half-step-invariant sets on which it is admissible. In practice, such sets often
arise in simple forms. For instance, in
Example~\ref{ex:U-examples}\emph{(3)}, take
\(
E=h(\mathbb Q\cap[0,\frac12)).
\)
Then one obtains
\[
U=\bigsqcup_{v\in E}\left(v+\frac h2\mathbb Z\right)=h\mathbb Q.
\]
In this case
all representatives lie on the affine line \(h\mathbb R\), although they still determine infinitely many distinct
cosets of \(\tfrac h2\mathbb Z\). For a fixed \(s_0\in E\), on the corresponding full-step arithmetic progression
\(
s_0+h\mathbb Z,
\)
one has
\begin{align*}
X(s_0+kh)
=
e^{\pi i(s_0+kh)/h}=
(-1)^kX(s_0),
\end{align*}
which is precisely the alternating form in
Theorem~\ref{thm:admissible-lattice-classification-U}\emph{(ii)}, with
\[
a_{s_0}=X(s_0), \quad b=0.
\]
Thus the coefficient multiplying \((-1)^k\) is allowed to depend on the chosen
full-step progression, exactly as the theorem states. In the present example,
this dependence is genuinely nontrivial. Indeed, if \(s_0,s_0'\in E\) and
\(X(s_0)=X(s_0')\), then
\[
e^{\pi i(s_0-s_0')/h}=1,
\]
so that
\[
\frac{s_0-s_0'}{h}\in 2\mathbb Z.
\]
But, since \(s_0,s_0'\in E=h\bigl(\mathbb Q\cap[0,\tfrac12)\bigr)\), one has
\[
\frac{s_0-s_0'}{h}\in \mathbb Q\cap\left(-\frac12,\frac12\right),
\]
and the only even integer in that interval is \(0\). Hence \(s_0=s_0'\).
Therefore the map \(s_0\mapsto X(s_0)\) is injective on \(E\). Since \(E\) is
infinite, it follows that \(X(E)\), and therefore also \(X(h\mathbb Q)\),
is infinite.

We shall not pursue here the full range of alternating admissible maps allowed
by Theorem~\ref{thm:admissible-lattice-classification-U}. The normalised case
already contains the structural mechanism that distinguishes the alternating
regime from the quadratic and \(q\)-exponential ones. Accordingly, throughout
the remainder of this section, let \(h\in\mathbb C^\times\), let
\(U\subseteq\mathbb C\) be half-step-invariant, and let
\(
X:U\to\mathbb C
\)
be an admissible normalised alternating map on \(U\). Then, for every
polynomial \(p\),
\begin{align*}
(Dp)(X)
=
\frac{p\!\left(iX\right)-p\!\left(-\,iX\right)}{2i\,X},\quad
(Sp)(X)
=
\frac{p\!\left(iX\right)+p\!\left(-\,iX\right)}{2},
\end{align*}
throughout \(U\). This is the key structural feature of the normalised alternating case. In the quadratic and \(q\)-exponential regimes, the operators \(D\) and \(S\) are governed by neighbouring values of the parametrising map within the ordinary polynomial calculus. Here, by contrast, the half-step neighbours are simply \(iX\) and \(-iX\), and the corresponding calculus factors through the decomposition of \(\mathcal P\) induced by the quadratic substitution. For this reason, the alternating case is better treated as a separate structural regime rather than as a naive \(q\to -1\) limit.

The following proposition makes this decomposition explicit.

\begin{proposition}\label{prop:alternating-DS}
Fix \(h\in\mathbb C^\times\), let \(U\subseteq\mathbb C\) be half-step-invariant, and let
\(
X:U\to\mathbb C
\)
be the admissible map on \(U\) defined by
\[
X(s)=e^{\pi i s/h}
\]
for every \(s\in U\). Let
\(
D,S:\mathcal P\to\mathcal P
\)
be the associated divided-difference and averaging operators. Then every
polynomial \(p\) admits a unique decomposition
\[
p(x)=a(-x^2)+x\,b(-x^2),
\]
where \(a\) and \(b\) are polynomials. For these uniquely determined
polynomials \(a\) and \(b\), one has
\[
(Dp)(X)=b(X^2),
\quad
(Sp)(X)=a(X^2)
\]
throughout \(U\).
\end{proposition}

\begin{proof}
Fix \(p\in\mathcal P\). By decomposition into even and odd parts, there exist
unique polynomials \(a\) and \(b\) such that
\(
p(x)=a(-x^2)+x\,b(-x^2).
\)
Evaluating at \(x=X\), one obtains
\[
p(X)=a(-X^2)+X\,b(-X^2)
\]
throughout \(U\). Now
\[
p(iX)=a(X^2)+iX\,b(X^2),
\]
since \((iX)^2=-X^2\). Likewise,
\[
p(-iX)=a(X^2)-iX\,b(X^2).
\]
Hence
\[
p(iX)-p(-iX)=2iX\,b(X^2),
\]
and therefore
\[
(Dp)(X)
=
b(X^2).
\]
Similarly,
\[
(Sp)(X)=a(X^2).
\]
This proves the stated formulas.
\end{proof}

\begin{definition}[Quadratic substitution operator]\label{def:alternating-sigma}
The quadratic substitution operator is the linear map
\(
\sigma:\mathcal P\to\mathcal P
\)
defined by
\[
(\sigma p)(x)=p(x^2)
\]
for every \(p\in\mathcal P\). The transpose of \(\sigma\) is denoted by
\(
\boldsymbol{\sigma}:\mathcal P'\to\mathcal P'
\)
and defined by
\[
\langle \boldsymbol{\sigma}\mathbf u,p\rangle
=
\langle \mathbf u,\sigma p\rangle
\]
for every \(p\in\mathcal P\) and every \(\mathbf u\in\mathcal P'\).
\end{definition}

The formulas in Proposition~\ref{prop:alternating-DS} also show, without any
further calculation, that
\[
D^2=0,
\quad
DS=0,
\quad
\ker D=\sigma(\mathcal P),
\quad
\ker S=x\,\sigma(\mathcal P),
\quad
D(\mathcal P)=S(\mathcal P)=\sigma(\mathcal P).
\]
Thus the failure of surjectivity in the alternating regime is not a secondary
technical inconvenience: it is already encoded in the algebra of the two
operators.

The operator \(\sigma\) is continuous on \(\mathcal P\). Indeed, for every
\(n\in\mathbb N\) one has
\(
\sigma(\mathcal P_n)\subseteq \mathcal P_{2n}.
\)
Since each restriction
\[
\sigma|_{\mathcal P_n}:\mathcal P_n\to\mathcal P_{2n}
\]
is linear between finite-dimensional spaces, it is continuous. By the
continuity criterion for the inductive-limit topology on \(\mathcal P\),
\(\sigma\) is therefore continuous. Hence its transpose
\(\boldsymbol{\sigma}\) is well defined and \(\sigma(\mathcal P',\mathcal P)\)-continuous.

Proposition~\ref{prop:alternating-DS} shows that \(\mathcal P\) splits,
in the alternating case, according to the quadratic substitution. The operators
\(S\) and \(D\) detect the two components of the even--odd decomposition written
in the form
\[
p(x)=a(-x^2)+x\,b(-x^2).
\]
More generally, for every fixed \(\tau\in\mathbb C\), one also has the
direct-sum decomposition
\[
\mathcal P=\sigma(\mathcal P)\oplus (x-\tau)\sigma(\mathcal P).
\]
Thus \(\sigma\) records the passage from the original polynomial variable
\(x\) to the quadratic variable \(x^2\), while its transpose
\(\boldsymbol{\sigma}\) records the corresponding quadratic component on the
dual side. We now translate this decomposition into the corresponding dual
statement. The next proposition makes the resulting annihilation conditions
explicit.

\begin{proposition}\label{prop:alternating-structural-equation}
Fix \(h\in\mathbb C^\times\), let \(U\subseteq\mathbb C\) be half-step-invariant,
and let
\(
X:U\to\mathbb C
\)
be the admissible map on \(U\) defined by
\[
X(s)=e^{\pi i s/h}
\]
for every \(s\in U\). Let
\(
D,S:\mathcal P\to\mathcal P
\)
be the associated divided-difference and averaging operators, with transposes
\(
\mathbf D,\mathbf S:\mathcal P'\to\mathcal P'.
\)
Let \(\mathbf u\in\mathcal P'\), and let \(\phi\) and \(\psi\) be polynomials,
not both identically zero, such that \(\phi\) has degree at most \(2\) and
\(\psi\) has degree at most \(1\). Then the following conditions are equivalent:
\begin{itemize}
\item[\emph{(i)}]
\[
\mathbf D(\phi\,\mathbf u)=\mathbf S(\psi\,\mathbf u)
\]
in \(\mathcal P'\).

\item[\emph{(ii)}]
\begin{align*}
\langle \mathbf u,\phi(x)\,p(x^2)\rangle=0,\quad
\langle \mathbf u,\psi(x)\,p(x^2)\rangle=0,
\end{align*}
for every \(p\in\mathcal P\).
\end{itemize}
\end{proposition}
\begin{proof}
Assume \emph{(i)}, and let \(p\in\mathcal P\). First take
\[
p_o(x)=x\,p(-x^2).
\]
By Proposition~\ref{prop:alternating-DS}, the decomposition of \(p_o\) is obtained
with even part equal to \(0\) and odd part equal to \(p\). Hence
\[
(Dp_o)(X)=p(X^2),
\quad
(Sp_o)(X)=0
\]
throughout \(U\). Since \(X\) is admissible, \(X(U)\) is infinite. Hence, as both sides are
polynomials in the free variable \(x\), these identities imply
\[
Dp_o(x)=p(x^2),
\quad
Sp_o(x)=0
\]
as polynomial identities. Therefore
\[
0
=
\bigl\langle \mathbf D(\phi\,\mathbf u)-\mathbf S(\psi\,\mathbf u),p_o\bigr\rangle
=
-\langle \mathbf u,\phi(x)\,p(x^2)\rangle.
\]
Thus
\[
\langle \mathbf u,\phi(x)\,p(x^2)\rangle=0
\]
for every \(p\in\mathcal P\). Next take
\[
p_e(x)=p(-x^2).
\]
By Proposition~\ref{prop:alternating-DS}, the decomposition of \(p_e\) has odd
part equal to \(0\) and even part equal to \(p\). Hence
\[
(Dp_e)(X)=0,
\quad
(Sp_e)(X)=p(X^2)
\]
throughout \(U\). Again, since \(X\) is admissible, \(X(U)\) is infinite, and therefore these are
polynomial identities:
\[
Dp_e(x)=0,
\quad
Sp_e(x)=p(x^2).
\]
Hence
\begin{align*}
0=
\bigl\langle \mathbf D(\phi\,\mathbf u)-\mathbf S(\psi\,\mathbf u),p_e\bigr\rangle=
-\langle \mathbf u,\psi(x)\,p(x^2)\rangle.
\end{align*}
Thus \emph{(ii)} follows.

Conversely, assume \emph{(ii)}, and let \(q\in\mathcal P\). By
Proposition~\ref{prop:alternating-DS}, there exist polynomials \(a\) and \(b\)
such that
\[
(Dq)(X)=b(X^2),
\quad
(Sq)(X)=a(X^2)
\]
throughout \(U\). Since \(X\) is admissible, \(X(U)\) is infinite, and this gives the polynomial
identities
\[
Dq(x)=b(x^2),
\quad
Sq(x)=a(x^2).
\]
Using \emph{(ii)}, we obtain
\[
\langle \mathbf u,\phi(x)Dq(x)\rangle=0,
\quad
\langle \mathbf u,\psi(x)Sq(x)\rangle=0.
\]
Therefore
\begin{align*}
\bigl\langle \mathbf D(\phi\,\mathbf u)-\mathbf S(\psi\,\mathbf u),q\bigr\rangle=
-\langle \mathbf u,\phi(x)Dq(x)\rangle
-
\langle \mathbf u,\psi(x)Sq(x)\rangle=
0.
\end{align*}
Since this holds for every \(q\in\mathcal P\), the two functionals agree on
\(\mathcal P\). Hence \emph{(i)} follows.
\end{proof}

Before turning to the non-degenerate case \(\deg\psi=1\), we record a phenomenon peculiar to the alternating map. In the infinite and finite non-resonant ordinary quadratic and \(q\)-exponential regimes treated above, the regularity
criteria force
\[
d_0=\psi_1\neq0,
\]
and hence \(\deg\psi=1\). Thus the cases \(\psi=0\) and \(\psi\) constant are
automatically excluded there in the regular situation. For the normalised
alternating map this is no longer automatic. By
Proposition~\ref{prop:alternating-structural-equation}, the structural equation
is equivalent to
\[
\langle \mathbf u,\phi(x)\,p(x^2)\rangle=0,
\quad
\langle \mathbf u,\psi(x)\,p(x^2)\rangle=0
\]
for every \(p\in\mathcal P\).

If \(\psi=0\), the condition reduces to
\[
\langle \mathbf u,\phi(x)\,p(x^2)\rangle=0,
\]
that is, \(\mathbf u\) vanishes on
\(
\phi\,\sigma(\mathcal P).
\)
Writing the even and odd components of \(\mathbf u\) as
\[
\langle \mathbf v,p\rangle=\langle \mathbf u,p(x^2)\rangle,
\quad
\langle \mathbf w,p\rangle=\langle \mathbf u,xp(x^2)\rangle,
\]
and writing
\[
\phi(x)=\phi_2x^2+\phi_1x+\phi_0,
\]
this is equivalent to
\[
\phi_1\mathbf w=-(\phi_2x+\phi_0)\mathbf v
\]
in \(\mathcal P'\). Hence, if \(\phi_1\neq0\), the odd component
\(\mathbf w\) is determined by the even component \(\mathbf v\). If
\(\phi_1=0\), the condition reduces to
\[
(\phi_2x+\phi_0)\mathbf v=0
\]
and imposes no restriction on \(\mathbf w\). Thus the case \(\psi=0\) is not empty; it may contain regular functionals, but
it is governed only by this relation between the even and odd components.
Concrete examples can be produced, although they are not especially revealing,
since in the alternating setting the same regular functional may also satisfy
the equation for non-zero choices of \(\psi\). We do not pursue this degenerate case further, nor do we modify the definition of classicality so as to exclude it, since that would artificially remove part of the alternating phenomenon.

If \(\psi\) is a non-zero constant, then the second annihilation condition gives
\[
\langle \mathbf u,p(x^2)\rangle=0
\]
for every \(p\in\mathcal P\). In particular,
\(
\langle \mathbf u,1\rangle=0,
\)
so this case is incompatible with regularity.

We now pass to the non-degenerate case. Thus \(\psi\) is assumed to have degree
\(1\), and, after multiplication of the structural equation by a non-zero
constant, we may write
\[
\psi(x)=x-\tau
\]
for a uniquely determined \(\tau\in\mathbb C\). This is the case in which the
quadratic substitution yields the reconstruction theorem below.

\begin{definition}[Alternating splitting operator]\label{def:alternating-Jtau}
Fix \(\tau\in\mathbb C\). The alternating splitting operator
\(
J_\tau:\mathcal P\to\mathcal P
\)
is defined by
\(
J_\tau p=a,
\)
where \(a\) is determined by the unique decomposition
\[
p(x)=a(x^2)+(x-\tau)b(x^2)
\]
of \(p\in\mathcal P\). The transpose of \(J_\tau\) is denoted by
\(
\mathbf J_\tau:\mathcal P'\to\mathcal P'
\)
and defined by
\[
\langle \mathbf J_\tau\mathbf v,p\rangle
=
\langle \mathbf v,J_\tau p\rangle,
\]
for every \(p\in\mathcal P\) and every \(\mathbf v\in\mathcal P'\).
\end{definition}

The operator \(J_\tau\) is continuous on \(\mathcal P\). Indeed, if \(p\in\mathcal P_n\), then in the decomposition
\(
p(x)=a(x^2)+(x-\tau)b(x^2)
\)
one has
\(
\deg a\le \left\lfloor \frac n2\right\rfloor,
\)
and, for every \(n\in \mathbb N^\times\),
\(
\deg b\le \left\lfloor \frac{n-1}{2}\right\rfloor,
\)
while \(b=0\) when \(n=0\). Hence
\[
J_\tau(\mathcal P_n)\subseteq \mathcal P_{\lfloor n/2\rfloor}.
\]
The restrictions
\(
J_\tau|_{\mathcal P_n}:\mathcal P_n\to\mathcal P_{\lfloor n/2\rfloor}
\)
are linear maps between finite-dimensional spaces, and therefore continuous.
By the continuity criterion for the inductive-limit topology on \(\mathcal P\),
\(J_\tau\) is continuous. Consequently, its transpose
\(\mathbf J_\tau\) is well defined and continuous for the weak topology
\(\sigma(\mathcal P',\mathcal P)\).

\begin{proposition}\label{prop:alternating-construction}
Fix \(\tau\in\mathbb C\), and let \(\mathbf v\in\mathcal P'\). If
\(
\mathbf u=\mathbf J_\tau\mathbf v,
\)
then
\[
\langle \mathbf u,(x-\tau)r(x^2)\rangle=0
\]
for every \(r\in\mathcal P\), and
\(
\boldsymbol{\sigma}\mathbf u=\mathbf v.
\)
Conversely, if \(\mathbf u\in\mathcal P'\) satisfies
\[
\langle \mathbf u,(x-\tau)r(x^2)\rangle=0
\]
for every \(r\in\mathcal P\), then
\(
\mathbf u=\mathbf J_\tau(\boldsymbol{\sigma}\mathbf u).
\)
If, in addition,
\(
\langle \mathbf u,1\rangle\neq0,
\)
then \(\tau\) is uniquely determined by \(\mathbf u\), and
\[
\tau=\frac{\langle \mathbf u,x\rangle}{\langle \mathbf u,1\rangle}.
\]
\end{proposition}

\begin{proof}
The decomposition
\[
p(x)=a(x^2)+(x-\tau)b(x^2)
\]
is unique for every \(p\in\mathcal P\), so the definition of
\(\mathbf J_\tau\mathbf v\) is well posed. Let \(\mathbf v\in\mathcal P'\), and set
\(\mathbf u=\mathbf J_\tau\mathbf v\). Applying the definition to polynomials of the form
\[
p(x)=(x-\tau)r(x^2),
\]
for which \(a=0\), gives
\[
\langle \mathbf u,(x-\tau)r(x^2)\rangle=0
\]
for every \(r\in\mathcal P\). Applying it to polynomials of the form
\[
p(x)=r(x^2),
\]
for which \(a=r\), gives
\[
\langle \mathbf u,r(x^2)\rangle=\langle \mathbf v,r\rangle
\]
for every \(r\in\mathcal P\). Hence
\(
\boldsymbol{\sigma}\mathbf u=\mathbf v.
\)

Conversely, suppose that \(\mathbf u\in\mathcal P'\) satisfies
\[
\langle \mathbf u,(x-\tau)r(x^2)\rangle=0
\]
for every \(r\in\mathcal P\). If
\[
p(x)=a(x^2)+(x-\tau)b(x^2),
\]
then
\[
\langle \mathbf u,p\rangle
=
\langle \mathbf u,a(x^2)\rangle
=
\langle \boldsymbol{\sigma}\mathbf u,a\rangle.
\]
By the definition of \(\mathbf J_\tau\), this is precisely
\(
\mathbf u=\mathbf J_\tau(\boldsymbol{\sigma}\mathbf u).
\)
Finally, taking \(r=1\) gives
\(
\langle \mathbf u,x-\tau\rangle=0.
\)
If \(\langle \mathbf u,1\rangle\neq0\), then
\[
\tau=\frac{\langle \mathbf u,x\rangle}{\langle \mathbf u,1\rangle},
\]
so \(\tau\) is uniquely determined by \(\mathbf u\).
\end{proof}

The preceding proposition shows that, once the annihilation condition
\[
\langle \mathbf u,(x-\tau)r(x^2)\rangle=0
\]
is imposed, the functional \(\mathbf u\) is completely recovered from its
quadratic component \(\boldsymbol{\sigma}\mathbf u\). Thus, in the
non-degenerate alternating case, the orthogonality problem is reduced to an
orthogonality problem in the quadratic variable. The next theorem gives the
precise reconstruction.

\begin{theorem}\label{thm:first-order-alternating}
Fix \(h\in\mathbb C^\times\), let \(U\subseteq\mathbb C\) be half-step-invariant, and let
\(
X:U\to\mathbb C
\)
be the admissible map on \(U\) defined by
\[
X(s)=e^{\pi i s/h}
\]
for every \(s\in U\). Let
\(
D,S:\mathcal P\to\mathcal P
\)
be the associated divided-difference and averaging operators, with transposes
\(
\mathbf D,\mathbf S:\mathcal P'\to\mathcal P'.
\)
Let \(\mathbf u\in\mathcal P'\) be such that
\(
\langle \mathbf u,1\rangle\neq0.
\)

Then the following statements are equivalent:
\begin{itemize}
\item[\emph{(i)}]
There exist polynomials \(\phi\) and \(\psi\), where \(\phi\) has degree at most
\(2\) and \(\psi\) has degree \(1\), such that
\[
\mathbf D(\phi\,\mathbf u)=\mathbf S(\psi\,\mathbf u)
\]
in \(\mathcal P'\).

\item[\emph{(ii)}]
There exists \(\tau\in\mathbb C\) such that
\[
\langle \mathbf u,(x-\tau)r(x^2)\rangle=0
\]
for every \(r\in\mathcal P\).

\item[\emph{(iii)}]
There exist \(\tau\in\mathbb C\) and \(\mathbf v\in\mathcal P'\) such that
\[
\mathbf u=\mathbf J_\tau\mathbf v.
\]
\end{itemize}
In that case, \(\tau\) and \(\mathbf v\) are uniquely determined by \(\mathbf u\), and are given by
\[
\tau=\frac{\langle \mathbf u,x\rangle}{\langle \mathbf u,1\rangle},
\quad
\mathbf v=\boldsymbol{\sigma}\mathbf u.
\]
Assume henceforth that these equivalent conditions hold. Let \(I\subseteq\mathbb N\) be either \(I=\mathbb N\), or
\(
I=\{0,1,\dots,2N+2\}
\)
for some \(N\in\mathbb N\). In the finite case set
\[
J_I=\{0,1,\dots,N+1\},
\quad
K_I=\{0,1,\dots,N\},
\]
whereas in the infinite case set
\[
J_I=K_I=\mathbb N.
\]

Then the following statements are equivalent:
\begin{itemize}
\item[\emph{(iv)}]
There exists a monic orthogonal polynomial sequence
\(
(R_n)_{n\in J_I}
\)
with respect to \(\mathbf v\) such that
\[
R_n(\tau^2)\neq0
\]
for every \(n\in J_I\).

\item[\emph{(v)}]
There exists a monic orthogonal polynomial sequence
\(
(P_n)_{n\in I}
\)
with respect to \(\mathbf u\).
\end{itemize}

Whenever these conditions hold, the functional
\(
(x-\tau^2)\mathbf v
\)
admits a monic orthogonal polynomial sequence
\(
(S_n)_{n\in K_I}
\)
with respect to \((x-\tau^2)\mathbf v\), where, for every \(n\in K_I\),
\[
S_n(x)=
\frac{
R_{n+1}(x)-\dfrac{R_{n+1}(\tau^2)}{R_n(\tau^2)}\,R_n(x)
}{
x-\tau^2
}.
\]
Moreover, the monic orthogonal polynomial sequence with respect to \(\mathbf u\) is given, for every \(n\in K_I\), by
\[
P_{2n}(x)=R_n(x^2),
\quad
P_{2n+1}(x)=(x-\tau)S_n(x^2),
\]
together with
\[
P_{2N+2}(x)=R_{N+1}(x^2)
\]
in the finite case.
\end{theorem}

\begin{proof}
The implication \emph{(i)} \(\Rightarrow\) \emph{(ii)} follows immediately from
Proposition~\ref{prop:alternating-structural-equation}, since the condition
\(\deg\psi=1\) means precisely that
\[
\psi(x)=c(x-\tau)
\]
for some \(c\in\mathbb C^\times\) and some \(\tau\in\mathbb C\). Dividing by
\(c\), one obtains
\[
\langle \mathbf u,(x-\tau)r(x^2)\rangle=0
\]
for every \(r\in\mathcal P\). Conversely, if \emph{(ii)} holds, then
Proposition~\ref{prop:alternating-structural-equation} applies with
\[
\phi=0,
\quad
\psi=x-\tau,
\]
and yields \emph{(i)}. The equivalence of \emph{(ii)} and \emph{(iii)} is exactly
Proposition~\ref{prop:alternating-construction}. That proposition also gives
\[
\tau=\frac{\langle \mathbf u,x\rangle}{\langle \mathbf u,1\rangle},
\quad
\mathbf v=\boldsymbol{\sigma}\mathbf u,
\]
and hence the uniqueness of \(\tau\) and \(\mathbf v\). Assume henceforth that these equivalent conditions hold, and set
\(
a=\tau^2.
\)

We treat the finite case. Thus let
\(
I=\{0,1,\dots,2N+2\}
\)
for some \(N\in\mathbb N\). The case \(I=\mathbb N\) is obtained by the same argument, with the finite
index ranges removed; the precise point at which this extension is made will be
indicated at the end of the proof.

\medskip
\noindent{\em Step 1:}
Assume that \(\mathbf v\) admits a monic orthogonal polynomial sequence
\(
R_0,R_1,\dots,R_{N+1}
\)
with respect to \(\mathbf v\), and that
\[
R_n(a)\neq0
\]
for \(n\in\{0,1,\dots,N+1\}\). Since the distinguished factor in the original variable
is \(x-\tau\), the corresponding distinguished point in the quadratic variable
is
\(
a=\tau^2.
\)
Accordingly, the relevant transformed functional is \((x-a)\mathbf v\), and by
the Christoffel criterion for multiplication by \(x-a\), the functional
\(
(x-a)\mathbf v
\)
admits a monic orthogonal polynomial sequence
\(
S_0,S_1,\dots,S_N
\)
with respect to \((x-a)\mathbf v\). Moreover, for \(n=0,1,\dots,N\), one has
\[
S_n(x)=
\frac{
R_{n+1}(x)-\dfrac{R_{n+1}(a)}{R_n(a)}\,R_n(x)
}{
x-a
}.
\]

\medskip
\noindent{\em Step 2:}
Since \(\mathbf u=\mathbf J_\tau\mathbf v\),
Proposition~\ref{prop:alternating-construction} gives
\[
\langle \mathbf u,r(x^2)\rangle=\langle \mathbf v,r\rangle,
\quad
\langle \mathbf u,(x-\tau)r(x^2)\rangle=0
\]
for every \(r\in\mathcal P\). Moreover, for every \(r\in\mathcal P\),
\[
\langle \mathbf u,(x-\tau)^2r(x^2)\rangle
=
\langle \mathbf u,(x^2-a)r(x^2)\rangle
=
\langle (x-a)\mathbf v,r\rangle.
\]
Indeed,
\[
(x-\tau)^2-(x^2-a)=-2\tau(x-\tau),
\]
and therefore
\[
\bigl((x-\tau)^2-(x^2-a)\bigr)r(x^2)
=
-2\tau(x-\tau)r(x^2),
\]
whose pairing with \(\mathbf u\) vanishes by the defining annihilation
condition. For every \(n=0,1,\dots,N\), define
\[
P_{2n}(x)=R_n(x^2),
\quad
P_{2n+1}(x)=(x-\tau)S_n(x^2),
\]
and also set
\[
P_{2N+2}(x)=R_{N+1}(x^2).
\]
For \(0\le m,n\le N\), using the preceding identities, we obtain
\begin{align*}
\langle \mathbf u,P_{2n}P_{2m}\rangle
&=
\langle \mathbf v,R_nR_m\rangle,\\[7pt]
\langle \mathbf u,P_{2n}P_{2m+1}\rangle
&=0,\\[7pt]
\langle \mathbf u,P_{2n+1}P_{2m+1}\rangle
&=
\langle (x-a)\mathbf v,S_nS_m\rangle.
\end{align*}
Moreover, for \(m=0,1,\dots,N\),
\[
\langle \mathbf u,P_{2N+2}P_{2m+1}\rangle=0,
\]
and
\[
\langle \mathbf u,P_{2N+2}P_{2m}\rangle
=
\langle \mathbf v,R_{N+1}R_m\rangle=0.
\]
Hence
\[
\langle \mathbf u,P_nP_m\rangle=0
\]
whenever
\(
0\le m<n\le 2N+2.
\)
Moreover, for \(n=0,1,\dots,N\), one has
\[
\langle \mathbf u,P_{2n}^2\rangle
=
\langle \mathbf v,R_n^2\rangle\neq0,
\]
and
\[
\langle \mathbf u,P_{2n+1}^2\rangle
=
\langle (x-a)\mathbf v,S_n^2\rangle\neq0.
\]
Finally,
\[
\langle \mathbf u,P_{2N+2}^2\rangle
=
\langle \mathbf v,R_{N+1}^2\rangle\neq0.
\]
Thus
\(
P_0,P_1,\dots,P_{2N+2}
\)
is a monic orthogonal polynomial sequence with respect to \(\mathbf u\). This proves that
\emph{(iv)} implies \emph{(v)} in the finite case. The displayed formula for
\(S_n\) is already known from Step~1, and the displayed formulas for the
polynomials \(P_n\) hold by construction.

\medskip
\noindent{\em Step 3:}
Assume now that \(\mathbf u\) admits a monic orthogonal polynomial sequence
through degree \(2N+2\), denoted by \(P_0,P_1,\dots,P_{2N+2}\). Since this
sequence is monic and orthogonal, it follows that, for
each \(m=0,1,\dots,N+1\), the bilinear form induced by \(\mathbf u\) is
non-degenerate on \(\mathcal P_{2m}\), and, for each \(m=0,1,\dots,N\), it is
also non-degenerate on \(\mathcal P_{2m+1}\). Indeed, with respect to the basis
\(
P_0,P_1,\dots,P_{2m}
\)
(resp.\ \(P_0,P_1,\dots,P_{2m+1}\)), the Gram matrix of the bilinear form induced by
\(\mathbf u\) on \(\mathcal P_{2m}\) (resp.\ \(\mathcal P_{2m+1}\)) is diagonal with non-zero
diagonal entries. Now, for each \(m=0,1,\dots,N+1\), one has the orthogonal decomposition
\[
\mathcal P_{2m}
=
\sigma(\mathcal P_m)\oplus (x-\tau)\sigma(\mathcal P_{m-1}),
\]
with the convention that \(\sigma(\mathcal P_{-1})=\{0\}\), and for each
\(m=0,1,\dots,N\),
\[
\mathcal P_{2m+1}
=
\sigma(\mathcal P_m)\oplus (x-\tau)\sigma(\mathcal P_m).
\]
As above, these summands are orthogonal with respect to the bilinear form
induced by \(\mathbf u\). Since the whole form is non-degenerate on
\(\mathcal P_{2m}\) and \(\mathcal P_{2m+1}\), respectively, the restrictions to
the summands are non-degenerate as well.

The restriction of the bilinear form induced by \(\mathbf u\) to
\(\sigma(\mathcal P_m)\) is precisely the bilinear form induced by
\(\mathbf v\) on \(\mathcal P_m\). Hence, for every \(m=0,1,\dots,N+1\), that
form is non-degenerate on \(\mathcal P_m\). Likewise, the restriction of the
bilinear form induced by \(\mathbf u\) to \((x-\tau)\sigma(\mathcal P_m)\) is
precisely the bilinear form induced by \((x-a)\mathbf v\) on \(\mathcal P_m\),
where \(a=\tau^2\). Hence, for every \(m=0,1,\dots,N\), that form is
non-degenerate on \(\mathcal P_m\). Therefore \(\mathbf v\) is regular of order \(N+2\), while \((x-a)\mathbf v\) is
regular of order \(N+1\). Equivalently, \(\mathbf v\) admits a monic orthogonal
family
\(
R_0,R_1,\dots,R_{N+1}
\)
and \((x-a)\mathbf v\) admits a monic orthogonal polynomial sequence
\(
S_0,S_1,\dots,S_N.
\)
For each \(n=0,1,\dots,N+1\), the polynomial
\(
R_n(x^2)
\)
is monic of degree \(2n\). Let \(q\in\mathcal P\) have degree at most \(2n-1\). By the
direct-sum decomposition underlying the definition of \(J_\tau\), we may write
\[
q(x)=A(x^2)+(x-\tau)B(x^2),
\]
with \(\deg A\le n-1\). Hence
\[
\langle \mathbf u,R_n(x^2)q(x)\rangle
=
\langle \mathbf v,R_nA\rangle
=
0.
\]
Thus \(R_n(x^2)\) is orthogonal to every polynomial of degree at most \(2n-1\).
Hence, by uniqueness of the monic orthogonal polynomial of degree \(2n\),
\[
P_{2n}(x)=R_n(x^2)
\]
for \(n=0,1,\dots,N+1\).

\medskip
\noindent{\em Step 4:}
Let \(n=0,1,\dots,N\). Since \(P_{2n+1}\) is monic of degree \(2n+1\), its
unique decomposition associated with \(J_\tau\) has the form
\[
P_{2n+1}(x)=A_n(x^2)+(x-\tau)S_n(x^2),
\]
where \(A_n,S_n\in\mathcal P\). If \(Q\in\mathcal P\) has degree at most \(n\),
then \(\deg Q(x^2)\le2n\), so orthogonality gives
\begin{align*}
0=
\langle \mathbf u,P_{2n+1}(x)Q(x^2)\rangle=
\langle \mathbf v,A_nQ\rangle.
\end{align*}
Since
\(
R_0,R_1,\dots,R_n
\)
are orthogonal with respect to \(\mathbf v\), the bilinear form induced by
\(\mathbf v\) is non-degenerate on \(\mathcal P_n\). Hence \(A_n=0\). It follows that
\[
P_{2n+1}(x)=(x-\tau)S_n(x^2)
\]
for \(n=0,1,\dots,N\), with \(S_n\) monic.

\medskip
\noindent{\em Step 5:}
Let \(n=0,1,\dots,N\), and let \(Q\in\mathcal P\) have degree at most \(n-1\). Then
\[
\deg\bigl((x-\tau)Q(x^2)\bigr)\le2n-1,
\]
and therefore
\begin{align*}
0=
\langle \mathbf u,P_{2n+1}(x)(x-\tau)Q(x^2)\rangle=
\langle (x-a)\mathbf v,S_nQ\rangle.
\end{align*}
Thus \(S_n\) is orthogonal, with respect to \((x-a)\mathbf v\), to every
polynomial of degree at most \(n-1\). Moreover,
\[
0\neq \langle \mathbf u,P_{2n+1}^2\rangle
=
\langle (x-a)\mathbf v,S_n^2\rangle.
\]
Hence
\(
S_0,S_1,\dots,S_N
\)
is the monic orthogonal polynomial sequence with respect to \((x-a)\mathbf v\). In particular,
\((x-a)\mathbf v\) admits monic polynomials of degrees \(0,1,\dots,N\)
orthogonal with respect to it. Applying the Christoffel criterion, we conclude that
\[
R_n(a)\neq0
\]
for
\(
n\in\{0,1,\dots,N+1\},
\)
and that
\[
S_n(x)=
\frac{R_{n+1}(x)-\dfrac{R_{n+1}(a)}{R_n(a)}\,R_n(x)}{x-a}
\]
for
\(
n\in\{0,1,\dots,N\}.
\)
This proves that \emph{(v)} implies \emph{(iv)} in the finite case. This establishes the theorem in the case
\(
I=\{0,1,\dots,2N+2\}.
\)

\medskip
\noindent{\em Step 6:}
We now turn to the case
\(
I=\mathbb N.
\) Assume that \(\mathbf v\) admits a monic orthogonal polynomial sequence \((R_n)_{n\in\mathbb N}\) with respect to \(\mathbf v\). If
\[
R_n(\tau^2)\neq0
\]
for every \(n\in\mathbb N\), then for each fixed \(N\in\mathbb N\) the finite family
\(
R_0,R_1,\dots,R_{N+1}
\)
satisfies the finite case already proved. Hence there exist monic polynomials
\(
P_0,P_1,\dots,P_{2N+2}
\)
orthogonal with respect to \(\mathbf u\), where
\[
P_{2n}(x)=R_n(x^2),
\quad
P_{2n+1}(x)=(x-\tau)S_n(x^2)
\]
for \(n=0,1,\dots,N\), and
\[
P_{2N+2}(x)=R_{N+1}(x^2),
\]
with
\[
S_n(x)=
\frac{
R_{n+1}(x)-\dfrac{R_{n+1}(\tau^2)}{R_n(\tau^2)}\,R_n(x)
}{
x-\tau^2
}
\]
for \(n=0,1,\dots,N\). For each fixed \(n\), the formulas
\[
P_{2n}(x)=R_n(x^2),
\quad
P_{2n+1}(x)=(x-\tau)S_n(x^2),
\]
and
\[
S_n(x)=
\frac{
R_{n+1}(x)-\dfrac{R_{n+1}(\tau^2)}{R_n(\tau^2)}\,R_n(x)
}{
x-\tau^2
}
\]
do not depend on the truncation level \(N\). Hence the finite families obtained
for different values of \(N\) are compatible. Therefore there exist monic
polynomials \((P_n)_{n\in\mathbb N}\)
orthogonal with respect to \(\mathbf u\), and monic polynomials \((S_n)_{n\in\mathbb N}\) orthogonal with respect to \((x-\tau^2)\mathbf v\), with the stated formulas.
Thus \emph{(iv)} implies \emph{(v)} when
\(
I=\mathbb N.
\)

Conversely, assume that \(\mathbf u\) admits a monic orthogonal polynomial
sequence \((P_n)_{n\in\mathbb N}\). Then, for each \(N\in\mathbb N\), the finite
subfamily
\(
P_0,P_1,\dots,P_{2N+2}
\)
satisfies the finite case already proved. Hence, for each \(N\), the functional
\(\mathbf v\) is regular of order \(N+2\). Since \(N\) is arbitrary,
\(\mathbf v\) admits a monic orthogonal polynomial sequence
\(
(R_n)_{n\in\mathbb N}.
\)
By uniqueness of the monic orthogonal polynomial of each degree, the finite
polynomials obtained from the finite case coincide with the corresponding
initial segment of this global sequence. Applying the finite case once more, we
obtain
\[
R_n(\tau^2)\neq0
\]
for \(n=0,1,\dots,N+1\). Since \(N\) is arbitrary, it follows that
\[
R_n(\tau^2)\neq0
\]
for every \(n\in\mathbb N\). Thus \emph{(v)} implies \emph{(iv)} when
\(
I=\mathbb N.
\)
Finally, the formulas for \(S_n\), \(P_{2n}\), and \(P_{2n+1}\) in the case
\(
I=\mathbb N
\)
follow, for each fixed \(n\), by applying the finite case with \(N\ge n\).
\end{proof}

\begin{remark}\label{rem:alternating-not-primitive}
Theorem~\ref{thm:first-order-alternating} shows that, in the non-degenerate
case, the alternating theory is not primitive: it is reconstructed from an
ordinary orthogonality problem in the quadratic variable, together with a single
linear annihilation condition. More precisely, once \(\tau\) is fixed, the
functional \(\mathbf u\) is completely determined by
\(\mathbf v=\boldsymbol{\sigma}\mathbf u\) and by
\[
\langle \mathbf u,(x-\tau)r(x^2)\rangle=0
\]
for every \(r\in\mathcal P\). Thus the normalised alternating case is
structurally rigid, but in a manner entirely different from the quadratic and
\(q\)-exponential regimes.
\end{remark}

\begin{example}[Jacobi family for the normalised alternating map]\label{ex:alternating-Jacobi}
Fix \(h\in\mathbb C^\times\), let \(U\subseteq\mathbb C\) be half-step-invariant,
and let
\(
X:U\to\mathbb C
\)
be the admissible map given by
\[
X(s)=e^{\pi i s/h}
\]
for every \(s\in U\). Fix \(\alpha,\beta\in\mathbb C\) such that
\[
-\alpha,\ -\beta,\ -(\alpha+\beta+1)\notin\mathbb N^\times,
\]
and let
\[
\mathbf v_{\alpha,\beta}\in\mathcal P'
\]
be the shifted Jacobi functional associated with the monic shifted Jacobi polynomial sequence
with parameters \((\alpha,\beta)\); see \cite{CG26a}. Let \((R_n)_{n\in\mathbb N}\) denote the corresponding monic orthogonal polynomial sequence. Thus
\[
R_n(t)
=
(-1)^n\frac{(\alpha+1)_n}{(n+\alpha+\beta+1)_n}\,
{}_2F_1\!\left(
\begin{matrix}
-n,\ n+\alpha+\beta+1\\[7pt]
\alpha+1
\end{matrix}
\,;\, t
\right).
\]
We now choose
\(
\tau=1,
\)
and define
\[
\mathbf u_{\alpha,\beta}=\mathbf J_1\mathbf v_{\alpha,\beta}.
\]
Since
\[
R_n(1)=\frac{(\beta+1)_n}{(n+\alpha+\beta+1)_n},
\]
the assumptions on \(\alpha\) and \(\beta\) imply that
\[
R_n(1)\neq0
\]
for every \(n\in\mathbb N\). Hence Theorem~\ref{thm:first-order-alternating} applies.
Moreover, by Proposition~\ref{prop:alternating-construction},
\[
\langle \mathbf u_{\alpha,\beta},(x-1)P(x^2)\rangle=0
\]
for every \(P\in\mathcal P\), and
\[
\boldsymbol{\sigma}\mathbf u_{\alpha,\beta}
=
\mathbf v_{\alpha,\beta}.
\]
Now
\[
(x-\tau^2)\mathbf v_{\alpha,\beta}=(x-1)\mathbf v_{\alpha,\beta},
\]
and, up to a non-zero scalar factor, this is precisely the shifted Jacobi
functional with parameters \((\alpha,\beta+1)\). Therefore its monic orthogonal polynomial sequence \((S_n)_{n\in\mathbb N}\) is given by
\[
S_n(t)
=
(-1)^n\frac{(\alpha+1)_n}{(n+\alpha+\beta+2)_n}\,
{}_2F_1\!\left(
\begin{matrix}
-n,\ n+\alpha+\beta+2\\[7pt]
\alpha+1
\end{matrix}
\,;\, t
\right).
\]
Theorem~\ref{thm:first-order-alternating} now yields the monic orthogonal polynomial sequence
\((P_n)_{n\in\mathbb N}\) of \(\mathbf u_{\alpha,\beta}\). For
\[
\varepsilon_n=\frac{1-(-1)^n}{2},
\quad
m_n=\left\lfloor\frac n2\right\rfloor,
\]
one has
\[
\begin{aligned}
P_n(x)
={}&
(-1)^{m_n}(x-1)^{\varepsilon_n}
\frac{(\alpha+1)_{m_n}}{(m_n+\alpha+\beta+1+\varepsilon_n)_{m_n}}
\\[7pt]
&{}\times{}_2F_1\!\left(
\begin{matrix}
-m_n,\ m_n+\alpha+\beta+1+\varepsilon_n\\[7pt]
\alpha+1
\end{matrix}
\,;\, x^2
\right).
\end{aligned}
\]
\end{example}

\begin{example}[Hermite family for the normalised alternating map]\label{ex:alternating-Hermite}
Retain the normalised alternating map \(X\) introduced in
Example~\ref{ex:alternating-Jacobi}. Let
\(
\mathbf v\in\mathcal P'
\)
be the Laguerre functional with parameter \(-\tfrac12\), and let \((R_n)_{n\in\mathbb N}\)
denote its monic orthogonal polynomial sequence. Thus
\[
R_n(t)
=
(-1)^n\left(\frac12\right)_n
{}_1F_1\!\left(
\begin{matrix}
-n\\[7pt]
\tfrac12
\end{matrix}
\,;\, t
\right).
\]
We now choose
\(
\tau=0,
\)
and define
\[
\mathbf u=\mathbf J_0\mathbf v.
\]
Since
\[
R_n(0)=(-1)^n\left(\frac12\right)_n,
\]
one has
\[
R_n(0)\neq0
\]
for every \(n\in\mathbb N\). Hence Theorem~\ref{thm:first-order-alternating} applies.
Moreover, by Proposition~\ref{prop:alternating-construction},
\[
\langle \mathbf u,x\,P(x^2)\rangle=0
\]
for every \(P\in\mathcal P\), and
\[
\boldsymbol{\sigma}\mathbf u=\mathbf v.
\]
Now
\[
(x-\tau^2)\mathbf v=x\,\mathbf v,
\]
and, up to a non-zero scalar factor, this is precisely the Laguerre functional with parameter
\(\tfrac12\). Therefore its monic orthogonal polynomial sequence \((S_n)_{n\in\mathbb N}\) is given by
\[
S_n(t)
=
(-1)^n\left(\frac32\right)_n
{}_1F_1\!\left(
\begin{matrix}
-n\\[7pt]
\frac32
\end{matrix}
\,;\, t
\right).
\]
Theorem~\ref{thm:first-order-alternating} now yields the monic orthogonal polynomial sequence
\((P_n)_{n\in\mathbb N}\) of \(\mathbf u\). For
\[
\varepsilon_n=\frac{1-(-1)^n}{2},
\quad
m_n=\left\lfloor\frac n2\right\rfloor,
\]
one has
\[
P_n(x)
=
(-1)^{m_n}x^{\varepsilon_n}
\left(\frac12+\varepsilon_n\right)_{m_n}
\,{}_1F_1\!\left(
\begin{matrix}
-m_n\\[7pt]
\frac12+\varepsilon_n
\end{matrix}
\,;\, x^2
\right).
\]
More precisely, if \(H_n\) denotes the classical Hermite polynomial
normalised by
\[
H_n(x)=2^n x^n+\cdots,
\]
then
\[
P_n=2^{-n}H_n.
\]
In particular, the same annihilation condition
\[
\langle \mathbf u,x\,P(x^2)\rangle=0
\]
has two interpretations through
Proposition~\ref{prop:alternating-structural-equation}. Taking \(\phi=0\) and \(\psi=x\), one obtains the non-degenerate structural equation
\[
\mathbf S(x\,\mathbf u)=0.
\]
Taking instead \(\phi=x\), \(\psi=0\), one obtains the degenerate structural equation
\[
\mathbf D(x\,\mathbf u)=0.
\]
Thus the Hermite family occupies a distinguished position in the alternating
theory: it belongs simultaneously to the non-degenerate case \(\psi(x)=x\) and
to the degenerate case \(\psi=0\). In this sense, it sits exactly at the point
where the two descriptions meet.
\end{example}

The preceding theorem suggests a different reading of the so-called
\(-1\) families. In \cite{GVZ13}, these families are described through
Dunkl-type operators involving reflections and as \(q\to-1\) limits of
certain \(q\)-polynomials. Such descriptions are useful, and in many concrete
situations indispensable, but they do not identify the structural mechanism
isolated here. In the present framework, this mechanism is the normalised
alternating map together with the quadratic-substitution decomposition of
\(\mathcal P\). The \(q\to-1\) limit should therefore be viewed as one
realisation of the mechanism, not as its source.

We now illustrate this point with representative examples. It is convenient to isolate the reconstruction mechanism obtained above in a
form that may be applied directly to concrete families.

\begin{corollary}\label{cor:alternating-pullback}
Fix \(\tau\in\mathbb C\) and \(N\in\mathbb N\), and let \(\mathbf v\in\mathcal P'\) be regular, with monic orthogonal polynomial sequence
\(
R_0,R_1,\dots,R_{N+1}.
\)
Assume that
\[
R_n(\tau^2)\neq0
\]
for
\(
n=0,1,\dots,N+1.
\)
Then there exists a unique functional \(\mathbf u\in\mathcal P'\) such that
\[
\boldsymbol{\sigma}\mathbf u=\mathbf v,
\quad
\langle \mathbf u,(x-\tau)P(x^2)\rangle=0
\]
for every \(P\in\mathcal P\), and \(\mathbf u\) is regular of order \(2N+3\). More precisely, if
\[
S_n(x)=
\frac{
R_{n+1}(x)-\dfrac{R_{n+1}(\tau^2)}{R_n(\tau^2)}\,R_n(x)
}{
x-\tau^2
},
\]
for
\(
n=0,1,\dots,N,
\)
then
\(
S_0,S_1,\dots,S_N
\)
is the monic orthogonal polynomial sequence associated with
\[
(x-\tau^2)\mathbf v,
\]
and the monic orthogonal polynomial sequence of \(\mathbf u\) is given by
\[
P_{2n}(x)=R_n(x^2),
\quad
P_{2n+1}(x)=(x-\tau)S_n(x^2),
\]
for
\(
n=0,1,\dots,N,
\)
together with
\[
P_{2N+2}(x)=R_{N+1}(x^2).
\]
\end{corollary}
\begin{proof}
Set
\[
\mathbf u=\mathbf J_\tau\mathbf v.
\]
By Proposition~\ref{prop:alternating-construction}, this is the unique
functional satisfying
\[
\boldsymbol{\sigma}\mathbf u=\mathbf v,
\quad
\langle \mathbf u,(x-\tau)P(x^2)\rangle=0
\]
for every \(P\in\mathcal P\). Since \(\mathbf v\) is regular, one has
\(
\langle \mathbf v,1\rangle\neq0.
\)
Hence
\[
\langle \mathbf u,1\rangle
=
\langle \mathbf u,\sigma(1)\rangle
=
\langle \boldsymbol{\sigma}\mathbf u,1\rangle
=
\langle \mathbf v,1\rangle
\neq0.
\]
Therefore \(\mathbf u\) satisfies condition \emph{(ii)} of
Theorem~\ref{thm:first-order-alternating}, with associated functional
\[
\mathbf v=\boldsymbol{\sigma}\mathbf u.
\]
Now apply Theorem~\ref{thm:first-order-alternating} with
\(
I=\{0,1,\dots,2N+2\}.
\)
Then
\[
J_I=\{0,1,\dots,N+1\},
\quad
K_I=\{0,1,\dots,N\},
\]
so the non-vanishing assumption
\[
R_n(\tau^2)\neq0
\]
for
\(
n=0,1,\dots,N+1
\)
is precisely condition \emph{(iv)} in that theorem. It follows that there exists a monic orthogonal polynomial sequence
\(
P_0,P_1,\dots,P_{2N+2}
\)
with respect to \(\mathbf u\). In particular, \(\mathbf u\) is regular of order \(2N+3\). The same theorem also shows that
\(
S_0,S_1,\dots,S_N
\)
is the monic orthogonal polynomial sequence of
\[
(x-\tau^2)\mathbf v,
\]
and that the monic orthogonal polynomial sequence of \(\mathbf u\) is given by
\[
P_{2n}(x)=R_n(x^2),
\quad
P_{2n+1}(x)=(x-\tau)S_n(x^2),
\]
for
\(
n=0,1,\dots,N,
\)
together with
\[
P_{2N+2}(x)=R_{N+1}(x^2).
\]
This completes the proof.
\end{proof}

\begin{example}[Complementary Bannai--Ito polynomials]
\label{ex:alternating-CBI}
Fix parameters \(\alpha,\beta,\gamma,\delta\in\mathbb C\), and set
\[
g=\alpha+\beta-\gamma-\delta.
\]
Assume that \(\mathbf v\in\mathcal P'\) is a regular functional whose monic orthogonal polynomial sequence \((R_n)_{n\in\mathbb N}\) satisfies
\(R_{-1}=0\), \(R_0=1\), and the recurrence
\[
R_1(t)+\bigl(\beta^2-a_0+c_0\bigr)R_0(t)=t\,R_0(t),
\]
together with, for \(n\in\mathbb N^\times\), the relation
\begin{equation}\label{eq:CBI-R-recurrence}
R_{n+1}(t)+\bigl(\beta^2-a_n+c_n\bigr)R_n(t)
-a_{n-1}c_n\,R_{n-1}(t)=t\,R_n(t),
\end{equation}
where
\begin{align*}
a_n&=
\frac{(n+g+1)(n+\alpha+\beta+1)
\bigl(n+\beta-\gamma+\frac12\bigr)
\bigl(n+\beta-\delta+\frac12\bigr)}
{(2n+g+1)(2n+g+2)},
\\[7pt]
c_n&=
-\frac{n(n-\gamma-\delta)
\bigl(n+\alpha-\gamma+\frac12\bigr)
\bigl(n+\alpha-\delta+\frac12\bigr)}
{(2n+g)(2n+g+1)}.
\end{align*}
Assume that all denominators occurring in the definitions of \(a_n\) and
\(c_n\) are non-zero for the indices under consideration. In the infinite case,
assume moreover that
\[
R_n(\beta^2)\neq0
\]
for every
\(
n\in\mathbb N.
\)
In a finite truncation, the same assumption is required only for the indices
which occur in the construction. We now write out the infinite case. Set
\[
\mathbf u=\mathbf J_\beta\mathbf v.
\]
By Proposition~\ref{prop:alternating-construction}, this is the unique
functional satisfying
\[
\boldsymbol{\sigma}\mathbf u=\mathbf v,
\quad
\langle \mathbf u,(x-\beta)P(x^2)\rangle=0
\]
for every \(P\in\mathcal P\). We now apply
Theorem~\ref{thm:first-order-alternating} with
\[
\tau=\beta.
\]
In the infinite case, since \(\mathbf v\) is regular and
\[
R_n(\beta^2)\neq0
\]
for every \(n\in\mathbb N\), the functional \(\mathbf u\) is regular and its
monic orthogonal polynomial sequence
\(
(P_n)_{n\in\mathbb N}
\)
is given by
\[
P_{2n}(x)=R_n(x^2),
\quad
P_{2n+1}(x)=(x-\beta)S_n(x^2),
\]
where
\[
S_n(t)=
\frac{
R_{n+1}(t)-\dfrac{R_{n+1}(\beta^2)}{R_n(\beta^2)}\,R_n(t)
}{
t-\beta^2
},
\]
for every \(n\in\mathbb N\).

We first compute the quotient
\[
\frac{R_{n+1}(\beta^2)}{R_n(\beta^2)}.
\]
Since \(c_0=0\), the initial recurrence gives
\[
R_1(t)+(\beta^2-a_0)R_0(t)=tR_0(t),
\]
and therefore
\[
R_1(t)=t-\beta^2+a_0.
\]
Evaluating at \(t=\beta^2\), we obtain
\[
R_1(\beta^2)=a_0=a_0R_0(\beta^2).
\]

Assume inductively that
\[
R_n(\beta^2)=a_{n-1}R_{n-1}(\beta^2).
\]
Evaluating \eqref{eq:CBI-R-recurrence} at \(t=\beta^2\), we obtain
\[
R_{n+1}(\beta^2)
=
(a_n-c_n)R_n(\beta^2)+a_{n-1}c_nR_{n-1}(\beta^2).
\]
Using the induction hypothesis, this becomes
\[
R_{n+1}(\beta^2)
=
(a_n-c_n)R_n(\beta^2)+c_nR_n(\beta^2)
=
a_nR_n(\beta^2).
\]
Hence
\[
\frac{R_{n+1}(\beta^2)}{R_n(\beta^2)}=a_n
\]
for every \(n\in\mathbb N\). Consequently,
\[
S_n(t)=\frac{R_{n+1}(t)-a_nR_n(t)}{t-\beta^2}
\]
for every \(n\in\mathbb N\). Since
\[
R_{n+1}(\beta^2)-a_nR_n(\beta^2)=0,
\]
the numerator is divisible by \(t-\beta^2\), and therefore \(S_n\) is indeed a polynomial. Substituting this expression into the definition of \(P_{2n+1}\), we obtain
\begin{align*}
P_{2n+1}(x)
=
(x-\beta)\,
\frac{R_{n+1}(x^2)-a_nR_n(x^2)}{x^2-\beta^2}=
\frac{R_{n+1}(x^2)-a_nR_n(x^2)}{x+\beta}.
\end{align*}
This identity is legitimate because the numerator vanishes at \(x=-\beta\), being equal there to
\[
R_{n+1}(\beta^2)-a_nR_n(\beta^2)=0.
\]
In particular, \(P_0(x)=1\), \(P_1(x)=x-\beta\). We claim that
\begin{equation}\label{eq:CBI-main-recurrence}
P_{n+1}(x)+(-1)^n\beta\,P_n(x)+\tau_nP_{n-1}(x)=xP_n(x),
\end{equation}
where
\[
\tau_{2n}=c_n,
\quad
\tau_{2n+1}=-a_n.
\]
Indeed, from the preceding identity one has
\[
R_{n+1}(x^2)=(x+\beta)P_{2n+1}(x)+a_nP_{2n}(x).
\]
For \(n=0\), the recurrence gives directly
\[
xP_0(x)=P_1(x)+\beta P_0(x).
\]
Now let \(n\in\mathbb N^\times\). Replacing \(n\) by \(n-1\), one finds
\[
P_{2n-1}(x)=\frac{R_n(x^2)-a_{n-1}R_{n-1}(x^2)}{x+\beta},
\]
and hence
\[
a_{n-1}R_{n-1}(x^2)=R_n(x^2)-(x+\beta)P_{2n-1}(x)
=P_{2n}(x)-(x+\beta)P_{2n-1}(x).
\]
Substituting these identities into \eqref{eq:CBI-R-recurrence} with \(t=x^2\), we get
\begin{align*}
x^2P_{2n}(x)
&=
R_{n+1}(x^2)+(\beta^2-a_n+c_n)R_n(x^2)-a_{n-1}c_nR_{n-1}(x^2)
\\[7pt]
&=
\bigl((x+\beta)P_{2n+1}(x)+a_nP_{2n}(x)\bigr)
+(\beta^2-a_n+c_n)P_{2n}(x)
\\[7pt]
&\quad
-c_n\bigl(P_{2n}(x)-(x+\beta)P_{2n-1}(x)\bigr).
\end{align*}
The terms involving \(a_n\) and \(c_n\) cancel, and one obtains
\[
x^2P_{2n}(x)
=
(x+\beta)P_{2n+1}(x)+\beta^2P_{2n}(x)+c_n(x+\beta)P_{2n-1}(x).
\]
Therefore
\[
(x+\beta)(x-\beta)P_{2n}(x)
=
(x+\beta)\bigl(P_{2n+1}(x)+c_nP_{2n-1}(x)\bigr).
\]
Since both sides are polynomials and \(x+\beta\) is a common factor, it follows that
\[
(x-\beta)P_{2n}(x)=P_{2n+1}(x)+c_nP_{2n-1}(x),
\]
that is,
\[
xP_{2n}(x)=P_{2n+1}(x)+\beta P_{2n}(x)+c_nP_{2n-1}(x).
\]
This is exactly \eqref{eq:CBI-main-recurrence} for even index, since
\[
\tau_{2n}=c_n.
\]

For odd index, multiplying the formula for \(P_{2n+1}\) by \(x+\beta\) yields
\[
(x+\beta)P_{2n+1}(x)=R_{n+1}(x^2)-a_nR_n(x^2)=P_{2n+2}(x)-a_nP_{2n}(x).
\]
Thus
\[
xP_{2n+1}(x)=P_{2n+2}(x)-\beta P_{2n+1}(x)-a_nP_{2n}(x),
\]
that is,
\[
P_{2n+2}(x)-\beta P_{2n+1}(x)-a_nP_{2n}(x)=xP_{2n+1}(x).
\]
This is exactly \eqref{eq:CBI-main-recurrence} for odd index, since
\[
\tau_{2n+1}=-a_n.
\]
Therefore the monic polynomial sequence \((P_n)_{n\in\mathbb N}\) satisfies the recurrence \eqref{eq:CBI-main-recurrence}, with
\[
\tau_{2n}=c_n,
\quad
\tau_{2n+1}=-a_n.
\]
By the definitions of \(a_n\) and \(c_n\), this means explicitly
\begin{align*}
\tau_{2n}&=
-\frac{n(n-\gamma-\delta)
\bigl(n+\alpha-\gamma+\tfrac12\bigr)
\bigl(n+\alpha-\delta+\tfrac12\bigr)}
{(2n+g)(2n+g+1)},
\\[7pt]
\tau_{2n+1}&=
-\frac{(n+g+1)(n+\alpha+\beta+1)
\bigl(n+\beta-\gamma+\tfrac12\bigr)
\bigl(n+\beta-\delta+\tfrac12\bigr)}
{(2n+g+1)(2n+g+2)}.
\end{align*}
Together with \(P_0(x)=1\) and \(P_1(x)=x-\beta\), this is precisely the standard recurrence relation for the monic complementary Bannai--Ito polynomials. More explicitly, under the parameter identification \(\rho_1=\alpha\), \(\rho_2=\beta\), \(r_1=\gamma\), and \(r_2=\delta\), the recurrence obtained above is exactly the monic complementary Bannai--Ito recurrence displayed in \cite[(3.4), (3.5)]{GVZ13}.
\end{example}

\begin{remark}
In \cite{GVZ13}, the complementary Bannai--Ito polynomials are introduced
through their recurrence and orthogonality relations, and their even--odd
splitting is written out explicitly. They are then shown to satisfy a
Dunkl-type bispectral problem. The point of the preceding construction is not
merely that the complementary Bannai--Ito recurrence is recovered from the
alternating pullback. Rather, it shows that the even--odd splitting is the
expected functional manifestation of the alternating mechanism.
\end{remark}

\begin{example}[Dual \((-1)\)-Hahn polynomials: the even-\(N\) case]
\label{ex:alternating-dual-minus-one-Hahn}
We now illustrate the alternating mechanism with the dual \((-1)\)-Hahn
polynomials, restricting ourselves to the case where \(N\) is even. This is
already sufficient to make the structural point completely transparent. Let
\(N=2M\) be a positive even integer, fix \(\alpha,\beta\in\mathbb C\), and let
\(
R_0,R_1,\dots,R_N
\)
denote the monic dual \((-1)\)-Hahn polynomial sequence. By
\cite[(3.6), (3.7)]{TVZ13}, one has
\(R_{-1}=0\), \(R_0=1\), and, for
\(
n\in\{0,1,\dots,N-1\},
\)
one has
\[
R_{n+1}(t)+b_nR_n(t)+u_nR_{n-1}(t)
=
t\,R_n(t),
\]
where
\begin{align*}
u_n&=
\begin{cases}
4n(\alpha-n), & \text{if } n \text{ is even},\\[7pt]
4(N-n+1)(n+\beta-N-1), & \text{if } n \text{ is odd},
\end{cases}\\[7pt]
b_n&=
\begin{cases}
2N+1-\alpha-\beta, & \text{if } n \text{ is even},\\[7pt]
-2N-3+\alpha+\beta, & \text{if } n \text{ is odd}.
\end{cases}
\end{align*}
We shall also use the same recurrence at \(n=N\) to define the auxiliary
monic polynomial \(R_{N+1}\). Only \(R_0,R_1,\dots,R_N\) belong to the finite
orthogonal family; the extra polynomial merely records the terminal
coefficient \(u_N\), and no non-zero norm is asserted for it.
Assume that the displayed coefficients are defined and that
\[
u_n\neq0,
\quad n=1,2,\dots,N.
\]
The finite Favard theorem then gives a functional
\(\widehat{\mathbf u}\in\mathcal P'\), regular of order \(N+1\), with respect to
which the shifted sequence introduced below is orthogonal.
Set
\[
\tau=2N+2-\alpha-\beta.
\]
Then
\[
b_n=(-1)^n\tau-1.
\]
Define the shifted monic polynomials, including the auxiliary terminal one, by
\[
\widehat R_n(t)=R_n(t-1).
\]
It follows immediately that \(\widehat R_{-1}=0\), \(\widehat R_0=1\), and
\begin{equation}\label{eq:d-1H-even-shifted-alternating}
\widehat R_{n+1}(t)+(-1)^n\tau\,\widehat R_n(t)+u_n\widehat R_{n-1}(t)
=
t\,\widehat R_n(t).
\end{equation}
for \(n=0,1,\dots,N\). Here again only
\(\widehat R_0,\dots,\widehat R_N\) form the regular finite family.
Indeed, substituting \(t-1\) for \(t\) in the recurrence for \(R_n\), one obtains
\[
\widehat R_{n+1}(t)+\bigl(b_n+1\bigr)\widehat R_n(t)+u_n\widehat R_{n-1}(t)
=
t\,\widehat R_n(t),
\]
and the identity \(b_n+1=(-1)^n\tau\) yields
\eqref{eq:d-1H-even-shifted-alternating}. Thus, after the elementary shift \(t\mapsto t-1\), the even-\(N\) dual
\((-1)\)-Hahn family falls exactly into the alternating recurrence pattern
considered above. This is precisely the recurrence-level shadow of the normalised alternating
pullback. At the recurrence level, the shifted family
\(
\widehat R_0,\widehat R_1,\dots,\widehat R_N
\)
has exactly the alternating form of the full monic sequence in the present
framework. Its even and odd subsequences therefore play the roles, respectively,
of the quadratic base family and of its Christoffel transform. At this stage we
assert only what the finite recurrence proves; the corresponding functional
identity will be recovered below through the degrees actually present. Thus the
even members of the full family must be polynomials in \(t^2\), while the odd
members must contain the distinguished factor \(t-\tau\).
Consequently, the recurrence-level alternating splitting applies here and
yields monic polynomial sequences
\(
P_0,P_1,\dots,P_M
\)
and
\(
Q_0,Q_1,\dots,Q_{M-1}.
\)
More precisely, for \(n=0,1,\dots,M\), the even part is described by
\[
\widehat R_{2n}(t)=P_n(t^2),
\]
whereas, for \(n=0,1,\dots,M-1\), the odd part is described by
\[
\widehat R_{2n+1}(t)=(t-\tau)Q_n(t^2).
\]
Since the polynomials \(Q_0,Q_1,\dots,Q_{M-1}\) are monic of consecutive
degrees, the odd members form a basis of
\((t-\tau)\sigma(\mathcal P_{M-1})\). Their orthogonality to
\(\widehat R_0=1\) therefore gives
\[
\langle\widehat{\mathbf u},(t-\tau)p(t^2)\rangle=0,
\quad p\in\mathcal P_{M-1}.
\]
Put
\[
\mathbf v=\boldsymbol{\sigma}\widehat{\mathbf u}.
\]
Every polynomial \(p\in\mathcal P_N\) has a unique decomposition
\[
p(t)=A(t^2)+(t-\tau)B(t^2),
\quad A\in\mathcal P_M,
\quad B\in\mathcal P_{M-1}.
\]
Consequently,
\[
\langle\widehat{\mathbf u},p\rangle
=
\langle\mathbf v,A\rangle
=
\langle\mathbf J_\tau\mathbf v,p\rangle,
\quad p\in\mathcal P_N.
\]
Thus the functional mechanism is obtained exactly on the finite polynomial
space controlled by the dual \((-1)\)-Hahn family; no infinite identity has
been inferred from a finite recurrence.
The corresponding quadratic recurrence relations are then the following. For
\(n=0\), the last term is absent and one has
\[
P_1(x)
+\bigl(u_0+u_1+\tau^2\bigr)P_0(x)
=
x\,P_0(x).
\]
For
\(
n\in\{1,\dots,M-1\},
\)
one has
\begin{align*}
P_{n+1}(x)
+\bigl(u_{2n}+u_{2n+1}+\tau^2\bigr)P_n(x)
+u_{2n}u_{2n-1}P_{n-1}(x)
=
x\,P_n(x).
\end{align*}
Similarly, one has
\[
Q_{n+1}(x)
+\bigl(u_{2n+2}+u_{2n+1}+\tau^2\bigr)Q_n(x)
+u_{2n}u_{2n+1}Q_{n-1}(x)
=
x\,Q_n(x)
\]
for
\(
n=0,1,\dots,M-2
\),
with the convention \(Q_{-1}=0\). Finally, for
\(
n=0,1,\dots,M-1
\), the Christoffel relation is
\begin{equation}\label{eq:d-1H-even-Christoffel}
Q_n(x)
=
\frac{P_{n+1}(x)+u_{2n+1}P_n(x)}{x-\tau^2}.
\end{equation}
Now
\[
u_{2n}=8n(\alpha-2n),
\quad
u_{2n+1}=4(N-2n)(2n+\beta-N),
\]
so the preceding formulas agree with the quadratic recurrence relations displayed
in \cite[(4.2)--(4.5)]{TVZ13}. Moreover, the explicit hypergeometric formulas in
\cite[(4.6), (4.7)]{TVZ13} identify \(P_n\) and \(Q_n\) as ordinary dual Hahn
polynomials, after the corresponding identification of the quadratic variable.

In the even-\(N\) case, the decomposition exhibited in \cite{TVZ13} is
naturally read as an instance of the alternating splitting mechanism. What
appears there through direct manipulation of the recurrence is seen here as the
expected manifestation of the quadratic-substitution pattern underlying the
alternating case. The odd-\(N\) case leads, after the corresponding modification
of the coefficients, to the same conclusion. We do not repeat it here, since,
for the present purpose, the even-\(N\) case already makes the mechanism
entirely clear.
\end{example}

The examples discussed above also clarify the precise role of the families
considered in \cite{TVZ13} and \cite{GVZ13}. From the viewpoint developed here,
these families should not be regarded as isolated exceptional objects. Nor is
their nature adequately explained by saying that they arise as \(q\to -1\)
limits, or that they are governed by Dunkl-type operators. Such statements may
be computationally correct, but they do not by themselves provide the relevant
structural explanation. The same structural question must be put, mutatis
mutandis, to the
Bannai--Ito polynomials~\cite{TVZ12}, to the big \((-1)\)-Jacobi
polynomials~\cite{VZ12b}, to the \((-1)\)-Meixner--Pollaczek
polynomials~\cite{PVZ24}, and to other families appearing in the literature.

The point is not to deny the validity of the calculations leading to
Dunkl-type operators or to \(q\to -1\) limits. The point is that such
calculations should not be mistaken for the conceptual origin of the
phenomenon. The alternating case is not merely the formal substitution
\(q=-1\) inside a \(q\)-exponential construction, just as the quadratic case is
not merely the value \(q=1\) of that construction. It is a distinct admissible
regime. This returns us to the principle suggested at the outset by the Erdős
anecdote: a correct formal description is not enough; one must identify the
structure that governs the phenomenon.

Perhaps the clearest way to see this is to recall the much-cited article
\emph{A ``missing'' family of classical orthogonal polynomials} \cite{VZ11}.
Nor is the connection merely verbal. If
\[
c_n=\langle\mathbf u,x^n\rangle
\]
denotes the moment sequence displayed in \cite[(2.9)]{VZ11}, then
\(c_{2n}=c_{2n-1}\) for every \(n\in\mathbb N^\times\). Define
\begin{align*}
\langle\mathbf v,p\rangle&=\langle\mathbf u,p(x^2)\rangle,
\\[7pt]
\langle\mathbf w,p\rangle&=\langle\mathbf u,xp(x^2)\rangle.
\end{align*}
The moment identity is exactly \(\mathbf w=x\mathbf v\), and hence
\[
\langle\mathbf u,x(1-x)p(x^2)\rangle=0,
\quad p\in\mathcal P.
\]
By Proposition~\ref{prop:alternating-structural-equation}, this is the concrete
alternating structural equation obtained with
\(\phi(x)=x(1-x)\) and \(\psi=0\). The title is revealing. In the usual classifications, the structural
possibilities explicitly encoded are essentially the quadratic and
\(q\)-exponential ones. From that standpoint the family described in
\cite{VZ11} may appear to be ``missing''. From the viewpoint developed here,
however, it is not missing at all; the theory itself tells us that it was being
sought in a place where it could not reside. It is a concrete manifestation of
the alternating mechanism, and could not be expected to occupy a place in a
classification which does not admit that mechanism. As the preceding sections
show, the case ``\(q=-1\)'' belongs to a different admissible reality: the
alternating regime.

\section{The Nikiforov--Uvarov equation}\label{NUSec}

The next theorem, which is independent of any regularity assumption, makes clear the direct connection between the developments obtained so far and the Nikiforov--Uvarov framework.

\begin{theorem}\label{thm:NU-formal-symmetry}
Fix \(h\in\mathbb C^\times\), let \(U\subseteq\mathbb C\) be half-step-invariant, and let
\(
X:U\to\mathbb C
\)
be an admissible map on \(U\). Let
\(
D,S:\mathcal P\to\mathcal P
\)
be the associated divided-difference and averaging operators, with transposes
\(
\mathbf D,\mathbf S:\mathcal P'\to\mathcal P'.
\)
Let \(\mathbf u\in\mathcal P'\), let \(\phi\) and \(\psi\) be polynomials, and define
\[
L=\phi\,D^2+\psi\,SD.
\]
Then the following statements hold.

\begin{itemize}
\item[\emph{(i)}]
If
\(
\mathbf D(\phi\,\mathbf u)=\mathbf S(\psi\,\mathbf u)
\)
in \(\mathcal P'\), then
\[
\langle \mathbf u,(Lp)\,r\rangle
=
\langle \mathbf u,p\,(Lr)\rangle
\]
for every \(p,r\in\mathcal P\).

\item[\emph{(ii)}]
Assume, in addition, that
\(
D(\mathcal P_{n+1})=\mathcal P_n
\)
for every \(n\in\mathbb N\). If
\(
\langle \mathbf u,(Lp)\,r\rangle
=
\langle \mathbf u,p\,(Lr)\rangle
\)
for every \(p,r\in\mathcal P\), then
\[
\mathbf D(\phi\,\mathbf u)=\mathbf S(\psi\,\mathbf u)
\]
in \(\mathcal P'\).

\item[\emph{(iii)}]
Let \(N\in\mathbb N^\times\). Assume that
\(
D(\mathcal P_{n+1})=\mathcal P_n
\)
for \(
n=0,1,\dots,2N-1.
\)
If
\[
\langle \mathbf u,(Lp)\,r\rangle
=
\langle \mathbf u,p\,(Lr)\rangle
\]
for every \(p\in\mathcal P\) and every \(r\in\mathcal P_{2N}\), then
\[
\bigl\langle \mathbf D(\phi\,\mathbf u)-\mathbf S(\psi\,\mathbf u),\,v\bigr\rangle=0
\]
for every \(v\in\mathcal P_{2N-1}\).
\end{itemize}

Moreover, for every function \(f:U\to\mathbb C\), set
\[
(\Delta f)(s)=f(s+h)-f(s),
\quad
(\nabla f)(s)=f(s)-f(s-h).
\]
Then, for every \(p\in\mathcal P\),
\begin{align*}
(Lp)(X(s))
&=
\phi(X(s))
\frac{1}{Y(s)-Z(s)}
\,
\Delta\!\left(
\frac{\nabla(p\circ X)}{X-Z_1}
\right)(s)
\\[7pt]
&\quad
+
\frac{\psi(X(s))}{2}
\left(
\frac{\Delta(p\circ X)(s)}{Y_1(s)-X(s)}
+
\frac{\nabla(p\circ X)(s)}{X(s)-Z_1(s)}
\right)
\end{align*}
for every \(s\in U\)\footnote{The denominators which occur in this pointwise expression are
non-zero by admissibility. Indeed, \(Y(s)-Z(s)\) is the admissible denominator
at \(s\), while
\[
Y_1(s)-X(s)
=
X(s+h)-X(s)
=
Y\!\left(s+\frac h2\right)-Z\!\left(s+\frac h2\right)
\]
and
\[
X(s)-Z_1(s)
=
X(s)-X(s-h)
=
Y\!\left(s-\frac h2\right)-Z\!\left(s-\frac h2\right).
\]
Since \(U\) is half-step-invariant, the shifted points \(s+\frac h2\) and
\(s-\frac h2\) still belong to \(U\).}.
\end{theorem}

\begin{proof}
By Definition~\ref{def:divided-difference-operator}, evaluating the defining identity at
\(s+\tfrac h2\) and \(s-\tfrac h2\) gives
\begin{align*}
(Dp)(Y(s))
=
\frac{\Delta(p\circ X)(s)}{Y_1(s)-X(s)},\quad
(Dp)(Z(s))
=
\frac{\nabla(p\circ X)(s)}{X(s)-Z_1(s)}.
\end{align*}
Hence
\begin{align*}
(D^2p)(X(s))
=
\frac{1}{Y(s)-Z(s)}
\,
\Delta\!\left(
\frac{\nabla(p\circ X)}{X-Z_1}
\right)(s).
\end{align*}
Likewise,
\begin{align*}
(SDp)(X(s))
=
\frac12
\left(
\frac{\Delta(p\circ X)(s)}{Y_1(s)-X(s)}
+
\frac{\nabla(p\circ X)(s)}{X(s)-Z_1(s)}
\right).
\end{align*}
This establishes the displayed formula for \(L\). Now let \(r\in\mathcal P\), and define
\[
w=(Dp)(Sr)-(Sp)(Dr).
\]
A direct computation gives
\[
w(X)
=
\frac{p(Y)\,r(Z)-p(Z)\,r(Y)}{Y-Z}.
\]
Evaluating at \(s+\tfrac h2\) and \(s-\tfrac h2\), one obtains
\begin{align*}
w(Y)
=
\frac{p(Y_1)\,r(X)-p(X)\,r(Y_1)}{Y_1-X},\quad
w(Z)
=
\frac{p(X)\,r(Z_1)-p(Z_1)\,r(X)}{X-Z_1}.
\end{align*}
Hence
\begin{align*}
(Dw)(X)
&=
\frac{1}{Y-Z}
\Biggl(
r(X)
\frac{p(Y_1)-p(X)}{Y_1-X}
-
r(X)
\frac{p(X)-p(Z_1)}{X-Z_1}
\Biggr)\\[7pt]
&\quad
-\frac{1}{Y-Z}
\Biggl(
p(X)
\frac{r(Y_1)-r(X)}{Y_1-X}
-
p(X)
\frac{r(X)-r(Z_1)}{X-Z_1}
\Biggr)\\[7pt]
&=
(D^2p)(X)\,r(X)-p(X)(D^2r)(X).
\end{align*}
Therefore
\[
Dw=(D^2p)\,r-p\,(D^2r).
\]
Likewise,
\[
Sw=(SDp)\,r-p\,(SDr).
\]
Therefore
\begin{align*}
(Lp)\,r-p\,(Lr)
&=
\phi\bigl((D^2p)r-p(D^2r)\bigr)
+
\psi\bigl((SDp)r-p(SDr)\bigr)
\\[7pt]
&=
\phi\,Dw+\psi\,Sw.
\end{align*}
Assume first that
\[
\mathbf D(\phi\,\mathbf u)=\mathbf S(\psi\,\mathbf u).
\]
Then, for every \(v\in\mathcal P\),
\begin{align}
0=
\bigl\langle \mathbf D(\phi\,\mathbf u)-\mathbf S(\psi\,\mathbf u),\,v\bigr\rangle=
-\langle \mathbf u,\phi\,Dv+\psi\,Sv\rangle.
\label{eq:NU-tested}
\end{align}
Applying \eqref{eq:NU-tested} with \(v=w\), we obtain
\(
\langle \mathbf u,(Lp)\,r\rangle
=
\langle \mathbf u,p\,(Lr)\rangle.
\)
This proves \emph{(i)}. Conversely, assume that
\[
\langle \mathbf u,(Lp)\,r\rangle
=
\langle \mathbf u,p\,(Lr)\rangle
\]
for every \(p,r\in\mathcal P\), and that
\(
D(\mathcal P_{n+1})=\mathcal P_n
\)
for every \(n\in\mathbb N\). Taking \(p=1\), and using \(L1=0\), we obtain
\(
\langle \mathbf u,Lr\rangle=0.
\)
Thus
\[
\langle \mathbf u,\phi\,D^2r+\psi\,SDr\rangle=0.
\]
Using the definitions of \(\mathbf D\) and \(\mathbf S\), this becomes
\[
\bigl\langle \mathbf S(\psi\,\mathbf u)-\mathbf D(\phi\,\mathbf u),Dr\bigr\rangle=0.
\]
Now let \(p\in\mathcal P_n\). By hypothesis, there exists
\(r\in\mathcal P_{n+1}\) such that
\(
Dr=p.
\)
Hence
\[
\bigl\langle \mathbf S(\psi\,\mathbf u)-\mathbf D(\phi\,\mathbf u),p\bigr\rangle=0
\]
for every \(p\in\mathcal P\). Since the pairing between \(\mathcal P'\) and \(\mathcal P\) separates points, condition \emph{(ii)} follows.

Finally, assume the hypotheses of \emph{(iii)}. Taking again \(p=1\), we obtain
\(
\langle \mathbf u,Lr\rangle=0
\)
for every \(r\in\mathcal P_{2N}\). Hence
\[
\bigl\langle \mathbf S(\psi\,\mathbf u)-\mathbf D(\phi\,\mathbf u),Dr\bigr\rangle=0
\]
for every \(r\in\mathcal P_{2N}\). Now let \(v\in\mathcal P_{2N-1}\). By the hypothesis
\(
D(\mathcal P_{n+1})=\mathcal P_n
\)
for \(n=0,1,\dots,2N-1\), there exists \(r\in\mathcal P_{2N}\) such that
\(
Dr=v.
\)
Therefore
\(
\bigl\langle \mathbf S(\psi\,\mathbf u)-\mathbf D(\phi\,\mathbf u),v\bigr\rangle=0
\)
for every \(v\in\mathcal P_{2N-1}\). Equivalently,
\[
\bigl\langle \mathbf D(\phi\,\mathbf u)-\mathbf S(\psi\,\mathbf u),\,v\bigr\rangle=0
\]
for every \(v\in\mathcal P_{2N-1}\).
\end{proof}

\begin{remark}\label{rem:NU-converse-alternating}
The additional surjectivity hypothesis in parts~\emph{(ii)} and~\emph{(iii)}
is automatic in both quadratic half-step branches, but not in general in the
\(q\)-exponential one, and it fails in the alternating case. Indeed, in the
ordinary quadratic branch one has
\[
D(x^{n+1})-(n+1)x^n\in\mathcal P_{n-1},
\]
with the convention \(\mathcal P_{-1}=\{0\}\),
whereas in the reflected quadratic branch one has
\[
D(x^{n+1})-(n+1)(-1)^n x^n\in\mathcal P_{n-1}.
\]
The leading coefficient is non-zero in either case, and therefore
\[
D(\mathcal P_{n+1})=\mathcal P_n
\]
for every \(n\in\mathbb N\). In the \(q\)-exponential case one has
\[
D(x^{n+1})-\gamma_{n+1}x^n\in\mathcal P_{n-1},
\]
where the numbers \(\gamma_n\) are those introduced in
Theorem~\ref{thm:regularity-admissible-orbit}. Hence
\[
D(\mathcal P_{n+1})=\mathcal P_n
\]
for every \(n\in\mathbb N\) if and only if
\(
\gamma_{n+1}\neq0
\)
for every \(n\in\mathbb N\). In particular, this holds whenever \(q\) is not a root of
unity. If \(q\) is a root of unity, then \(\gamma_m=0\) at certain indices, so
the global surjectivity-by-degree condition fails; a finite theory may
nevertheless persist below the first such index. In
that case, part~\emph{(iii)} of Theorem~\ref{thm:NU-formal-symmetry} may still
be applicable up to the corresponding finite level, whereas part~\emph{(ii)}
need not apply globally. Finally, by Proposition~\ref{prop:alternating-DS}, in the normalised alternating case one
has
\(
D(\mathcal P)=\sigma(\mathcal P).
\)
Hence,
\[
D(\mathcal P_{n+1})\subseteq \sigma(\mathcal P)\cap \mathcal P_n,
\]
whereas \(\mathcal P_n\) contains odd polynomials. It follows that
\[
D(\mathcal P_{n+1})\neq\mathcal P_n
\]
for every \(n\in\mathbb N^\times\). Thus the implication asserted in part~\emph{(ii)} of
Theorem~\ref{thm:NU-formal-symmetry} applies automatically in the quadratic
case, applies in the \(q\)-exponential case precisely under the non-vanishing
condition stated above, and does not apply in the normalised alternating case; in the
torsion \(q\)-exponential situation, however, the finite-level variant given by
part~\emph{(iii)} may still remain available.
\end{remark}

In the alternating case, the converse direction in
Theorem~\ref{thm:NU-formal-symmetry} fails in general. This is not merely an
inference from the failure of surjectivity. Since \(D^2=0\), the choice
\(\phi=1\), \(\psi=0\) gives \(L=0\), which is formally symmetric for every
functional. The corresponding structural equation would be
\(\mathbf D\mathbf u=0\); but, if \(\langle\mathbf u,1\rangle\neq0\), then
\[
\langle\mathbf D\mathbf u,x\rangle=-\langle\mathbf u,1\rangle\neq0.
\]
Thus the converse really does fail. When \(\psi\) is of degree \(1\), however,
formal symmetry still recovers the first-order structural condition appearing
in Definition~\ref{def:classical}. The following corollary makes this precise.

\begin{corollary}\label{cor:NU-alternating-degree-one}
Retain the hypotheses and notation of Theorem~\ref{thm:NU-formal-symmetry}. Assume, in addition, that \(X\) is the normalised alternating map given by
\[
X(s)=e^{\pi i s/h}
\]
for every \(s\in U\), and that \(\psi\) is of degree \(1\).
If
\[
\langle \mathbf u,(Lp)\,r\rangle
=
\langle \mathbf u,p\,(Lr)\rangle
\]
for every \(p,r\in\mathcal P\), then there exist \(c\in\mathbb C^\times\) and
\(\tau\in\mathbb C\) such that
\[
\psi(x)=c(x-\tau),
\]
and
\[
\mathbf S\bigl((x-\tau)\mathbf u\bigr)=0
\]
in \(\mathcal P'\). Equivalently,
\[
\langle \mathbf u,(x-\tau)p(x^2)\rangle=0
\]
for every \(p\in\mathcal P\).
\end{corollary}

\begin{proof}
By hypothesis,
\(
\langle \mathbf u,(Lp)\,r\rangle
=
\langle \mathbf u,p\,(Lr)\rangle
\)
for every \(p,r\in\mathcal P\). Taking \(p=1\), and using \(L1=0\), we obtain
\(
\langle \mathbf u,Lr\rangle=0.
\)
Hence
\[
\bigl\langle \mathbf S(\psi\,\mathbf u)-\mathbf D(\phi\,\mathbf u),Dr\bigr\rangle=0.
\]
Now assume that \(X\) is the normalised alternating map. By
Proposition~\ref{prop:alternating-DS}, one has
\[
D(\mathcal P)=\sigma(\mathcal P).
\]
Therefore
\[
\boldsymbol{\sigma}\bigl(\mathbf S(\psi\,\mathbf u)-\mathbf D(\phi\,\mathbf u)\bigr)=0.
\]
Moreover,
\[
\boldsymbol{\sigma}\mathbf D(\phi\,\mathbf u)=0.
\]
Indeed, for every \(p\in\mathcal P\),
\begin{align*}
\bigl\langle \boldsymbol{\sigma}\mathbf D(\phi\,\mathbf u),p\bigr\rangle
=
\bigl\langle \mathbf D(\phi\,\mathbf u),p(x^2)\bigr\rangle=
-\bigl\langle \phi\,\mathbf u,D(p(x^2))\bigr\rangle
=
0,
\end{align*}
since \(D(p(x^2))=0\). Thus
\[
\boldsymbol{\sigma}\mathbf S(\psi\,\mathbf u)=0.
\]
Since \(\psi\) is of degree \(1\), there exist \(c\in\mathbb C^\times\) and
\(\tau\in\mathbb C\) such that
\[
\psi(x)=c(x-\tau).
\]
It follows that
\[
\boldsymbol{\sigma}\mathbf S\bigl((x-\tau)\mathbf u\bigr)=0.
\]
By definition of \(\boldsymbol{\sigma}\) and \(\mathbf S\), this is equivalent to
\[
\langle \mathbf u,(x-\tau)p(-x^2)\rangle=0.
\]
Replacing \(p(x)\) by \(p(-x)\), we obtain
\[
\langle \mathbf u,(x-\tau)p(x^2)\rangle=0
\]
for every \(p\in\mathcal P\). Since, in the normalised alternating case, one has
\(
S(\mathcal P)=\sigma(\mathcal P),
\)
the last identity is equivalent to
\[
\mathbf S\bigl((x-\tau)\mathbf u\bigr)=0
\]
in \(\mathcal P'\). This proves the conclusion.
\end{proof}

\begin{corollary}\label{cor:NU-eigenvalue-equation}
Retain only the notation \(U,X,D,S,\mathbf D,\mathbf S\) from
Theorem~\ref{thm:NU-formal-symmetry}. Let \(\mathbf u\in\mathcal P'\) be
classical, and fix a pair \((\phi,\psi)\), with \(\deg\phi\le2\) and
\(\deg\psi\le1\), not both identically zero, such that
\[
\mathbf D(\phi\,\mathbf u)=\mathbf S(\psi\,\mathbf u).
\]
Define
\[
L=\phi D^2+\psi SD.
\]
Let
\((P_n)_{n\in I}\) be the monic orthogonal polynomial sequence with respect to
\(\mathbf u\), where either \(I=\mathbb N\) or \(I=\{0,1,\dots,N\}\) for some
\(N\geq 2\). Then, for every \(n\in I\), there exists
\(\lambda_n\in\mathbb C\) such that
\[
LP_n=\lambda_nP_n.
\]
\end{corollary}

\begin{proof}
By Theorem~\ref{thm:NU-formal-symmetry}\emph{(i)}, the operator \(L\)
is formally symmetric with respect to \(\mathbf u\), that is,
\[
\langle \mathbf u,(Lp)\,r\rangle
=
\langle \mathbf u,p\,(Lr)\rangle
\]
for every \(p,r\in\mathcal P\). Since \(\phi\) and \(\psi\) have degrees at most \(2\) and \(1\), respectively, one has
\[
\phi\,D^2(\mathcal P_n)\subseteq\mathcal P_n,
\quad
\psi\,SD(\mathcal P_n)\subseteq\mathcal P_n,
\]
and therefore
\[
L(\mathcal P_n)\subseteq\mathcal P_n
\]
for every \(n\in\mathbb N\). In particular,
\[
LP_n\in\mathcal P_n.
\]
If \(n=0\), then \(P_0=1\) and \(L1=0\). Hence
\(
LP_0=0,
\)
so one may take
\(
\lambda_0=0.
\)
Assume now that \(n\in I^\times\), and let \(r\in\mathcal P_{n-1}\). Then, by formal
symmetry,
\[
\langle \mathbf u,LP_n\,r\rangle
=
\langle \mathbf u,P_n\,(Lr)\rangle.
\]
Since \(Lr\in\mathcal P_{n-1}\) and \(P_n\) is orthogonal to
\(\mathcal P_{n-1}\), it follows that
\[
\langle \mathbf u,LP_n\,r\rangle=0.
\]
Thus \(LP_n\in\mathcal P_n\) is orthogonal to \(\mathcal P_{n-1}\). Since the
bilinear form induced by \(\mathbf u\) is non-degenerate on
\(\mathcal P_{n-1}\), and \(P_n\) is the monic orthogonal polynomial of degree
\(n\), there exists \(\lambda_n\in\mathbb C\) such that
\[
LP_n=\lambda_nP_n,
\]
which is precisely the desired conclusion.
\end{proof}

\begin{remark}\label{rem:NU-eigenvalue-nonvanishing}
Let \(\lambda_n\) be as in Corollary~\ref{cor:NU-eigenvalue-equation}. Then
\[
\lambda_0=0.
\]
Moreover, for each \(n\in I^\times\), \(\lambda_n\) is the leading coefficient of
\(L(x^n)\). Indeed, since \(P_n\) is monic, one has
\[
P_n-x^n\in\mathcal P_{n-1},
\]
and since \(L(\mathcal P_{n-1})\subseteq\mathcal P_{n-1}\), it follows that
\[
LP_n-L(x^n)\in\mathcal P_{n-1}.
\]
Hence the coefficient of \(x^n\) in \(LP_n\) coincides with that in \(L(x^n)\).
Since
\[
LP_n=\lambda_n P_n
\]
and \(P_n\) is monic, this coefficient is precisely \(\lambda_n\). Consequently,
\[
\lambda_n\neq0
\]
if and only if \(L(x^n)\) is of degree \(n\). In the ordinary quadratic
branch, since
\begin{align*}
D(x^n)&=n x^{n-1}+\cdots,
\\[7pt]
S(x^m)&=x^m+\cdots,
\end{align*}
one obtains
\[
L(x^n)
=
n\bigl(\phi_2(n-1)+\psi_1\bigr)x^n+\cdots.
\]
Thus
\[
\lambda_n=n\,d_{n-1}.
\]
In the reflected quadratic branch, by contrast,
\begin{align*}
D^2(x^n)&=-n(n-1)x^{n-2}+\cdots,
\\[7pt]
SD(x^n)&=n x^{n-1}+\cdots,
\end{align*}
and hence
\[
\lambda_n
=
n\bigl(\psi_1-\phi_2(n-1)\bigr).
\]
In the \(q\)-exponential case, since
\begin{align*}
D(x^n)&=\gamma_n x^{n-1}+\cdots,
\\[7pt]
S(x^m)&=\alpha_m x^m+\cdots,
\end{align*}
one obtains
\[
L(x^n)
=
\gamma_n\bigl(\phi_2\gamma_{n-1}+\psi_1\alpha_{n-1}\bigr)x^n+\cdots.
\]
Thus
\[
\lambda_n=\gamma_n d_{n-1}.
\]
Finally, in the normalised alternating case,
\[
D^2=0,
\quad
D(x^{2m})=0,
\quad
SD(x^{2m+1})=x^{2m}.
\]
Therefore, for every admissible index,
\[
\lambda_{2m}=0,
\quad
\lambda_{2m+1}=\psi_1.
\]
The associated Nikiforov--Uvarov operator thus detects the parity splitting
itself, rather than an ordinary degree-by-degree spectrum.
Consequently, for the ordinary quadratic and \(q\)-exponential branches, in
both the infinite and finite non-resonant sectors, the possible vanishing of
\(\lambda_n\) is
governed by the same factors that occur in the regularity criteria established
above. In a finite resonant system, however,
the vanishing of such a leading factor need not destroy regularity: it may
instead mark a compatible moment equation with a degree of freedom, precisely
as in Theorem~\ref{thm:finite-resonant-regularity}.
\end{remark}


As a final point, we record one structural feature of particular significance:
the point at which the distinction between ordinary quadratic and non-torsion
\(q\)-exponential maps, on the one hand, and the normalised alternating map, on the other, becomes
intrinsic and ineliminable, while at the same time marking a decisive departure
from the ordinary theory, where \(D\) is ordinary differentiation and
\(S=\mathrm{id}_{\mathcal P}\). This is the forward half of the
Sonine--Hahn phenomenon. Its ordinary differential form goes back to Sonine's
1887 memoir \cite{Son87}; Hahn returned to the question in \cite{H35}.
Higher derivatives were subsequently considered by Krall \cite{K36} and by
Hahn \cite{H38}. In the ordinary differential setting, the
converse characterisation asks whether an orthogonal polynomial sequence whose
derivative sequence is again orthogonal must belong to the classical families.
That converse is not part of the statement below.

\begin{proposition}[Stability under the first divided difference]\label{thm:normalised-Hahn-property}
Fix \(h\in\mathbb C^\times\), let \(U\subseteq\mathbb C\) be half-step-invariant, and let
\(X:U\to\mathbb C\) be an admissible map.
Let
\(
D,S:\mathcal P\to\mathcal P
\)
be the associated divided-difference and averaging operators, with transposes
\(
\mathbf D,\mathbf S:\mathcal P'\to\mathcal P'.
\)
Let \(I\subseteq\mathbb N\) be either \(I=\mathbb N\), or
\(
I=\{0,1,\dots,N+1\}
\)
for some \(N\geq 2\). Let \(\mathbf u\in\mathcal P'\) have the monic
orthogonal polynomial sequence
\(
(P_n)_{n\in I}.
\)
In the finite case, set
\(
I_1=\{0,1,\dots,N\},
\)
whereas in the infinite case set
\(
I_1=\mathbb N.
\)
Assume that \(X\) belongs to the ordinary quadratic half-step branch or to the
\(q\)-exponential branch with \(q\) not a root of unity, in the sense of
Theorem~\ref{thm:admissible-lattice-classification-U} and
Corollary~\ref{cor:admissible-half-step-branches}. Fix polynomials \(\phi\) and
\(\psi\), where \(\deg\phi\le2\) and \(\deg\psi\le1\), not both identically
zero, such that
\[
\mathbf D(\phi\,\mathbf u)=\mathbf S(\psi\,\mathbf u).
\]
In the finite case, assume that this pair lies in the non-resonant sector of
the corresponding regularity theorem; equivalently,
\[
d_j\neq0,
\quad
j=0,1,\dots,2N+1,
\]
with \(d_j\) defined as in that theorem. Set
\[
U_2=S(x^2)-(Sx)^2,
\quad
\mathbf u^{[1]}
=
\mathbf D(U_2\psi\mathbf u)-\mathbf S(\phi\mathbf u),
\]
and
\[
Q_n^{[1]}=\gamma_{n+1}^{-1}DP_{n+1},
\]
where, in the \(q\)-exponential case, the numbers \(\gamma_n\) are those introduced
in Theorem~\ref{thm:regularity-admissible-orbit}, whereas in the quadratic case
one has \(\gamma_n=n\). Then the first derived functional
\(\mathbf u^{[1]}\in\mathcal P'\)
is classical, and the family
\(
(Q_n^{[1]})_{n\in I_1}
\)
is its monic orthogonal polynomial sequence.
\end{proposition}
\begin{proof}
Write
\[
h_n=\langle\mathbf u,P_n^2\rangle.
\]
In the ordinary quadratic branch,
\eqref{eq:quadratic-derived-norm} gives
\[
\bigl\langle\mathbf u^{[1]},Q_j^{[1]}Q_k^{[1]}\bigr\rangle
=
\frac{d_j}{j+1}h_{j+1}\delta_{j,k}.
\]
In the \(q\)-exponential branch, the corresponding identity is
\eqref{eq:q-global-derived-norm}, namely
\[
\bigl\langle\mathbf u^{[1]},Q_j^{[1]}Q_k^{[1]}\bigr\rangle
=
\alpha\,\frac{d_j}{\gamma_{j+1}}h_{j+1}\delta_{j,k}.
\]
In the infinite case, the necessity parts of the two regularity theorems show
that every \(d_j\) is non-zero. In the finite case this is precisely the
non-resonance hypothesis imposed above. Since \(h_{j+1}\neq0\), and since
\(\alpha\) and \(\gamma_{j+1}\) do not vanish in the non-torsion
\(q\)-exponential branch, all the displayed norms are non-zero. The polynomials
\(Q_j^{[1]}\) are monic; hence they form the monic orthogonal polynomial
sequence of \(\mathbf u^{[1]}\). Thus \(\mathbf u^{[1]}\) is regular in the
infinite case and regular of order \(N+1\ge3\) in the finite case.

It remains to verify, rather than assume, the classicality of
\(\mathbf u^{[1]}\). In the ordinary quadratic branch it follows, at the first level,
from \eqref{eq:quadratic-shifted-functional-equation}; in the
\(q\)-exponential branch it follows from
\eqref{eq:shifted-relation-present-proof-solid2}. In either case the
transformed pair has the required degrees and is not identically zero, since
its affine member has non-zero leading coefficient \(d_2\). Thus
\(\mathbf u^{[1]}\) is classical, as asserted.
\end{proof}

\begin{remark}
In the \(q\)-exponential case, the exclusion of roots of unity ensures that the
normalising factors \(\gamma_{n+1}\) in the first derived family do not vanish.
At a root of unity the same first-step construction may survive on a finite
segment, provided that the required normalising factors are non-zero and the
corresponding non-resonance conditions hold. The proposition is deliberately
first-order; no all-orders conclusion is asserted.
\end{remark}

The reader should not underestimate the difficulty concealed by the reverse
direction. No complete Sonine--Hahn converse is asserted here; indeed, we do
not believe it to be valid at the present level of generality without further
hypotheses. Finite, torsion, reflected, and alternating phenomena must be
examined on their own terms: the ease of a few familiar cases is no licence to
promote the converse to a general theorem.
Proposition~\ref{thm:normalised-Hahn-property} records only the forward
implication actually established by the regularity theory. In the ordinary
differential setting, for a regular functional and under the standard
admissibility hypotheses on the defining pair, the functional equation and
Hahn's property are equivalent characterisations; see
\cite{M91a,M91b}. In the present framework, the analogous role is played by the
Nikiforov--Uvarov equation formulated in the same functional-analytic
language.

\section{Conclusions}

At the first international conference on orthogonal polynomials, held in
Bar-le-Duc in 1984, Andrews and Askey wrote, in a particularly interesting
article \cite[Section~4]{AA85}, under the deliberately resonant title \emph{The final sets of
classical orthogonal polynomials}, the following words: ``This section is
titled in a very strong way, and we hope that someone will come along and
prove we are wrong, just as has happened to everyone else who has tried to
characterize the classical polynomials.'' The notion of classicality proposed
here is not intended to suggest that there was any error in that legacy, nor
in any of the developments that followed it. Rather, it seeks to embrace that
legacy, together with many other results that have remained apparently outside
that framework owing to the technical restrictions imposed by earlier
definitions. Its aim is to gather them and to place them within a setting
sufficiently broad for their internal unity to become visible, and from which
Bochner's result may once again be viewed in continuity with the historical
tradition to which this problem belongs. This viewpoint, moreover, cannot be
separated from half-step-invariant sets or from the form of admissible maps,
for it is precisely through them that one gains access to what, in our view,
constitutes the true rigidity associated with classicality: a rigidity often
neither visible nor necessary at first sight, until one attempts to dig deeply
into a particular problem and suddenly comes up against it with full force.
What is perhaps most striking, and what the preceding pages have sought to lay
fully bare, is that the path may now appear almost disarmingly simple. It had
been there all along; what was required was to follow it. Nor was it ours to
trace unaided. The way had already been opened, from different but convergent
directions, by Nikiforov and Uvarov and by Maroni; we have had the good fortune
to stand on the shoulders of giants whose most structural ideas,
though decisive, have too often remained outside the principal narrative
through which the field has come to understand itself. That narrative has too
often been dominated by ever-expanding catalogues of separately named families
and by scattered manifestations of a few underlying mechanisms, recorded
under different names as though their separation belonged to the mathematics
itself. Yet a profusion of names is not a profusion of structures.

The space in which those lines of thought meet is constructed here. The theory of
half-step-invariant sets and admissible maps, placed within the continuous dual
of \(\mathcal P\), is not an ornamental language laid over familiar formulae.
It is the global setting in which the structural insight of Nikiforov and
Uvarov and the functional insight of Maroni can operate together, and
therefore the setting in which global compatibility, finite resonance,
torsion, and the alternating mechanism become visible as parts of a single
theory. The essential step was to take those ideas at their full strength: to
proceed from the general to the particular, and to resist the recurrent
temptation to reconstruct the structure from its proliferating special
appearances. The novelty of the present paper lies neither in the manufacture
of another family nor in the invention \emph{ex nihilo} of another formalism.
It lies in constructing that setting, in recognising how far the inherited
ideas can be carried within it, and in carrying them far enough for the unity,
rigidity, and unsuspected reach of the classical theory to become visible.

\section*{Acknowledgements}
The author acknowledges financial support from the Centre for Mathematics of the University of Coimbra (CMUC), funded by the Portuguese Foundation for Science and Technology (FCT), under the projects UID/00324/2025 (\url{https://doi.org/10.54499/UID/00324/2025}) and UID/PRR/00324/2025.
The author also acknowledges financial support from the FCT under the grant \url{https://doi.org/10.54499/2022.00143.CEECIND/CP1714/CT0002}.

\bibliographystyle{amsplain}

\bibliography{bib}

\end{document}